\documentclass[11pt, a4paper]{article}
\usepackage{cite}
\usepackage{setspace,caption}
\usepackage{amsmath,amsthm, amssymb,wasysym}
\usepackage{amsfonts}
\usepackage{graphicx}
\usepackage{subfig,tikz}
\allowdisplaybreaks[1]
\DeclareMathOperator*{\argmin}{arg\,min}
\definecolor{dgreen}{rgb}{0,0.6,0}
\definecolor{dbrown}{rgb}{0.45,0.25,0}

\newcommand{\Scal}{\mathcal{S}}
\newcommand{\Tcal}{\mathcal{T}}
\newcommand{\Pcal}{\mathcal{P}}

\newcommand{\Ical}{\mathcal{I}}

\newcommand{\Ibld}{{\mbox{\bf I}}}
\newcommand{\onebld}{{\mbox{\bf 1}}}

\newcommand{\wt}{\widetilde}

\newcommand{\wh}{\widehat}

\newcommand{\ol}{\overline}

\newcommand{\gap}{\vspace{0.1in}}

\newcommand{\ppm}{\wh p}

\newtheorem{theorem}{Theorem}[section] %[section]
\newtheorem{lemma}{Lemma}[section] %[section]
\newtheorem{corollary}{Corollary}[section] %[section]
\newtheorem{proposition}{Proposition}[section] %[section]
\newtheorem{remark}{Remark}[section]%[section]
\begin{document}

\title{Uniform Lipschitz Property of Nonnegative Derivative Constrained B-Splines  and Applications to Shape Constrained Estimation}

\author{Teresa \,  M. \,   Lebair
%
%\footnote{Department of Mathematics and Statistics, University of Maryland Baltimore County, Baltimore, MD %21250, U.S.A. Email: ei44375@umbc.edu.}
%
\ \ \ and \ \ \
Jinglai \, Shen\footnote{Department of Mathematics and Statistics, University of Maryland, Baltimore County, Baltimore, MD 21250, U.S.A. Email:
\{teresa.lebair, shenj\}@umbc.edu. Phone: 410-455-2402; Fax: 410-455-1066. } }
%
%This research was partially supported by the NSF grants CMMI-1030804 and DMS-1042916.} }
%

\date{\today}

\maketitle

\vspace{-0.21in}

\begin{abstract}
Inspired by shape constrained estimation under general nonnegative derivative constraints, this paper considers the B-spline approximation of constrained functions and studies the asymptotic performance of the constrained B-spline estimator.
%
%Nonsmooth inequality shape constraints pose many challenges in performance analysis.
%
By invoking a deep result in B-spline theory (known as de Boor's conjecture) first proved by
A. Shardin as well as other new analytic techniques, we establish a critical uniform Lipschitz property of the B-spline estimator subject to arbitrary nonnegative derivative constraints under the $\ell_\infty$-norm with possibly non-equally spaced design points and knots. This property leads to important asymptotic analysis results of the B-spline estimator, e.g., the uniform convergence and consistency on the entire interval under consideration. The results developed in this paper not only recover the well-studied monotone and convex approximation and estimation as special cases, but also treat general nonnegative derivative constraints in a unified framework and open the door for the constrained B-spline approximation and estimation
%
%minimax risk analysis of
%
%constrained estimators
%
subject to a broader class of shape constraints.
\gap


\end{abstract}

\vspace{-0.18in}

%----------------------introduction -----------------------------
%
\section{Introduction}

Shape constrained estimation is concerned with the nonparametric estimation of an underlying function subject to a pre-specified shape condition, e.g., the monotone or convex condition, such that
a constructed estimator preserves the given shape condition. This problem has garnered substantial attention from approximation and estimation theory, and nonparametric statistics, due to the vast number of applications; see \cite{DeVoreL_book93, EgMartin_book10, Pal_SPL08, Sanyal_CDC03, ShenLebair_Auto15, SWang_ACC10} and the references therein.

A shape constrained estimator is subject to nonsmooth inequality constraints, which yield many challenges in its characterization and performance analysis. Two lines of research have been carried out for shape constrained estimation. The first line pertains to estimator characterization and numerical computation.
Various shape constrained estimators are characterized by constrained smoothing splines, which can be treated as constrained control theoretic splines \cite{EgMartin_book10, SunEgerMartin_TAC00} and formulated as inequality constrained optimal control problems. Nonsmooth analytic or optimization techniques have been used to characterize and compute constrained estimators \cite{CheeWang_CSDA14, EgMartin_book10, Dontchev_siopt02, Dontchev_constapprox03, PappAlizadeh_JCGS14, ShenLebair_Auto15}.
%
%; other related results include constrained optimization approaches
%%
%\cite{CheeWang_CSDA14, EgMartin_book10, PappAlizadeh_JCGS14}.
%
However, all these papers (along the first line) do not take performance analysis into account.
The second research line is concerned with the asymptotic performance analysis of (univariate) shape constrained estimators. The current literature along this line focuses mostly on monotone estimation, e.g., \cite{Cator_B11, Pal_SPL08, PalWoodroofe_SS07, ShenWang_SICON11, WangShen_biometrika10}, and convex estimation, e.g., \cite{BirkeDette_SJoS06, Dumbgen_MMS04, Groeneboom_AoS01, ShenWang_ACC12, WangShen_SICON13}; additional results include $k$-monotone estimation \cite{Balabdaoui_SN10}. These papers study the performance of certain constrained estimators, e.g., the least-squares, B-spline, and P-spline estimators, for sufficiently large sample size
%
%is sufficiently large.
%
%Typical performance issues include  consistency, convergence rates, and minimax risk
%
\cite{Nemirovski_notes00, Tsybakov_book10}.

It is observed that the monotone (resp. convex) constraint on a univariate function roughly corresponds to the first (resp. second) order nonnegative derivative constraint, under suitable smoothness conditions on the underlying function. Despite extensive research on asymptotic analysis of monotone and convex estimation,
very few performance analysis results are available for higher-order nonnegative derivative constraints, although such constraints arise in applications  \cite{Sanyal_CDC03}. Motivated by these applications and the lack of performance analysis of the associated constrained estimation, we consider the estimation of a univariate function subject to the $m$th order nonnegative derivative constraint via B-spline estimators, where $m \in \mathbb N$ is arbitrary. The B-splines are a popular tool in approximation and estimation theory thanks to their numerical advantages \cite{deBoor_book01, DeVoreL_book93}. Nonnegative derivative  constraints on a B-spline estimator can be easily imposed on spline coefficients, which can be efficiently computed  via quadratic programs. In spite of this numerical efficiency and simplicity, the asymptotic analysis of constrained B-spline estimators is far from trivial, particularly when uniform convergence and the supremum-norm (or sup-norm) risk are considered.

The asymptotic analysis of constrained B-spline estimators requires a deep understanding of the mapping from a (weighted) sample data vector to the corresponding B-spline coefficient vector. For a fixed sample size, this mapping is given by a Lipschitz piecewise linear function due to the inequality shape constraints. As the sample size increases and tends to infinity, it gives rise to an infinite family of size-varying piecewise linear functions. A critical {\em uniform Lipschitz property} has been established for monotone P-splines (corresponding to $m=1$) \cite{ShenWang_SICON11}  and convex B-splines  (corresponding to $m=2$) \cite{WangShen_SICON13}. This property states that the size-varying piecewise linear functions attain a uniform Lipschitz constant under the $\ell_\infty$-norm, regardless of sample size and the number of knots, and it leads to many important results in asymptotic analysis, e.g.,
uniform convergence, pointwise mean squared risk, and optimal rates of convergence  \cite{WangShen_SICON13}. It has been conjectured that this property can be extended to B-spline estimators subject to higher-order nonnegative derivative constraints \cite{WangShen_SICON13}. However, the extension  encounters a major difficulty:
the proof of the uniform Lipschitz property for the monotone and convex cases heavily relies on the diagonal dominance of certain matrices that  no longer holds in the higher-order case. In addition, the results in \cite{ShenWang_SICON11, WangShen_SICON13} are based on a restrictive assumption of equally spaced design points and knots,
%
%which is restrictive in practice,
%
but the extension to non-equally spaced case is nontrivial. To overcome these difficulties, we
%
%propose new techniques and
%
develop various new results for the proof of the
uniform Lipschitz property and asymptotic analysis of the B-spline estimators for an arbitrary $m$. We summarize these results and contributions of the paper as follows.

(1) A novel technique for the proof of the uniform Lipschitz property lies in a deep result in B-spline theory (dubbed de Boor's conjecture) first proved by A. Shardin \cite{Shardin_AM01}; see \cite{Golitschek_JoAT14} for a recent simpler proof. Informally speaking, this result says that the $\ell_\infty$-norm of the inverse of the Gramian formed by the normalized B-splines of order $m$ is uniformly bounded, regardless of the knot sequence and the number of B-splines (cf. Theorem~\ref{thm:Shardin_bound} in Section~\ref{sect:proof_roadmap}). Sparked by this result, we construct (nontrivial) coefficient matrices of the piecewise linear functions and approximate these matrices by suitable B-splines via analytic techniques. This yields the uniform bounds in the $\ell_\infty$-norm for arbitrary $m$ and possibly non-equally spaced design points and knots; see Theorem~\ref{thm:uniform_Lip}.

(2) Using the uniform Lipschitz property, we show that for any order $m$, the constrained B-spline estimator achieves uniform convergence and consistency on the entire interval of interest even when design points are not equally spaced. Furthermore, we develop a preliminary convergence rate of the B-spline estimator in the sup-norm; this rate sheds light on the optimal convergence and minimax risk analysis of the B-spline estimators subject to general nonnegative derivative constraints.

%%\gap

% Paper Organization
The paper is organized as follows.
%
%In the remainder of this section, we define some notation to be used throughout the rest of this paper.
%
In Section~\ref{sec:prob_form}, we introduce the constrained B-spline estimator and state the uniform Lipschitz property. Section~\ref{sect:proof_unif_Lip} is devoted to the proof of the uniform Lipschitz property. Section~\ref{sect:constrained_estimation} establishes the uniform bounds of the bias and stochastic error of the constrained B-spline estimator via the uniform Lipschitz property. Concluding remarks are made in Section~\ref{sect:conclusion}.
%

%-------------------------------------------------------
%
%\subsection{Notation}\label{sec:notation}
%

{\it Notation}.
We introduce some notation used in the paper.
 Define the function $\delta_{ij}$ on $\mathbb N \times \mathbb N$ so that $\delta_{ij} = 1$ if $i = j$, and $\delta_{ij} = 0$ otherwise.
 Let $\Ibld_S$ denote the indicator function for a set $S$.
For an index set $\alpha$, let $\ol\alpha$ denote its complement, and $|\alpha|$ denote its cardinality.  In addition, for $k \in \mathbb N$, define the set $\alpha+k:= \{i+k \,:\, i \in \alpha \}$.  Let $\onebld_k \in \mathbb R^k$ denote the column vector of all ones and $\onebld_{k_1\times k_2}$ denote the $k_1\times k_2$ matrix of all ones.
For a column vector $v \in \mathbb R^p$, let $v_i$ denote its $i$th component.  For a matrix $A \in \mathbb R^{k_1 \times k_2}$, let $[A]_{ij}$ or $[A]_{i,j}$ be its $(i, j)$-entry, let $(A)_{i\bullet}$ be its $i$th row, and $(A)_{\bullet j}$ be its $j$th column. If $i_1 \leq i_2$ and $j_1 \leq j_2$, let $(A)_{i_1:i_2,\bullet}$ be the submatrix of $A$ formed by its $i_1$th to $i_2$th rows, let $(A)_{\bullet,j_1:j_2}$ denote the submatrix of $A$ formed by its $j_1$th to $j_2$th columns, and let $(A)_{i_1:i_2,j_1:j_2}$ denote the submatrix of $A$ formed by its $i_1$th to $i_2$th rows and $j_1$th to $j_2$th columns. Given an index set $\alpha$, let $v_{\alpha}\in \mathbb R^{|\alpha|}$ denote the vector formed by the components of $v$ indexed by elements of $\alpha$, and $(A)_{\alpha\bullet}$ denote the matrix formed by the rows of $A$ indexed by elements of $\alpha$.

%----------------------------------------------------------------------------
%
%
\section{Shape Constrained B-splines: Uniform Lipschitz Property} \label{sec:prob_form}

Fix $m \in \mathbb N$. Consider the class of (generalized) shape constrained univariate functions on $[0, 1]$:
\begin{eqnarray}
\Scal_m & := \, \left\{ f:[0,1] \rightarrow \mathbb R \ \Big | \ \mbox{ the $(m-1)$th derivative } f^{(m-1)} \mbox{ exists almost everywhere on $[0, 1]$,  } \right. \notag \\
&  \quad \ \mbox{ and }
 \left. \big ( f^{(m-1)}(x_1) - f^{(m-1)}(x_2) \big) \cdot \big(x_1-x_2 \big) \geq 0 \ \mbox{ when $f^{(m-1)}(x_1), f^{(m-1)}(x_2)$ exist}
\right \}.  \label{eqn:shape_class}
\end{eqnarray}
When $m=1$, $\mathcal S_m$ represents the set of increasing functions on $[0, 1]$. Similarly, when $m=2$, $\mathcal S_m$ denotes the set of continuous convex functions on $[0, 1]$.

This paper focuses on the B-spline approximation of functions in $\mathcal S_m$. Toward this end,
we provide a brief review of B-splines as follows; see \cite{deBoor_book01} for more details. For a given $K \in \mathbb N$, let $T_\kappa:= \{\kappa_0 < \kappa_1 < \cdots < \kappa_K \}$ be a sequence of $(K+1)$ knots in $\mathbb R$.  Given $p \in \mathbb N$, let $\{B^{T_\kappa}_{p,k}\}_{k=1}^{K+p-1}$ denote the $(K+p-1)$ B-splines of order $p$ (or equivalently degree $(p-1)$) with knots at $\kappa_0, \kappa_1,\ldots, \kappa_K$, and the usual extension $\kappa_{1-p} = \cdots = \kappa_{-1} = \kappa_0$ on the left and $\kappa_{K+1} = \cdots = \kappa_{K+p-1} = \kappa_K$ on the right,
scaled such that $\sum_{k=1}^{K+p-1}  B^{T_\kappa}_{p,k}(x) = 1$ for any $x \in [\kappa_0, \kappa_K]$. The support of $B^{T_\kappa}_{p,k}$ is given by (i) $[\kappa_{k-p}, \kappa_k)$ when $p=1$ and $1\le k \le K-1$;  (ii) $[\kappa_{k-p}, \kappa_k]$ when $p=1$ and $k=K$ or for each $k=1, \ldots, K+p-1$ when $p \ge 2$. We summarize some properties of the B-splines to be used in the subsequent development below:
\begin{itemize}
  \item [(i)] Nonnegativity, upper bound, and partition of unity: for each $p$, $k$, and $T_\kappa$, $0 \le B^{T_\kappa}_{p,k}(x) \le 1$ for any $x\in [\kappa_0, \kappa_K]$, and $\sum_{k=1}^{K+p-1}  B^{T_\kappa}_{p,k}(x) = 1$ for any $x \in [\kappa_0, \kappa_K]$.
  \item [(ii)] Continuity and differentiability: when $p=1$, each $B^{T_\kappa}_{p, k}(x)$ is a (discontinuous) piecewise constant function given by %%$B^{T_\kappa}_{p, k}(x)=
      $\mathbf I_{[\kappa_{k-1}, \kappa_k)}(x)$ for $1\le k \le K_n-1$ or  $\mathbf I_{[\kappa_{K_n-1}, \kappa_{K_n}]}(x)$ for $k=K_n$, also $B^{T_\kappa}_{p, k}(x) \cdot B^{T_\kappa}_{p, j}(x)=0, \forall \ x$ if $k\ne j$;
  when   $p=2$, the $B^{T_\kappa}_{p, k}$'s are continuous piecewise linear splines, and there are at most three points in $\mathbb R$ where each  $B^{T_\kappa}_{p,k}$ is not differentiable; when $p>2$, each $B^{T_\kappa}_{p,k}$ is differentiable on $\mathbb R$. For $p\ge 2$, the derivative of $B^{T_\kappa}_{p,k}$ (when it exists) is
\begin{equation} \label{eqn:B_spline_derivative}
  \Big( B^{T_\kappa}_{p,k}(x) \Big)' \ = \ \frac{p-1}{\kappa_{k-1} - \kappa_{k-p} } B^{T_\kappa}_{p-1, k-1}(x) \, - \,  \frac{p-1}{\kappa_{k} - \kappa_{k-p+1} } B^{T_\kappa}_{p-1, k}(x),
\end{equation}
where we define $ \frac{p-1}{\kappa_{k} - \kappa_{k-p+1} } B^{T_\kappa}_{p-1, k}(x):=0, \forall \, x \in [0, 1]$ for $k=0$ and $k=K+p-1$.

  \item [(iii)] $L_1$-norm: for each $k$, the $L_1$-norm of $B^{T_\kappa}_{p,k}$ is known to be \cite[Chapter IX, eqns.(5) and (7)]{deBoor_book01}
\begin{equation} \label{eqn:L1_norm_Bspline}
    \left \| B^{T_\kappa}_{p,k} \right \|_{L_1} \, := \, \int_{\mathbb R} \left |  B^{T_\kappa}_{p,k}(x) \right  | dx = \frac{\kappa_k - \kappa_{k-p} }{p}.
\end{equation}
\end{itemize}

Let $T_\kappa := \{0=\kappa_0 < \kappa_1 < \dots < \kappa_{K_n}=1 \}$ be a given sequence of $(K_n+1)$ knots in $[0, 1]$, and let $g_{b, T_\kappa}:[0, 1]\rightarrow \mathbb R$ be such that $g_{b, T_\kappa}(x) = \sum^{K_n+m-1}_{k=1} b_k B^{T_\kappa}_{m, k}(x)$, where the $b_k$'s are real coefficients of B-splines and $b:=(b_1, \ldots, b_{K_n+m-1})^T$ is the spline coefficient vector. Here the subscript $n$ in $K_n$ corresponds to the number of design points to be used in the subsequent sections.

To derive a necessary and sufficient condition for $g_{b, T_\kappa}\in \mathcal S_m$, we introduce the following matrices.
Let $D^{(k)} \in \mathbb{R}^{k\times(k+1)}$ denote the first order difference matrix, i.e.,
\[
D^{(k)} \, := \,
\begin{bmatrix}
-1 & 1 & &  &     &      \\
  & -1  &1& &     &      \\

     & &\ddots  &\ddots  & &\\
     & &  & \ddots&\ddots& \\

    &  &  &&-1 &1 \\
\end{bmatrix} \in \mathbb R^{k \times (k+1)}.
\]
When $m=1$, let $\wt D_{m, T_\kappa} := D^{(K_n-1)}$. In what follows, consider $m>1$. For the given knot sequence $T_\kappa$ with the usual extension $\kappa_k = 0$ for any $k < 0$ and $\kappa_k  = 1$ for any $k > K_n$, define the following diagonal matrices: $\Delta_{0, T_\kappa} := I_{K_n-1}$, and for each $p =1,\ldots,m-1$,
\begin{equation} \label{eqn:matrix_Delta}
   \Delta_{p, T_\kappa} \, := \, \frac{1}{p} \, \mbox{diag} \Big( \kappa_1 - \kappa_{1-p}, \ \kappa_2 - \kappa_{2-p}, \ \ldots, \ \kappa_{K_n+p-1} - \kappa_{K_n-1} \Big) \in \mathbb R^{(K_n+p-1)\times (K_n+p-1)}.
\end{equation}
Furthermore, define the matrices $\wt D_{p, T_\kappa} \in \mathbb R^{(K_n+m-1-p)\times(K_n+m-1)}$ inductively as:
\begin{equation} \label{eqn:matrix_wtD}
   \wt D_{0, T_\kappa} \, := \, I, \quad \mbox{ and } \quad \wt D_{p, T_\kappa} \, := \,  \Delta^{-1}_{m-p, T_\kappa} \cdot D^{(K_n+m-1-p)} \cdot \wt D_{p-1, T_\kappa}, \quad p=1, \ldots, m.
\end{equation}
Roughly speaking, $\wt D_{p, T_\kappa}$ denotes the $p\,$th order difference matrix weighted by the knots of $T_\kappa$. When the knots are equally spaced, it is almost identical to a standard difference matrix (except on the boundary). Moreover, since $\Delta^{-1}_{m-p, T_\kappa}$ is invertible and $D^{(K_n+m-1-p)}$ has full row rank, it can be shown via induction that $\wt D_{p, T_\kappa}$ is of full row rank for any $p$ and $T_\kappa$.

In what follows, define $T_n:=K_n + m -1$ for a fixed spline order $m \in \mathbb N$.

%----------------------------------------------------------------------------
%
\begin{lemma}\label{lem:shape_ineq}
Fix $m \in \mathbb N$. Let $T_\kappa$ be a given  sequence of $(K_n+1)$ knots, and for each $p=1, \ldots, m$, let  $\{B^{T_\kappa}_{p,k}\}_{k=1}^{K_n+p-1}$ denote the B-splines of order $p$ defined by $T_\kappa$. Then the following hold:
\begin{itemize}
  \item [(1)] For any given $b\in \mathbb R^{T_n}$ and $j=0, 1, \ldots, m-1$, the $j$th derivative of $g_{b, T_\kappa}:=\sum^{T_n}_{k=1} b_k B^{T_\kappa}_{m, k}$ is $\sum^{T_n-j}_{k=1} \big( \wt D_{j, T_\kappa} b)_k  B^{T_\kappa}_{m-j, k}$, except at (at most) finitely many points on $[0, 1]$;
  \item [(2)] $ g_{b, T_\kappa} \in \mathcal S_m$ if and only if $\wt D_{m, T_\kappa} \, b \ge 0$.
\end{itemize}
\end{lemma}

%----------------------------------------------------------------------------
%
\begin{proof}For notational simplicity, we write $g_{b, T_\kappa}$ as $g$ and $\wt D_{j, T_\kappa}$ as $\wt D_j$ respectively in the proof.

(1)  We prove statement (1) by induction on $j=0,1,\dots,m-1$. Clearly, the statement holds for $j = 0$.  Consider $j$ with $1\le j \le m-1$, and assume the statement holds for $(j-1)$.
 It follows from (\ref{eqn:B_spline_derivative}),  the induction hypothesis, and the definitions of $\Delta_{j, T_\kappa}$ and $\wt D_{j}$ that
\begin{eqnarray*}
g^{(j)} &= & \left(g^{(j-1)} \right)^\prime  = \left(\sum_{k=1}^{T_n-j+1} \big(\wt D_{j-1} b \big)_k B^{T_\kappa}_{m-j+1,k}  \right)^\prime =  (m-j) \sum_{k=1}^{T_n-j} \frac{(\wt D_{j-1}b)_{k+1}-(\wt D_{j-1}b)_k}{\kappa_{k} - \kappa_{k-m+j}} B^{T_\kappa}_{m-j,k}\\
 &= & \sum_{k=1}^{T_n-j} \left (\Delta_{m-j, T_\kappa}^{-1} D^{(T_n-j)} \wt D_{j-1} b \right)_k \,B^{T_\kappa}_{m-j,k} \ = \ \sum_{k=1}^{T_n-j} \big(\wt D_j b \big)_k \,B^{T_\kappa}_{m-j,k},
\end{eqnarray*}
whenever $g^{(j)}$ and $g^{(j-1)}$ exist. Hence, statement (1) holds for $j$.

(2) It is easily seen that $g^{(m-1)}$ exists on $[0, 1]$  except at (at most) finitely many points in $[0, 1]$.  It thus follows from statement (1) that $g^{(m-1)}$ is a piecewise constant function on $[0, 1]$. Therefore, $g \in \mathcal S_m$ if and only if the spline coefficients of
 $g^{(m-1)}$ are increasing, i.e., $(\wt D_{m-1} b)_k \leq (\wt D_{m-1} b)_{k+1}$ for each $k =1,\ldots,K_n-1$. This is equivalent to $D^{(K_n-1)} \wt D_{m-1} b \geq 0$, which is further equivalent to $ \wt D_m \, b \ge 0$, in view of $\Delta_{0, T_\kappa}=I$ and $\wt D_m= D^{(K_n-1)} \wt D_{m-1}$. This gives rise to statement (2).
\end{proof}

%----------------------------------------------------------------------
%
\subsection{Shape constrained B-splines}  \label{subsec:shape_const_Bsplines}

Let $m\in \mathbb N$ be a fixed spline order throughout the rest of the paper.  Let $y := (y_0,y_1,\dots,y_n)^T \in \mathbb R^{(n+1)}$ be a given sample sequence corresponding to a sequence of  design points $P=(x_i)^n_{i=0}$ on $[0, 1]$. For a given sequence $T_\kappa$ of $(K_n+1)$ knots on $[0, 1]$, consider the following B-spline estimator that satisfies the shape constraint characterized by $\mathcal S_m$:
\begin{equation} \label{eqn:Bspline_estimator}
    \wh {f}^B_{P, T_\kappa} (x) \, := \, \sum_{k=1}^{K_n+m-1} \wh b_k B^{T_\kappa}_{m, k}(x),
\end{equation}
where the coefficient vector $\wh b_{P, T_\kappa} :=(\wh b_k)$ is given by the constrained quadratic optimization problem:
%
%Consider the constrained B-spline optimization problem
%
\begin{equation}\label{eqn:opt_spline_coeff}
\wh b_{P, T_\kappa} \, := \, \argmin_{\wt D_{m, T_\kappa} b \geq 0} \sum^n_{i=0} \big(x_{i+1} - x_i \big) \left( y_i- \sum_{k=1}^{T_n} b_k B^{T_\kappa}_{m, k}(x_i) \right)^2.
\end{equation}
Here $x_{n+1}:=1$. It follows from Lemma~\ref{lem:shape_ineq} that $\wh f^B_{P, T_\kappa} \in \mathcal S_m$. Note that $\wh f^B_{P, T_\kappa}$ depends on $P$ and $T_\kappa$.

Define the diagonal matrix $\Theta_n := \text{diag}(x_1-x_0,\, x_2-x_1,\ldots, x_{n+1} - x_n) \in \mathbb R^{(n+1)\times (n+1)}$, the design matrix $\wh X \in \mathbb R^{(n+1) \times T_n}$ with $[\wh X]_{i, \, k} := B^{T_\kappa}_{m, k}(x_i)$ for each  $i$ and $k$, the matrix $\Lambda_{K_n, P, T_\kappa} := K_n \cdot \wh X^T \Theta_n \wh X \in \mathbb R^{T_n\times T_n}$, and the weighted sample vector $\ol{y} := K_n \cdot \wh X^T\Theta_n y$. Therefore, the quadratic optimization problem in (\ref{eqn:opt_spline_coeff}) for $\wh b_{P, T_\kappa}$ can be written as:
\begin{equation} \label{eqn:opt_spline_coeff_matrix}
  \wh b_{P, T_\kappa}(\bar y) \, := \, \argmin_{\wt D_{m, T_\kappa} \, b \geq 0} \frac{1}{2} b^T \, \Lambda_{K_n, P, T_\kappa} \, b - b^T \ol y.
\end{equation}
%%%\end{equation}

For the given $P, T_\kappa$ and $K_n$, the matrix $\Lambda_{K_n, P, T_\kappa}$ is positive definite, and the function $\wh b_{P, T_\kappa}:\mathbb R^{T_n} \rightarrow \mathbb R^{T_n}$ is thus piecewise linear and globally Lipschitz continuous \cite{Scholtes_thesis94}.
The piecewise linear formulation of $\wh b_{P, T_\kappa}$ can be obtained from the KKT optimality conditions for (\ref{eqn:opt_spline_coeff_matrix}):
\begin{equation}\label{eqn:spline_KKT}
\Lambda_{K_n, P, T_\kappa} \, \wh {b}_{P, T_\kappa} - \ol y - \wt D_{m, T_\kappa}^T \, \chi \, = \, 0, \qquad 0 \, \leq \, \chi \, \perp \, \wt D_{m, T_\kappa} \, \wh{b}_{P, T_\kappa} \, \geq \, 0,
\end{equation}
where $\chi \in \mathbb{R}^{K_n-1}$ is the Lagrange multiplier, and $u \perp v$ means that the vectors $u$ and $v$ are orthogonal. It follows from a similar argument as in  \cite{ShenWang_SICON11, ShenWang_ACC12, WangShen_SICON13} that each linear piece of $\wh b_{P, T_\kappa}$ is characterized by index sets:
\begin{equation} \label{eqn:index_alpha}
 \alpha  \, := \, \left\{ \, i \,:\, \big(\wt D_{m, T_\kappa} \, \wh {b}_{P, T_\kappa} \big)_i = 0 \, \right\} \, \subseteq \{1,\ldots,K_n-1\}.
\end{equation}
Note that $\alpha$ may be the empty set.
For each $\alpha$, the KKT conditions (\ref{eqn:spline_KKT}) become
%
%\begin{equation}\label{eqn:KKT2}
%
\[
    (\wt D_{m, T_\kappa})_{\alpha \bullet} \, \wh{b}_{P, T_\kappa}= 0, \qquad \chi_{\ol{\alpha}}=0, \qquad \Lambda_{K_n, P, T_\kappa}\, \wh{b}_{P, T_\kappa} - \ol y - \big((\wt D_{m, T_\kappa})_{\alpha \bullet} \big)^T \, \chi_\alpha = 0.
\]
%
%\end{equation}
%
Denote the linear piece of $\wh b_{P, T_\kappa}$ for a given $\alpha$ as $\wh b^{\, \alpha}_{P, T_\kappa}$, and let $F_\alpha^T \in \mathbb R^{T_n\times(|\ol\alpha|+m)}$ be a matrix whose columns form a basis for the null space of $(\wt D_{m, T_\kappa})_{\alpha \bullet}$.  It is known \cite{ShenWang_SICON11, ShenWang_ACC12, WangShen_SICON13} that
\begin{equation}\label{eqn:piecewise_linear}
    \wh b^{\,\alpha}_{P, T_\kappa} (\ol y) \, = \, F_\alpha^T \big(F_\alpha \Lambda_{K_n,P, T_\kappa} \, F_\alpha^T \big)^{-1}F_\alpha \,\ol y.
\end{equation}
Note that for any invertible matrix $R \in \mathbb{R}^{(|\ol{\alpha}|+m) \times(|\ol{\alpha}|+m)}$,
$
(R F_\alpha)^T((R F_\alpha) \Lambda_{K_n, P, T_\kappa} (R F_\alpha)^T)^{-1} (R F_\alpha)
=  F_\alpha^T(F_\alpha \Lambda_{K_n, P, T_\kappa} F_\alpha^T)^{-1}F_\alpha.
$
Thus any choice of $F_\alpha$ leads to the same $\wh b^\alpha_{P, T_\kappa}$, provided that the columns of $F_\alpha^T$ form a basis of the null space of $(\wt D_{m, T_\kappa})_{\alpha \bullet}$.
%

%---------------------------------------------------------------------------------
%
\subsection{Uniform Lipschitz property of shape constrained B-splines: main result}

As indicated in the previous subsection, the piecewise linear function $\wh b_{P, T_\kappa}(\cdot)$ is Lipschitz continuous for fixed $K_n, P, T_\kappa$. An important question is whether the Lipschitz constants of size-varying $\wh b_{P, T_\kappa}$ with respect to the $\ell_\infty$-norm are uniformly bounded, regardless of $K_n$, $P$, and $T_\kappa$, as long as the numbers of design points and knots are sufficiently large. If so, we say that $\wh b_{P, T_\kappa}$ satisfies the {\em uniform Lipschitz property}. Originally introduced and studied in \cite{ShenWang_SICON11, ShenWang_ACC12, WangShen_biometrika10, WangShen_SICON13} for monotone P-splines and convex B-splines with equally spaced design points and knots, this property is shown to play a crucial role in the uniform convergence and asymptotic analysis of constrained B-spline estimators. In this section, we extend this property to constrained B-splines subject to general nonnegative derivative constraints under relaxed conditions on design points and knots.

Fix $c_\omega \ge 1$, and for each $n \in \mathbb N$,  define the set of sequences of $(n+1)$ design points on $[0, 1]$:
\begin{equation} \label{eqn:set_of_Pn}
   \mathcal P_n \, := \, \left\{ \,  (x_i)^n_{i=0} \  \Big | \  0 = x_0 < x_1 < \dots < x_n = 1, \ \ \mbox{ and } \ \ x_{i} - x_{i-1} \le \frac{c_\omega}{n}, \ \forall \ i=1, \ldots, n \, \right\}.
\end{equation}
Furthermore, let $c_{\kappa, 1}$ and $c_{\kappa, 2}$ with $0<c_{\kappa, 1} \le 1 \le c_{\kappa, 2}$ be given. For each $K_n \in \mathbb N$, define the set of sequences of $(K_n+1)$ knots on $[0, 1]$ with the usual extension on the left and right boundary:
\begin{equation}  \label{eqn:set_of_T_K_n}
  \mathcal T_{K_n} := \left\{  (\kappa_i)^{K_n}_{i=0} \,  \Big | \,  0=\kappa_0 < \kappa_1 < \cdots < \kappa_{K_n}=1, \mbox{ and } \ \frac{c_{\kappa, 1}}{K_n} \le \kappa_{i} - \kappa_{i-1} \le \frac{c_{\kappa, 2}}{K_n}, \ \forall \, i=1, \ldots, K_n  \right\}.
\end{equation}
For any  $p, K_n\in \mathbb N$ and $T_\kappa \in \Tcal_{K_n}$, it is noted that for any $\kappa_i \in T_\kappa$, we have,
$
   \frac{\kappa_i - \kappa_{i-p} }{p} \, = \, \frac{1}{p} \sum^{i-p+1}_{s=i} ( \kappa_s - \kappa_{s-1} ) \, \le \, \frac{1}{p} \cdot p \cdot \frac{c_{\kappa, 2} }{K_n} \, \le \, c_{\kappa, 2}/K_n.
$
Moreover, in view of $\kappa_{i-p}=0$ for any $i \le p$ and $\kappa_i=1$ for any $i\ge K_n$, it can be shown that for each $1\le i \le K_n+p-1$, $\frac{\kappa_i - \kappa_{i-p} }{p} \ge c_{\kappa, 1}/(p \cdot K_n)$ such that $\frac{p}{\kappa_i - \kappa_{i-p} } \le p\cdot K_n/c_{\kappa, 1}$. In summary, we have,  for each  $i=1, \ldots, K_n+p-1$,
\begin{equation} \label{eqn:knot_derivative_bd}
 \frac{ c_{\kappa, 1} }{p \cdot K_n} \, \le \, \frac{\kappa_i - \kappa_{i-p} }{p} \, \le \, \frac{c_{\kappa, 2} }{K_n}, \qquad \mbox{and} \qquad
  \frac{K_n}{c_{\kappa, 2} } \, \le \, \frac{p}{\kappa_i - \kappa_{i-p} } \, \le \,  \frac{p \cdot K_n}{c_{\kappa, 1} }. %%\ \ \forall \ i=1, \ldots, K_n+p-1.
\end{equation}

Using the above notation, we state the main result of the paper, i.e., the uniform Lipschitz property of $\wh b_{P, T_\kappa}$, as follows:

\begin{theorem} \label{thm:uniform_Lip}
Let $m\in \mathbb N$ and constants $c_\omega, c_{\kappa, 1}, c_{\kappa, 2}$ be fixed, where
$c_\omega \ge 1$ and $0<c_{\kappa,1} \le 1 \le c_{\kappa, 2}$.
For any $n, K_n \in \mathbb N$, %%$P\in \mathcal P_n$ and $T_\kappa\in \mathcal T_{K_n}$,
let $\wh b_{P, T_\kappa}:\mathbb R^{K_n+m-1} \rightarrow  \mathbb R^{K_n+m-1}$ be the piecewise linear function in (\ref{eqn:opt_spline_coeff_matrix}) corresponding to the $m$th order B-spline defined by the design point sequence $P \in \mathcal P_n$ and the knot sequence $T_\kappa \in \mathcal T_{K_n}$. Then there exists a positive constant $c_\infty$, depending on $m, c_{\kappa, 1}$ only, such that for any increasing sequence $( K_n )$ with $K_n \rightarrow \infty$ and $K_n/n \rightarrow 0$ as $n \rightarrow \infty$, there exists $n_* \in \mathbb N$, depending on $(K_n)$ (and the fixed constants $m, c_\omega, c_{\kappa, 1}, c_{\kappa, 2}$) only,  such that for any $P \in \mathcal P_n$ and $T_\kappa \in \mathcal T_{K_n}$ with all $n \ge n_*$,
\[
   \left \| \, \wh b_{P, T_\kappa} (u) - \wh b_{P, T_\kappa}(v) \, \right \|_\infty \, \le \, c_\infty \, \big \| \, u - v \, \big \|_\infty, \quad \forall \ u, v \in \mathbb R^{K_n+m-1}.
\]
\end{theorem}

The above result can be refined when we focus on a particular sequence of  $P$ and $T_\kappa$.

\begin{corollary}
  Let $( K_n )$ be an increasing sequence with $K_n \rightarrow \infty$ and $K_n/n \rightarrow 0$ as $n \rightarrow \infty$, and $\big( (P_n, T_{K_n} ) \big)$ be a sequence in $\mathcal P_n \times \mathcal T_{K_n}$. Then there exists a positive constant $c'_\infty$, independent of $n$, such that for each $n$,
  \[
     \left \| \, \wh b_{P_n, T_{K_n}} (u) - \wh b_{P_n, T_{K_n}}(v) \, \right \|_\infty \, \le \, c'_\infty \, \big \| \, u - v \, \big \|_\infty, \quad \forall \ u, v \in \mathbb R^{K_n+m-1}.
  \]
\end{corollary}

This corollary recovers the past results on the uniform Lipschitz property for $m=1, 2$ (e.g., \cite{WangShen_SICON13})  when the design points and knots are equally spaced on $[0, 1]$.

%------------------------------------------------------
%
\subsection{Overview of the proof} \label{sect:proof_roadmap}

The proof of Theorem~\ref{thm:uniform_Lip} is somewhat technical. To facilitate the reading, we outline its key ideas and provide a road map of the proof  as follows. In view of the piecewise linear formulation of  $\wh b_{P, T_\kappa}$ in (\ref{eqn:piecewise_linear}), it suffices to establish a uniform bound of $\| F_\alpha^T \big(F_\alpha \Lambda_{K_n,P, T_\kappa} \, F_\alpha^T \big)^{-1}F_\alpha \|_\infty$ for all large $n$, regardless of $K_n$, $\alpha$, $P\in \mathcal P_n$, and $T_\kappa \in \Tcal_{K_n}$. Note that $\| F^T_\alpha\|_\infty$ can be uniformly bounded by choosing a suitable basis of the null space of $(\wt D_{m, T_\kappa})_{\alpha \bullet}$. Moreover, motivated by \cite{WangShen_SICON13}, we introduce a diagonal matrix $\Xi'_\alpha$ with positive diagonal entries such that $\| \Xi'_\alpha F_\alpha \|_\infty$ is uniformly bounded. Since $F_\alpha^T \big(F_\alpha \Lambda_{K_n,P, T_\kappa} \, F_\alpha^T \big)^{-1}F_\alpha = F_\alpha^T \cdot \big( \Xi'_\alpha F_\alpha \Lambda_{K_n,P, T_\kappa} \, F_\alpha^T  \big)^{-1} \cdot \Xi'_\alpha F^T_\alpha$, the proof boils down to finding a uniform bound of $\| \big( \Xi'_\alpha F_\alpha \Lambda_{K_n,P, T_\kappa} \, F_\alpha^T  \big)^{-1}  \|_\infty$, one of the key steps of the proof.

 A critical technique for  establishing the uniform bounds of $\| \big( \Xi'_\alpha F_\alpha \Lambda_{K_n,P, T_\kappa} \, F_\alpha^T  \big)^{-1}  \|_\infty$ and other related matrices is based on  a deep result in B-spline theory (dubbed de Boor's conjecture) proved by Shardin \cite{Shardin_AM01}. Roughly speaking, this result says that the $\ell_\infty$-norm of the inverse of the Gramian formed by the normalized B-splines of order $m$ is uniformly bounded, regardless of the knot sequence and the number of B-splines. To formally describe this result, we introduce more notation. Let $\langle \cdot, \cdot \rangle$ denote the $L_2$-inner product of real-valued univariate functions on $\mathbb R$, i.e., $\langle f, g\rangle := \int_{\mathbb R} f(x)g(x)\,dx$, and $\|\cdot\|_{L_1}$ denote the $L_1$-norm of a real-valued univariate function on $\mathbb R$, i.e., $\|f\|_{L_1} := \int_{\mathbb R} |f(x)|\,dx$.

\begin{theorem}\label{thm:Shardin_bound}\cite[Theorem~I]{Shardin_AM01}
Fix a spline order $m \in \mathbb N$. Let $a,b \in \mathbb R$ with $a < b$, and $\{ B^{T_\kappa}_{m, k} \}^{K+m-1}_{k=1}$ be the $m$th order B-splines  on $[a, b]$ defined by a knot sequence $T_\kappa:=\{ a=t_0 < t_1 < \dots < t_K=b \}$ for some $K \in \mathbb N$.
 Let $G \in \mathbb R^{(K+m-1)\times (K+m-1)}$ be the Gramian matrix given by
 \[
    \big[ G \, \big]_{i, j} \ := \  \frac{  \left \langle \, B^{T_\kappa}_{m, i}, \ B^{T_\kappa}_{m, j}  \right \rangle } {  \big\| B^{T_\kappa}_{m, i} \big\|_{L_1} }, %=
    \qquad \forall \ i,j \,= \, 1, \ldots, K+m-1.
 \]
Then there exists a positive constant $\rho_m$, independent of $a, b, T_\kappa$ and $K$, such that $\|G^{-1}\|_\infty \leq \rho_m $.
\end{theorem}

Inspired by this theorem, we intend to approximate $\Xi'_\alpha F_\alpha \Lambda_{K_n,P, T_\kappa} \, F_\alpha^T$ and other relevant matrices by appropriate Gramian matrices of B-splines with uniform approximation error bounds. To achieve this goal, we construct a suitable matrix $F_\alpha$ (i.e., $F^{(m)}_{\alpha, T_\kappa}$) in Section~\ref{sect:construction_F} and approximated design matrix $X_{m, T_\kappa, L_n}$, which leads to $\wt \Lambda_{T_\kappa, K_n, L_n}$ that approximates $\Lambda_{K_n,P, T_\kappa}$, in Section~\ref{sect:approx_Lambda_matrix}. We then show via analytical tools and Theorem~\ref{thm:Shardin_bound} that these constructed matrices attain uniform bounds or uniform approximation error bounds  in Section~\ref{sect:uniform_bds}. With the help of these bounds, the uniform Lipschitz property is proven in Section~\ref{sect:proof_unif_Lip_final}.

%\newpage

%---------------------------------------------------------------------------------
%
\section{Proof of the Uniform Lipschitz Property} \label{sect:proof_unif_Lip}

In this section, we prove the uniform Lipschitz property stated in Theorem~\ref{thm:uniform_Lip}.

%---------------------------------------------------------------
%
\subsection{Technical lemmas}

%%%{\bf Technical preliminaries}

We present two technical lemmas for the proof of Theorem~\ref{thm:uniform_Lip}. The first lemma characterizes the difference between an integral of a continuous function and its discrete approximation; it  will be used multiple times through this section (cf. Propositions~\ref{prop:Z_approx_Bspline}, \ref{prop:Galpha_approximation}, and \ref{prop:error_bd_Lamda}).

\begin{lemma}\label{lem:funct_bound}
Let $\wt n \in \mathbb N$, $v=(v_k)\in \mathbb R^{\wt n}$, $[a,b] \subset \mathbb R$ with  $a<b$, and points $\{s_k\}_{k=0}^{\wt n}$ with $a = s_0 < s_1 < \cdots < s_{\wt n}=b$ such that $\max_{k=1, \ldots \wt n} |s_k - s_{k-1}| \leq \varrho$ for some $\varrho > 0$.
Let  $f:[a,b] \rightarrow \mathbb R$ be a continuous function which is differentiable on $[a,b]$ except at finitely many points in $[a,b]$.
Let the positive constants $\mu_1$ and $\mu_2$ be such that  %%we have $\|f\|_\infty := \|f\|_{\infty,[a,b]} \leq \mu_1$,
$\max_{k=1, \ldots \wt n} | v_k - f(s_{k-1})| \leq \mu_1$, and $|f^\prime(x)| \leq \mu_2$ for any $x\in [a, b]$ where $f'(x)$ exists. Then for each $i \in \{1,\ldots,\wt n\}$,
\[
\left| \, \sum_{k=1}^i v_k \, \big(s_k - s_{k-1} \big) - \int_{s_0}^{s_i} f(x)\,dx \, \right| \, \leq \, \mu_1(b-a) +\frac{3}{2}\,\mu_2\varrho (b-a).
\]
\end{lemma}

\begin{proof}
Fix an arbitrary $k\in \{1,\ldots,\wt n \}$. Suppose that $\wt s_1, \ldots, \wt s_{\ell-1} \in (s_{k-1}, s_k)$ with
 $s_{k-1}:=\wt s_0< \wt s_1< \wt s_2< \cdots <\wt s_{\ell-1} < \wt s_{\ell}:=s_{k}$ are the only points where $f$ is non-differentiable on the interval $(s_{k-1}, s_k)$. It follows from the continuity of $f$ and the Mean-value Theorem that for each $j=1, \ldots, \ell$, there exists $\xi_j \in (\wt s_{j-1}, \wt s_{j})$ such that $f(\wt s_j)=f(\wt s_{j-1}) + f'(\xi_j)(\wt s_{j}- \wt s_{j-1})=f(s_{k-1})+\sum^j_{r=1} f'(\xi_r)(\wt s_r-\wt s_{r-1})$. Since $f$ is continuous and piecewise differentiable, we have
\begin{eqnarray*}
\lefteqn{ \left| \,\int_{s_{k-1}}^{s_k} f(x) \,dx -  \big(s_k - s_{k-1} \big) \,f(s_{k-1}) \, \right| } \\
& = & \left | \, \sum^{\ell}_{j=1} \int_{\wt s_{j-1}}^{\wt s_{j}} \Big[ f(\wt s_{j-1}) + f'(\xi_x) \big(x - \wt s_{j-1} \big) \Big]dx - \big(s_k- s_{k-1} \big) \,f(s_{k-1})  \, \right|  \\
 & \le & \left | \, \sum^{\ell}_{j=1} \left [ f(s_{k-1}) + \sum^{j-1}_{r=1} f'(\xi_r) (\wt s_r - \wt s_{r-1}) \right ] (\wt s_j - \wt s_{j-1} ) - \big(s_k- s_{k-1} \big) \,f(s_{k-1}) \, \right| \\
 & & \   + \  \left | \, \sum^{\ell}_{j=1} \int_{\wt s_{j-1}}^{\wt s_{j}} f'(\xi_x) \big(x - \wt s_{j-1} \big)dx  \, \right| %\\
%
% &\leq &
%
\, \le \, \left | \, \sum^\ell_{j=1} \, f'(\xi_j) \big(\wt s_{j} - \wt s_{j-1} \big) \big(s_k - \wt s_{j} \big) \right | + \frac{\mu_2}{2} \sum^\ell_{j=1} \big( \wt s_j - \wt s_{j-1} \big)^2 \\
 & \leq & \frac{3 \mu_2}{2}  \Big( s_k - s_{k-1} \Big)^2 \ \le \ \frac{3 \mu_2 \varrho}{2}  \Big( s_k - s_{k-1} \Big).
\end{eqnarray*}
Consequently, for each $i \in \{1,\ldots,\wt n\}$,
\begin{eqnarray*}
\lefteqn{ \left|\sum_{k=1}^i v_k\,\big( s_k- s_{k-1} \big) -  \int_{s_0}^{s_i} f(x)\,dx \right|
\ \leq \ \sum_{k=1}^i \left|v_k\, \big( s_k- s_{k-1} \big) - \int_{s_{k-1}}^{s_k} f(x)\,dx\right| }\\
  & \leq & \sum_{k=1}^i \big( s_k- s_{k-1} \big) \big | v_k  - f(s_{k-1})\big | \, + \, \sum_{k=1}^i \left| \big( s_k- s_{k-1} \big) \, f(s_{k-1}) - \int_{s_{k-1}}^{s_k} f(x)\,dx\right| \\
& \leq & \mu_1 \sum_{k=1}^i \big( s_k- s_{k-1} \big) \, + \, \frac{3 \mu_2 \varrho}{2}\sum_{k=1}^i \big( s_k - s_{k-1} \big)
 %%% \, \leq \, \mu_2(b-a)+2J\mu_1\bar r+  \frac{3 \mu_3\rho}{2}\sum_{k=1}^i  \big( s_k- s_{k-1} \big) \\
\ \leq \ \mu_1(b-a) +\frac{3}{2}\, \mu_2\varrho (b-a).
\end{eqnarray*}
This completes the proof.
\end{proof}

%----------------------------------------------------------------------------
%
The second lemma asserts that if corresponding matrices from two families of square matrices are sufficiently close and the matrices from one family are invertible with uniformly bounded inverses, then so are the matrices from the other family.
This result is instrumental to attaining a uniform bound of the inverses of some size-varying matrices (cf. Corollary~\ref{coro:invert_HHt}).

\begin{lemma}\label{lem:matrix_approx}
Let $\{A_i \in \mathbb R^{n_i\times n_i} \,:\, i \in \Ical \, \}$ and $\{B_i \in \mathbb R^{n_i\times n_i} \,:\, i \in \Ical \, \}$ be two families of square matrices for a (possibly infinite) index set $\Ical$,  where
 $n_i \in \mathbb N$ need not be the same for different $i\in \Ical$.
Suppose that each $A_i$ is invertible with $\mu:=\sup_{i\in \Ical} \| A_i^{-1} \|_\infty<\infty$ and that for any $\varepsilon >0$, there are only finitely many $i \in \Ical$ satisfying
$\| A_i - B_i \|_\infty \geq \varepsilon$.  Then for all but finitely many $i\in \Ical$, $B_i$ is invertible with $\| B_i^{-1} \|_\infty \le \frac{3}{2} \mu$.
\end{lemma}

\begin{proof}
    For the given positive constant $\mu:=\sup_{i\in \Ical} \| A_i^{-1} \|_\infty<\infty$, define the positive constant $\varepsilon := 1/(3\mu)$.
 Let $\mathcal I_\varepsilon : =\{ i \in \Ical \, : \, \| A_i- B_i \|_\infty<\varepsilon\}$. Note that there exist only finitely many $i \in \Ical$ such that $\| A_i - B_i \|_\infty \geq \varepsilon$.
Define $C_i := B_i - A_i$ so that $\|C_i \|_\infty < \varepsilon$  for each $i\in \Ical_\varepsilon$.
Since $B_i = A_i + C_i$ and $A_i$ is invertible, we have $A^{-1}_i B_i = I + A^{-1}_i C_i$. Hence, we obtain, via $\varepsilon= 1/(3\mu)$,
\[
   \left\| A^{-1}_i C_i \right\|_\infty \, \le \, \| A^{-1}_i \|_\infty \cdot \| C_i \|_\infty \le \mu \cdot \varepsilon \, = \, \frac{1}{3}, \qquad \forall \ i\in  \Ical_\varepsilon.
\]
This shows that $I+ A^{-1}_i C_i$ is strictly diagonally dominant, and thus is invertible. Therefore, $A^{-1}_i B_i$ is invertible, and so is $B_i$ for each $i\in \Ical_\varepsilon$. Hence all but finitely many $B_i, i\in \Ical$, are invertible.

By the virtue of $\left\| A^{-1}_i C_i \right\|_\infty \le 1/3$ for any $i\in \Ical_\varepsilon$, we  deduce that
\[
   \left\| \big(I + A^{-1}_i C_i\big)^{-1} \right\|_\infty \, \le \, \frac{1}{1 - \left\| A^{-1}_i C_i \right\|_\infty} \, \le \, \frac{3}{2}, \qquad \forall \ i\in  \Ical_\varepsilon.
\]
Using $A^{-1}_i B_i = I + A^{-1}_i C_i$ again, we further have that
\[
   \big\| B^{-1}_i \big\|_\infty \, = \, \left \|   \big(I + A^{-1}_i C_i\big)^{-1} \cdot A^{-1}_i \right\|_\infty \, \le \, \left \|   \big(I + A^{-1}_i C_i\big)^{-1} \right\|_\infty \cdot \| A^{-1}_i \|_\infty \, \le \, \frac{3}{2} \mu, \quad \forall \ i\in  \Ical_\varepsilon.
\]
This yields the desired result.
\end{proof}

%------------------------------------------------------------------------------
%
\subsection{Construction of the matrix $F_\alpha$} \label{sect:construction_F}

In this subsection, we construct a suitable matrix $F_\alpha$ used in (\ref{eqn:piecewise_linear}).
For $K_n\in \mathbb N$, let $T_\kappa \in \mathcal T_{K_n}$ be a knot sequence, and   $\alpha \subseteq \{1, \ldots, K_n-1\}$ be an index set defined in (\ref{eqn:index_alpha}). The complement of $\alpha $ is $\ol{\alpha} = \{i_1,\ldots,i_{|\ol{\alpha}|}\}$ with $1 \leq i_1 < \cdots < i_{|\ol{\alpha}|} \leq K_n-1$.  For notational simplicity, define $q_\alpha := |\ol{\alpha}|+m$.

%%Let $I_{m-1}$ denote the identity matrix of order $(m-1)$.
%
We introduce the following two matrices, both of which have full row rank:
\begin{equation} \label{eqn:E_F1_matrices}
E_{\alpha, T_\kappa} :=
\begin{bmatrix}
\onebld^T_{i_1} &				& &\\
	       	& \onebld^T_{i_2-i_1}   & &\\
               	&			&\ddots&  \\
               	&			& &\onebld^T_{K_n-i_{|\ol{\alpha}|}}\\
\end{bmatrix} \in \mathbb{R}^{(|\ol \alpha|+1) \times K_n}, \ %%\ \text{and} \
F^{(1)}_{\alpha, T_\kappa}  :=
\begin{bmatrix}
I_{m-1}& 0 \\
 0 & E_{\alpha, T_\kappa} \\
\end{bmatrix} \in \mathbb R^{q_\alpha \times T_n},
\end{equation}
where we recall $T_n:=K_n+m-1$.
%
%which both have full row rank.
%
Note that each column of $F^{(1)}_{\alpha, T_\kappa}$ contains exactly one entry of 1, and all other entries are zero. The matrix $F^{(1)}_{\alpha, T_\kappa}$ characterizes the first order B-splines (i.e. the piecewise constant splines) with the knot sequence $\{ \kappa_{i_k} \}$ defined by $\ol\alpha$.
For the given $\alpha$,  define
\begin{equation}\label{eqn:tau_k}
\tau_{\alpha, {T_\kappa}, k} \, := \,
\begin{cases}
0,			   & \text{ for }  k = 1-m,\ldots,0\\
\kappa_{i_k}, 	   & \text{ for }  k = 1,\ldots,|\ol \alpha|\\
1,			   & \text{ for }  k = |\ol \alpha|+1,\ldots, q_\alpha. %%|\ol\alpha| + m
\end{cases}
\end{equation}
%
%where we note the dependence on $\alpha$, as each $i_k$ depends on $\alpha$.
%
It is easy to verify that for each $k \in \{1,\ldots,|\ol \alpha|+1\}$ and $\ell \in \{1,\ldots,K_n\}$,
\begin{equation}\label{eqn:entry_F1}
\left(F^{(1)}_{\alpha, T_\kappa} \right)_{(k+m-1),(\ell+m-1)} = \left(E_{\alpha, T_\kappa} \right)_{k\ell} =
\left\{
\begin{array}{ll}
1, & \text{ if }\; \kappa_{\ell-1} \in [\tau_{\alpha, {T_\kappa}, k-1},\tau_{\alpha, {T_\kappa}, k})\\
0, & \text{ otherwise}.
\end{array}
\right.
\end{equation}
For each $p=1, \ldots, m$, we further introduce the following diagonal matrices
\begin{align}
\Xi^{(p)}_{\alpha, T_\kappa} & \, :=  \, \begin{bmatrix} I_{m-p} & 0 \\ 0 & \Sigma_{p, T_\kappa} \end{bmatrix}  \label{eqn:Xi_matrix} \\
 & \, = \, \text{diag} \left(\underbrace{1,\ldots,1,}_{(m-p)-\text{copies}} \frac{p}{ \tau_{\alpha, T_\kappa, 1}- \tau_{\alpha, T_\kappa, 1-p} }, \frac{p}{ \tau_{\alpha, T_\kappa, 2}- \tau_{\alpha, T_\kappa, 2-p} }, \ldots, \frac{p}{ \tau_{\alpha, T_\kappa, |\ol\alpha|+p}- \tau_{\alpha, T_\kappa, |\ol\alpha|} }
\right),   \notag
\end{align}
where $\Xi^{(p)}_{\alpha, T_\kappa}$ is of order $q_\alpha$, and by using the definition of the matrix $\Delta_{p, T_\kappa}$ in (\ref{eqn:matrix_Delta}),
\begin{align}
\wh \Delta_{p, T_\kappa}  & \, := \, \begin{bmatrix} I_{m-p} & 0 \\ 0 & \Delta_{p, T_\kappa} \end{bmatrix} \ \label{eqn:wt_Delta} \\
& \, = \,
\text{diag} \left(\underbrace{1,\ldots,1,}_{(m-p)-\text{copies}}\,
\frac{\kappa_1-\kappa_{1-p}}{p}, \frac{\kappa_2-\kappa_{2-p}}{p}, \ldots, \frac{\kappa_{K_n+p-1}-\kappa_{K_n-1}}{p} \right)\in \mathbb R^{T_n \times T_n}. \notag
%
 %%%\in \mathbb R^{T_n \times T_n}
\end{align}
In addition, define the following two matrices of order $r\in \mathbb N$:
\begin{equation} \label{eqn:def_Swh}
\wh S^{(r)} \, : = \, \begin{bmatrix}
1&1&1&\hdots&1\\
&1&1&\hdots&1\\
&&1&\hdots&1\\
&&&\ddots&\vdots\\
&&& &1\\
\end{bmatrix}, \quad %%%\in \mathbb{R}^{r\times r} \
\text{ and } \quad \wh D^{(r)} := \big(\wh{S}^{(r)} \big)^{-1} = \begin{bmatrix}
1&-1&&&\\
&1&-1&&\\
&&\ddots&\ddots&\\
&&&1&-1\\
&&&&1\\
\end{bmatrix}. %%%\in \mathbb{R}^{r\times r}.
\end{equation}
Here multiplication by the matrix $\wh S^{(r)}$ from right acts as discrete integration, while $\wh D^{(r)}$ is similar to a difference matrix.

With the above notation, we now define $F^{(p)}_{\alpha, T_\kappa}$ inductively: $F^{(1)}_{\alpha, T_\kappa}$ is defined in (\ref{eqn:E_F1_matrices}), and
\begin{equation}  \label{eqn:def_F}
F^{(p)}_{\alpha, T_\kappa} \, := \,
\wh D^{(q_\alpha)} \cdot \Xi^{(p-1)}_{\alpha, T_\kappa} \cdot F^{(p-1)}_{\alpha, T_\kappa} \cdot \wh \Delta_{p-1, T_\kappa} \cdot \wh S^{(T_n)} \, \in \, \mathbb R^{q_\alpha \times T_n}, \qquad p=2, \ldots, m.
\end{equation}
Since $\wh D^{(q_\alpha)}$, $\Xi^{(p-1)}_{\alpha, T_\kappa}$, $\wh \Delta_{p-1, T_\kappa}$, and  $\wh S^{(T_n)}$  are all invertible, each $F^{(p)}_{\alpha, T_\kappa}$ has full row rank by induction. Furthermore, it is easy to see that if $\alpha$ is the empty set, then $\alpha=\{ 1, \ldots, K_n-1\}$ so that $F^{(1)}_{\alpha, T_k}$ is the identity matrix $I_{T_n}$ for any $T_\kappa$, which further shows via induction in (\ref{eqn:def_F}) that $F^{(p)}_{\alpha, T_\kappa} = I_{T_n}$ for any $p=2, \ldots, m$ and any $T_\kappa$.

It is shown below that $F^{(m)}_{\alpha, T_\kappa}$ constructed above is a suitable choice for $F_\alpha$ in the piecewise linear formulation of $\wh b_{P, T_\kappa}$ in (\ref{eqn:piecewise_linear}).

%--------------------------------------------------
%
\begin{proposition}\label{prop:null_basis_new}
If the index set $\alpha$ is nonempty, then
the columns of $\big(F^{(m)}_{\alpha, T_\kappa} \big)^T$ form a basis of the null space of $(\wt D_{m, T_\kappa})_{\alpha \bullet}$.
\end{proposition}

\begin{proof}
To simplify notation, we use $F_{\alpha, p}$ to denote $F^{(p)}_{\alpha, T_\kappa}$ and drop the subscript $T_\kappa$ in $\wt D_{p, T_\kappa}$, $\Delta_{p, T_\kappa}$, and $\wh \Delta_{p, T_\kappa}$ for $p=1, \ldots, m$; see (\ref{eqn:matrix_Delta})-(\ref{eqn:matrix_wtD}) for the definitions of $\Delta_{p, T_\kappa}$ and $\wt D_{p, T_\kappa}$ respectively.

We first show that the matrix product $\big(D^{(T_n-1)} \big)_{(\alpha+m-1)\bullet} F^T_{\alpha, 1}=0$. Since the first $(m-1)$ columns of $\big(D^{(T_n-1)}\big)_{(\alpha+m-1)\bullet}$ contain all zero entries and
\[
   \big(F_{\alpha, 1}^T \big)_{\bullet,1:(m-1)}  \, = \, \begin{bmatrix} I_{m-1}\\ 0 \end{bmatrix} \in \mathbb R^{T_n \times (m-1)},
\]
we see that the first $(m-1)$ columns of $\big(D^{(T_n-1)} \big)_{(\alpha+m-1)\bullet} F^T_{\alpha, 1}$ have all zero entries. (If $m=1$, this result holds trivially.)
Moreover, for each $j$ with $ m\le j \le q_\alpha:=|\ol\alpha|+m$,  we see, in light of (\ref{eqn:E_F1_matrices}) and (\ref{eqn:entry_F1}), that
\[
(F_{\alpha, 1})_{j \bullet} \, = \, \begin{pmatrix} \underbrace{0,\ldots,0}_{i_{k-1}+m-1}, & \underbrace{1,\ldots,1}_{i_k-i_{k-1}}, &\underbrace{0,\ldots,0}_{K_n-i_k}
\end{pmatrix}.
\]
where $k= j-m+1$, $i_0 := 0$, and $i_{|\ol{\alpha}|+1} := K_n$.
Note that for each $s\in \alpha$, the row $\big(D^{(T_n-1)}\big)_{(s+m-1)\bullet}$ is of the form $(\underbrace{0,\dots,0}_{s+m-2},-1,1,\underbrace{0,\dots,0}_{K_n-1-s})$. Using the fact that $i_{k-1},i_k\in \ol{\alpha} \cup \{0,K_n\}$ for any $k = 1,\ldots,|\ol{\alpha}|+1$, we deduce that if $s\in \alpha$, then $s \notin \{i_{k-1}, i_k \}$ for $k= 1,\ldots,|\ol{\alpha}|+1$. This shows that
$
  \big(D^{(T_n-1)}\big)_{(s+ m-1)\bullet} \cdot \big(F_{\alpha, 1}^T \big)_{\bullet j} = 0,  \ \ \forall \ s \in \alpha.
   %%%,  \quad \text{ if } \ s \notin \{i_{k-1}, i_k \}.
$
Hence the matrix product $\big(D^{(T_n-1)} \big)_{(\alpha+m-1)\bullet} F^T_{\alpha, 1}=0$.

It is easy to show via (\ref{eqn:matrix_wtD}), (\ref{eqn:E_F1_matrices}), and the above result that the proposition holds when $m=1$. Consider $m \ge 2$ for the rest of the proof. Let $S_0:=I_{T_n}$, and $S_p := \wh \Delta_{m-p} \cdot \wh S^{(T_n)} \cdot S_{p-1}$ for $p=1, \ldots, m-1$, where $\wh S^{(T_n)}$ is defined in (\ref{eqn:def_Swh}). It follows from the definition of $S_p$ and (\ref{eqn:def_F}) that $F_{\alpha, m} = Q \cdot F_{\alpha, 1} \cdot S_{m-1}$ for a suitable matrix $Q$.
We next show via induction on $p$ that %%for each $p=0,1, \ldots, m-1$,
\begin{equation}\label{eqn:wtD_S}
\wt D_p \cdot S_p^T  \, = \, \begin{bmatrix} 0_{(T_n-p)\times p} & I_{T_n-p} \end{bmatrix}, \qquad \forall \ p=0,1, \ldots, m-1.
\end{equation}
Clearly, this result holds for $p= 0$. Given $p \ge 1$ and assuming that (\ref{eqn:wtD_S}) holds for $p-1$, it follows from (\ref{eqn:matrix_wtD}), (\ref{eqn:def_F}), and the induction hypothesis that
\begin{eqnarray*}
\wt D_p \cdot S^T_p &= & \Big( \Delta^{-1}_{m-p} \, D^{(T_n-p)} \wt D_{p-1} \Big) \cdot \Big( S_{p-1}^T \, ( \wh S^{(T_n)} )^T \, \wh \Delta_{m-p} \Big)\\
&= & \Delta_{m-p}^{-1} \, D^{(T_n-p)} \begin{bmatrix} 0_{(T_n-(p-1))\times (p-1)} & I_{T_n-(p-1)} \end{bmatrix} \big( \wh S^{(T_n)} \big)^T  \, \wh \Delta_{m-p}\\
& = & \Delta_{m-p}^{-1} \, D^{(T_n-p)} \begin{bmatrix}\onebld_{(T_n-(p-1))\times(p-1)} & \big(\wh S^{(T_n-(p-1))} \big)^T \end{bmatrix} \, \wh \Delta_{m-p} \\
& = &  \Delta^{-1}_{m-p} \, \begin{bmatrix} 0_{(T_n-p)\times p} & I_{T_n-p} \end{bmatrix} \, \wh \Delta_{m-p}\\
& = & \Delta^{-1}_{m-p} \, \begin{bmatrix}0_{(T_n-p)\times p} & \Delta_{m-p} \end{bmatrix}
 \, = \,  \begin{bmatrix}0_{(T_n-p)\times p} & I_{T_n-p} \end{bmatrix},
\end{eqnarray*}
where the second to last equality is due to (\ref{eqn:wt_Delta}). This gives rise to (\ref{eqn:wtD_S}).

Combining the above results, we have
\begin{align*}
  (\wt D_m)_{\alpha \bullet} \cdot F_{\alpha, m}^T & =  \Big( \big( D^{(K_n-1)} \big)_{\alpha \bullet}\, \wt D_{m-1}  \Big) \cdot \Big(Q \, F_{\alpha, 1} \, S_{m-1} \Big)^T =  \big( D^{(K_n-1)} \big)_{\alpha \bullet} \cdot \wt D_{m-1} \cdot S^T_{m-1} \cdot F^T_{\alpha, 1} \cdot Q^T \\
  & =  \big( D^{(K_n-1)} \big)_{\alpha \bullet} \begin{bmatrix}0_{K_n\times (m-1)} & I_{K_n} \end{bmatrix}  F^T_{\alpha, 1} \cdot Q^T \, = \, \big( D^{(T_n-1)} \big)_{(\alpha+m-1) \bullet} \cdot F^T_{\alpha, 1} \cdot Q^T \, = \, 0.
\end{align*}
Recall that $F_{\alpha, m}$ has full row rank. Hence, the $q_\alpha$ columns of $F_{\alpha, m}^T$ are linearly independent. Additionally, since $\wt D_{m}$ is of full row rank as indicated after (\ref{eqn:matrix_wtD}), so is $(\wt D_m)_{\alpha \bullet}$. Therefore, $\text{rank}[(\wt D_m)_{\alpha \bullet}] = |\alpha|$ and
the null space of $(\wt D_m)_{\alpha \bullet}$ has dimension $(K_n+m-1-|\alpha|)$, which is equal to $q_\alpha$ in light of $|\alpha|+|\ol\alpha|=K_n-1$. Therefore the columns of $F_{\alpha, m}^T$ form a basis for the null space of $(\wt D_m)_{\alpha \bullet}$. %%({\bf more ...})
\end{proof}

Before ending this subsection, we present a structure property of $F_{\alpha,T_\kappa}^{(p)}$ and a preliminary uniform bound for $\big\| F_{\alpha,T_\kappa}^{(m)} \big\|_\infty$, which will be useful later (cf. Corollary~\ref{coro:F_T_XiF_bounds} and Proposition~\ref{prop:F_times_B_T=B_alpha}).

\begin{lemma} \label{lem:F_m_bound}
For any  $m, K_n \in \mathbb N$, any knot sequence $T_\kappa \in \Tcal_{K_n}$ and any index set $\alpha$ defined in (\ref{eqn:index_alpha}), the following hold:
\begin{itemize}
  \item [(1)] For each $p=1, \ldots, m-1$,
    $
       F_{\alpha,T_\kappa}^{(p)} = \begin{bmatrix} I_{m-p} & 0 \\ 0 & W^{(p)}_{\alpha, T_\kappa} \end{bmatrix}
    $
    for some matrix $W^{(p)}_{\alpha, T_\kappa} \in \mathbb R^{(|\ol \alpha|+p)\times (K_n+p-1)}$.
  \item[(2)]
$\displaystyle
  \big\|F_{\alpha,T_\kappa}^{(m)} \big\|_\infty \, \leq \, \left(\frac{2m}{c_{\kappa,1}}\cdot \max\Big(1, \frac{c_{\kappa, 2}}{K_n} \Big) \cdot  T_n \right)^{m-1} \cdot \big( K_n \big)^{m},
$
where $T_n=K_n+m-1$.
 \end{itemize}
\end{lemma}

\begin{proof}
(1) Fix $m$, $K_n$, $T_\kappa \in \Tcal_{K_n}$, and $\alpha$.  We prove this result by induction on $p$. By the definition of $F_{\alpha,T_\kappa}^{(1)} $ in (\ref{eqn:E_F1_matrices}), we see that statement (1) holds for $p=1$ with $W^{(p)}_{\alpha, T_\kappa} = E_{\alpha, T_\kappa}$. Now suppose statement (1) holds for $p=1, \ldots, p^{\,\prime}$ with $p^{\,\prime}\le m-2$, and consider $p^{\,\prime}+1$. In view of the recursive definition (\ref{eqn:def_F}) and the definitions of $\wh D^{(q_\alpha)}$, $\Xi_{\alpha,T_\kappa}^{(p)}$,  $\wh \Delta_{p,T_\kappa}$, and $\wh S^{(T_n)}$ given in~(\ref{eqn:Xi_matrix}),~(\ref{eqn:wt_Delta}), and (\ref{eqn:def_Swh}), we deduce via the induction hypothesis that
\begin{eqnarray*}
       F^{(p^{\,\prime}+1)}_{\alpha, T_\kappa} & = & \wh D^{(q_\alpha)} \cdot \Xi^{(p^{\,\prime})}_{\alpha, T_\kappa} \cdot F^{(p^{\,\prime})}_{\alpha, T_\kappa} \cdot \wh \Delta_{p^{\,\prime}, T_\kappa} \cdot \wh S^{(T_n)} \\
       &  = & \wh D^{(q_\alpha)} \cdot \begin{bmatrix} I_{m-p^{\,\prime}} & 0 \\ 0 & \Sigma_{ p^{\,\prime}, T_\kappa } \end{bmatrix} \cdot \begin{bmatrix} I_{m-p^{\,\prime}} & 0 \\ 0 & W^{(p^{\,\prime})}_{\alpha, T_\kappa} \end{bmatrix} \cdot \begin{bmatrix} I_{m-p^{\,\prime}} & 0 \\ 0 & \Delta_{ p^{\,\prime}, T_\kappa }  \end{bmatrix} \cdot \wh S^{(T_n)} \\
       &  = &  \wh D^{(q_\alpha)} \cdot  \begin{bmatrix} \wh S^{(m-p^{\,\prime})} & \mathbf 1_{(m-p^{\,\prime})\times (K_n+p^{\,\prime}-1)} \\ 0 & \star \end{bmatrix} \, = \, \begin{bmatrix} I_{(m-p^{\,\prime}-1)} & 0 \\ 0 & \star' \end{bmatrix},
     \end{eqnarray*}
     where $\star$ and $\star'$ are suitable submatrices, and the last two equalities follow from the structure of $\wh S^{(T_n)}$ and $\wh D^{(q_\alpha)}$. Letting $W^{(p^{\,\prime}+1)}_{\alpha, T_\kappa}:=\star'$, we obtain the desired equality via induction.

(2) It follows from the definition of $F_{\alpha,T_\kappa}^{(1)}$ in (\ref{eqn:E_F1_matrices}) that $\big\|F_{\alpha,T_\kappa}^{(1)} \big\|_\infty \leq K_n$. Furthermore,  by~(\ref{eqn:def_F}) and the definitions of $\wh D^{(q_\alpha)}$, $\Xi_{\alpha,T_\kappa}^{(p)}$,  $\wh \Delta_{p,T_\kappa}$, and $\wh S^{(T_n)}$ given in~(\ref{eqn:Xi_matrix}),~(\ref{eqn:wt_Delta}), and (\ref{eqn:def_Swh}), we have, for each $p=1, \ldots, m-1$,
\begin{eqnarray*}
\big\|F_{\alpha,T_\kappa}^{(p+1)} \big\|_\infty & = &
 \left \|\wh D^{(q_\alpha)} \cdot \Xi^{(p)}_{\alpha, T_\kappa} \cdot F^{(p)}_{\alpha, T_\kappa} \cdot \wh \Delta_{p, T_\kappa} \cdot \wh S^{(T_n)} \right \|_\infty \\
& \leq &
 \big\|\wh D^{(q_\alpha)} \big\|_\infty \cdot \big \| \Xi^{(p)}_{\alpha, T_\kappa} \big\|_\infty \cdot \big \|F^{(p)}_{\alpha, T_\kappa} \big\|_\infty \cdot \big \|\wh \Delta_{p, T_\kappa}\big \|_\infty\cdot \big \|\wh S^{(T_n)} \big\|_\infty \\
   & \leq & 2 \cdot \frac{mK_n}{c_{\kappa,1}} \cdot \big \|F_{\alpha,T_\kappa}^{(p)} \big\|_\infty \cdot \max\left(1, \frac{c_{\kappa, 2}}{K_n} \right) \cdot  T_n, %%\leq \frac{3m}{c_{\kappa,1}} K_n^2  \|F_{\alpha,T_\kappa}^{(p)}\|_\infty,
\end{eqnarray*}
where we use (\ref{eqn:knot_derivative_bd}) to bound $\big \| \Xi^{(p)}_{\alpha, T_\kappa} \big\|_\infty$ and $\big \|\wh \Delta_{p, T_\kappa}\big \|_\infty$ in the last inequality. In view of this result and $\|F_{\alpha,T_\kappa}^{(1)}\|_\infty \leq K_n$, the desired inequality follows.
\end{proof}

More properties of $F^{(m)}_{\alpha, T_\kappa}$ will be shown in Proposition~\ref{prop:F_times_B_T=B_alpha} and Corollary~\ref{coro:F_T_XiF_bounds}.

%-------------------------------------------------------------------------
%
\subsection{Approximation of the matrix $\Lambda_{K_n, P, T_\kappa}$} \label{sect:approx_Lambda_matrix}

This subsection is concerned with the construction of certain matrices used to approximate $\Lambda_{K_n, P, T_\kappa}$.
For a given $K_n$, let $L_n \in \mathbb N$ with $ L_n >K_n/c_{\kappa, 1}$, which will be taken later to hold for all large $n$; see Property $\bf H$ in Section~\ref{sect:uniform_bds}.
For a given knot sequence $T_\kappa \in \Tcal_{K_n}$ of $(K_n+1)$ knots on $[0, 1]$ with $T_\kappa=\{ 0=\kappa_0<\kappa_1<\cdots<\kappa_{K_n}=1\}$, define the matrix $\wt E_{T_\kappa, L_n} \in \mathbb R^{K_n \times L_n}$ as
\begin{equation} \label{eqn:def_wtE}
  \big[ \wt E_{T_\kappa, L_n} \big]_{j \ell} \, := \, \mathbf I_{[\kappa_{j-1}, \kappa_{j})}\left( \frac{\ell-1}{L_n} \right), \qquad \forall \ j=1, \ldots, K_n, \quad \ell=1, \ldots, L_n,
\end{equation}
where $\mathbf I_{[\kappa_{j-1}, \kappa_{j})}$ is the indicator function on the interval $[\kappa_{j-1}, \kappa_{j})$. For each $j=1, \ldots, K_n$, let $\ell_j$ be the cardinality of the index set $\{\ell \in \mathbb N \, | \, L_n \kappa_{j-1} +1 \le \ell < L_n \kappa_j+1\}$.
Hence, we have
\[
 \wt E_{T_\kappa, L_n} =
\begin{bmatrix}
\ \onebld^T_{\ell_1}\; & &  & \\
  &\; \onebld^T_{\ell_2}\; && \\
  &		& \ddots\;&\\
  &		&& \onebld^T_{\ell_{K_n}}\;\\
\end{bmatrix} \in \mathbb R^{K_n \times L_n}.
\]
Let $L_n':=L_n+m-1$, and for each $p=1, \ldots, m$, define
\[
\Gamma_p \, := \,
\begin{bmatrix}
I_{(m-p)} & 0 \\
        0  & L_n^{-1} \cdot I_{(L_n+p-1)}
\end{bmatrix} \in \mathbb R^{L'_n\times L'_n}, \quad \text{ and } \quad \wt S^{(p)}_{L_n} := \Gamma_p \cdot \wh S^{(L'_n)} \in \mathbb R^{L'_n\times L'_n},
\]
where $\wh S^{(r)}$ is defined in (\ref{eqn:def_Swh}). We then define the matrices $X_{p, T_\kappa, L_n} \in \mathbb R^{T_n \times L'_n}$ for the given $T_\kappa$ and $L_n$ inductively as:
\begin{equation} \label{eqn:def_Xmatrix}
  X_{1, T_\kappa, L_n}  :=  \begin{bmatrix} I_{m-1} & 0 \\ 0 & \wt E_{T_\kappa, L_n} \end{bmatrix}, \ \ \text{ and } \ \ X_{p, T_\kappa, L_n}  := \wh D^{(T_n)} \cdot \wh \Delta^{-1}_{p-1, T_\kappa} \cdot X_{p-1, T_\kappa, L_n} \cdot \wt S^{(p-1)}_{L_n}, \ p=2, \ldots, m.
\end{equation}
 Note that $X_{1, T_\kappa, L_n}$ is of full row rank for any $T_\kappa, L_n$, and so is $X_{p, T_\kappa, L_n}$ for each $p=2, \ldots, m$, since $\wh D^{(T_n)}$, $\wh \Delta^{-1}_{p-1, T_\kappa}$, $\wh S^{(L'_n)}$, and $\Gamma_{p-1}$ are all invertible.
Finally, define the matrix
\begin{equation} \label{eqn:wt_Lambda}
   \wt \Lambda_{T_\kappa, K_n, L_n} \, := \, \frac{K_n}{L_n} \cdot \big( X_{m, T_\kappa, L_n}\big)_{1:T_n, 1:L_n} \cdot \left[ \big( X_{m, T_\kappa, L_n}\big)_{1:T_n, 1:L_n} \right]^T \in \mathbb R^{T_n \times T_n}.
\end{equation}
It will be shown later (cf. Proposition~\ref{prop:error_bd_Lamda}) that $\wt \Lambda_{T_\kappa, K_n, L_n}$ approximates $\Lambda_{K_n, P, T_\kappa}$ for all large $n$ when $L_n$ is suitably chosen.

As discussed in Section~\ref{sect:proof_roadmap}, the proof of the uniform Lipschitz property boils down to certain uniform bounds in $\ell_\infty$-norm, including that for $\|\big( \Xi'_\alpha  F_\alpha \Lambda_{K_n,P, T_\kappa} \, F_\alpha^T \big)^{-1} \|_\infty$, where $\Xi'_\alpha := K^{-1}_n \Xi^{(m)}_{\alpha, T_\kappa}$.
Therefore it is essential to know about $F_\alpha \Lambda_{K_n,P, T_\kappa} \, F_\alpha^T$, which is approximated by $F_\alpha \wt \Lambda_{ T_\kappa, K_n, L_n} \, F_\alpha^T$.
In view of the definition of $\wt \Lambda_{T_\kappa, K_n, L_n}$, we see that the latter matrix product is closely related to $F^{(m)}_{\alpha, T_\kappa} X_{m, T_\kappa, L_n}$ for a given index set $\alpha$, a knot sequence $T_k$, and $L_n$. In what follows, we show certain important properties of $F^{(m)}_{\alpha, T_\kappa} X_{m, T_\kappa, L_n}$ to be used in the subsequent development.

\begin{lemma}\label{lem:matrix_FtimesX}
 Fix $m \in \mathbb N$. For any given $\alpha$, $T_\kappa$, and $L_n$, the following hold:
 \begin{itemize}
   \item [(1)] For each $p=2, \ldots, m$, $F^{(p)}_{\alpha, T_\kappa} \cdot X_{p, T_\kappa, L_n} = \wh D^{(q_\alpha)} \cdot \Xi^{(p-1)}_{\alpha, T_\kappa} \cdot F^{(p-1)}_{\alpha, T_\kappa} \cdot X_{p-1, T_\kappa, L_n} \cdot \wt S^{(p-1)}_{L_n}$;
   \item [(2)] For each $p=1, \ldots, m-1$, there exists a matrix $Z_{p, \alpha, T_\kappa, L_n} \in \mathbb R^{(q_\alpha -m +p)\times (L_n+p-1)}$ such that
       \[
          F^{(p)}_{\alpha, T_\kappa} \cdot X_{p, T_\kappa, L_n} = \begin{bmatrix} I_{m-p} & 0 \\
            0 & Z_{p, \alpha, T_\kappa, L_n} \end{bmatrix} \in \mathbb R^{q_\alpha \times L'_n}.
       \]
 \end{itemize}
\end{lemma}

\begin{proof}
  (1) It follows from the definitions of $F^{(p)}_{\alpha, T_\kappa}$ in (\ref{eqn:def_F}) and $X_{p, T_\kappa, L_n}$ in (\ref{eqn:def_Xmatrix})  respectively and $\wh S^{(T_n)} \cdot \wh D^{(T_n)}=I$  that for each $p=2, \ldots, m$,
  \begin{align*}
   F^{(p)}_{\alpha, T_\kappa} \cdot X_{p, T_\kappa, L_n} & \, = \, \left( \wh D^{(q_\alpha)} \cdot \Xi^{(p-1)}_{\alpha, T_\kappa} \cdot F^{(p-1)}_{\alpha, T_\kappa} \cdot \wh \Delta_{p-1, T_\kappa} \cdot \wh S^{(T_n)}\right) \cdot \left( \wh D^{(T_n)} \cdot \wh \Delta^{-1}_{p-1, T_\kappa} \cdot X_{p-1, T_\kappa, L_n} \cdot \wt S^{(p-1)}_{L_n}  \right) \\
   &\, = \, \wh D^{(q_\alpha)} \cdot \Xi^{(p-1)}_{\alpha, T_\kappa} \cdot F^{(p-1)}_{\alpha, T_\kappa} \cdot X_{p-1, T_\kappa, L_n} \cdot \wt S^{(p-1)}_{L_n}.
  \end{align*}

 (2) When $p=1$, it is easily seen that
     \[
         F^{(1)}_{\alpha, T_\kappa} \cdot X_{1, T_\kappa, L_n} = \begin{bmatrix} I_{m-1} & 0 \\ 0 & E_{\alpha, T_\kappa} \end{bmatrix} \, \begin{bmatrix} I_{m-1} & 0 \\ 0 & \wt E_{T_\kappa, L_n} \end{bmatrix} = \begin{bmatrix} I_{m-1} & 0 \\ 0 & Z_{1, \alpha, T_\kappa, L_n} \end{bmatrix},
     \]
     where $Z_{1, \alpha, T_\kappa, L_n} := E_{\alpha, T_\kappa} \cdot \wt E_{T_\kappa, L_n}$. Hence, statement (2) holds for $p=1$. Suppose statement (2) holds for $p=1, \ldots, p^{\,\prime}$, and consider $p^{\,\prime}+1$. In view of statement (1),  the definitions of $\Gamma_{p^{\, \prime}}$ and $\wt S^{(p')}_{L_n}$, (\ref{eqn:Xi_matrix}) for $\Xi^{(p^{\,\prime})}_{\alpha, T_\kappa}$, and the induction hypothesis, we have
     \begin{align*}
      \lefteqn{ F^{(p^{\,\prime}+1)}_{\alpha, T_\kappa} \cdot X_{p^{\,\prime}+1, T_\kappa, L_n} \, = \, \wh D^{(q_\alpha)} \cdot \Xi^{(p^{\,\prime})}_{\alpha, T_\kappa} \cdot F^{(p^{\,\prime})}_{\alpha, T_\kappa} \cdot X_{p^{\,\prime}, T_\kappa, L_n} \cdot \wt S^{(p^{\,\prime})}_{L_n} } \\
       & \, = \, \wh D^{(q_\alpha)} \cdot \Xi^{(p^{\,\prime})}_{\alpha, T_\kappa} \cdot F^{(p^{\,\prime})}_{\alpha, T_\kappa} \cdot X_{p^{\,\prime}, T_\kappa, L_n} \cdot \Gamma_{p^{\,\prime}} \cdot \wh S^{L'_n} \\
       & \, = \, \wh D^{(q_\alpha)} \cdot \begin{bmatrix} I_{m-p^{\,\prime}} & 0 \\ 0 & \Sigma_{ p^{\,\prime}, T_\kappa } \end{bmatrix} \cdot \begin{bmatrix} I_{m-p^{\,\prime}} & 0 \\ 0 & Z_{ p^{\,\prime}, \alpha, T_\kappa,L_n } \end{bmatrix} \cdot \begin{bmatrix} I_{m-p^{\,\prime}} & 0 \\ 0 & L^{-1}_n \cdot I_{L_n +p^{\,\prime}-1} \end{bmatrix} \cdot \wh S^{(L'_n)} \\
       & \, = \,  \wh D^{(q_\alpha)} \cdot  \begin{bmatrix} \wh S^{(m-p^{\,\prime})} & \mathbf 1_{(m-p^{\,\prime})\times (L_n+p^{\,\prime}-1)} \\ 0 & \star \end{bmatrix} \, = \, \begin{bmatrix} I_{(m-p^{\,\prime}-1)} & 0 \\ 0 & \star' \end{bmatrix},
     \end{align*}
     where $\star$ and $\star'$ are suitable submatrices, and the last two equalities follow from the structure of $\wh S^{(L'_n)}$ and $\wh D^{(q_\alpha)}$. Letting $Z_{p^{\,\prime}+1, \alpha, T_\kappa, L_n}:=\star'$, we arrive at the desired equality via induction.
\end{proof}

In what follows, we develop an inductive formula to compute $Z_{p, \alpha, T_\kappa, L_n}$. For notational simplicity, we use $Z_p$, $Y_p$, and $\tau_s$ in place of $Z_{p, \alpha, T_\kappa, L_n}$, $F^{(p)}_{\alpha, T_\kappa} \cdot X_{p, T_\kappa, L_n}$, and $\tau_{\alpha, T_\kappa, s}$ (cf. (\ref{eqn:tau_k})) respectively for fixed $\alpha$, $T_\kappa$, and $L_n$. For each $p=2, \ldots, m$, it follows from statement (1) of Lemma~\ref{lem:matrix_FtimesX} that for any $j=1, \ldots, |\ol\alpha|+p$ and $k=1, \ldots, L_n+p-1$,
\[
   \big[ Z_p \big]_{j, k} =  \left[ \wh D^{(q_\alpha)} \cdot \Xi^{(p-1)}_{\alpha, T_\kappa}\cdot Y_{p-1} \cdot \wt S^{(p-1)}_{L_n} \right]_{j+m-p, \, k+m-p}  = \left[ \wh D^{(q_\alpha)} \Xi^{(p-1)}_{\alpha, T_\kappa} \right]_{(j+m-p) \bullet} \cdot  \left[Y_{p-1}  \wt S^{(p-1)}_{L_n} \right]_{\bullet (k+m-p)}.
\]
In light of (\ref{eqn:Xi_matrix}), we have, for any $j=1,\ldots, |\ol\alpha|+p$,
\[
\left(\wh D^{(q_\alpha)}  \cdot \Xi^{(p-1)}_{\alpha, T_\kappa} \right)_{(j+m-p) \bullet}  =
\begin{cases}
\left(\underbrace{0,\ldots,0}_{m-p}, \, 1,-\frac{p-1}{\tau_{1}-\tau_{2-p}}, \, \underbrace{0,\ldots,0}_{|\ol \alpha|+p-2}\right)   & \text{ if }\ j = 1 \\
\left(\underbrace{0,\ldots,0}_{j+m-p-1},\, \frac{p-1}{\tau_{j-1}-\tau_{j-p}},-\frac{p-1}{\tau_{j}  - \tau_{j-p+1} }, \, \underbrace{0,\ldots,0}_{|\ol\alpha|-j+p-1}\right)   & \text{ if }\  1 < j < |\ol\alpha|+p\\
\left(\underbrace{0,\ldots,0}_{|\ol\alpha|+m-1}, \, \frac{p-1}{\tau_{|\ol\alpha|+p-1}-\tau_{|\ol\alpha|}}\right)   & \text{ if } \ j =|\ol\alpha|+p.
\end{cases}
\]
Moreover, by the virtue of Lemma~\ref{lem:matrix_FtimesX}, we have
\begin{align*}
   Y_{p-1} \cdot \wt S^{(p-1)}_{L_n} & \, = \,  Y_{p-1} \cdot \Gamma_{p-1} \cdot \wh S^{(L'_n)} = \begin{bmatrix} I_{m-p+1} & 0 \\ 0 & \frac{Z_{p-1}}{L_n} \end{bmatrix} \cdot \wh S^{(L'_n)}  \\
   & \, = \,
   \begin{bmatrix} I_{m-p} & 0 & 0 \\ 0 & 1 & 0 \\ 0 & 0 & \frac{Z_{p-1}}{L_n} \end{bmatrix} \cdot \wh S^{(L'_n)}
    \, = \, \begin{bmatrix} \vspace{0.2cm} \wh S^{(m-p)} & \mathbf 1_{m-p} & \mathbf 1_{(m-p)\times (L_n+p-2)} \\ \vspace{0.2cm} 0 & 1 & \mathbf 1^T_{L_n+p-2} \\ 0 & 0 & \frac{Z_{p-1} \cdot \wh S^{(L_n+p-2)}}{L_n} \end{bmatrix}.
\end{align*}
This shows in particular that for each $j=2, \ldots, |\ol\alpha|+p$ and $k=2, \ldots, L_n+p-1$,
\[
 \left[ Y_{p-1} \cdot \wt S^{(p-1)}_{L_n} \right]_{j+m-p, \, k+m-p} \, = \, L^{-1}_n \left[ Z_{p-1} \cdot \wh S^{(L_n+p-2)} \right]_{j-1, \, k-1} \, = \,  L^{-1}_n \cdot \sum^{k-1}_{\ell=1} \big[ Z_{p-1} \big]_{j-1,\, \ell}.
\]
Combining the above results, we have,  for any $p \ge 2$, $j=1, \ldots, |\ol\alpha|+p$, and $k=1, \ldots, L_n+p-1$,
\begin{align}
\lefteqn{ \big[ Z_p \big]_{j, \, k} \, = \, \sum^{q_\alpha}_{s=1} \Big[ \wh D^{(q_\alpha)} \cdot \Xi^{(p-1)}_{\alpha, T_\kappa} \Big]_{j+m-p, \, s} \cdot \Big[ Y_{p-1}\cdot \wt S^{(p-1)}_{L_n} \Big]_{s, \, k+m-p}  }  \label{eqn:Zp_entry} \\
&
= \begin{cases}
%\vspace{0.2cm}
\delta_{j, 1}  & \text{ if } \ k =1\\
%\vspace{0.25cm}
\displaystyle 1-\frac{p-1}{L_n(\tau_1-\tau_{2-p})}  \sum_{\ell = 1}^{k-1}  \big [ Z_{p-1} \big]_{1, \ell}  & \text{ if } \ j=1, \text{ and } k>1 \\
\vspace{0.15cm}
\displaystyle \frac{p-1}{L_n(\tau_{j-1}-\tau_{j-p})}  \sum_{\ell =1}^{k-1}  \big[ Z_{p-1} \big]_{j-1,\ell} - \frac{p-1}{L_n(\tau_{j}-\tau_{j-p+1})}
\sum_{\ell = 1}^{k-1}  \big[ Z_{p-1} \big]_{j, \ell}   & \text{ if } \  1 < j< |\ol\alpha|+p, \text{ and } k> 1\\
%\vspace{0.05cm}
\displaystyle \frac{p-1}{L_n(\tau_{|\ol\alpha|+p-1}-\tau_{|\ol\alpha|})}   \sum_{\ell = 1}^{k-1}  \big[ Z_{p-1} \big]_{|\ol \alpha|+p-1, \ell} &  \text{ if } \ j = |\ol \alpha|+p, \text{ and } k >1.
\end{cases} \nonumber
\end{align}
The above results for $Z_{p, \alpha, T_\kappa, L_n}$ will be exploited to establish uniform bounds for the uniform Lipschitz property shown in the next subsection (cf. Proposition~\ref{prop:Z_approx_Bspline}).

%------------------------------------------------------------------------------
%
\subsection{Preliminary uniform bounds} \label{sect:uniform_bds}

This subsection establishes uniform bounds and uniform approximation error bounds of several of the constructed matrices.
These results lay a solid foundation for the proof of Theorem~\ref{thm:uniform_Lip}.

The first result  of this subsection (cf. Proposition~\ref{prop:Z_approx_Bspline}) shows that
the entries of each row of $Z_{p, \alpha, T_\kappa, L_n}$ introduced in Lemma~\ref{lem:matrix_FtimesX} are sufficiently close to the corresponding values of a B-spline defined on a certain knot sequence for large $L_n$. Hence, each row of $Z_{p, \alpha, T_\kappa, L_n}$ can be approximated by an appropriate B-spline; more importantly, the approximation error is shown to be uniformly bounded, regardless of $\alpha$ and $T_\kappa$. This result forms a cornerstone for many critical uniform bounds in the proof of the uniform Lipschitz property.

Recall in Section~\ref{sect:construction_F} that for a given $K_n\in \mathbb N$, a knot sequence $T_\kappa \in \mathcal T_{K_n}$,  and the index set $\alpha \subseteq \{1, \ldots, K_n-1\}$  defined in (\ref{eqn:index_alpha}), the complement of $\alpha $ is given by $\ol{\alpha} = \{i_1,\ldots,i_{|\ol{\alpha}|}\}$ with $1 \leq i_1 < \cdots < i_{|\ol{\alpha}|} \leq K_n-1$.
For the given $\ol\alpha$ and $T_\kappa$, define the knot sequence with the usual extension $\kappa_{i_s}=0$ for $s<0$ and $\kappa_{i_s}=1$ for $s > |\ol\alpha|+1$:
\begin{equation} \label{eqn:V_knot_seq}
    V_{\alpha, T_\kappa} \, := \, \Big\{ 0=\kappa_{i_0}<\kappa_{i_1}<\kappa_{i_2}< \cdots < \kappa_{i_{|\ol\alpha|}}< \kappa_{ i_{|\ol\alpha|+1}}=1 \Big\}.
\end{equation}
Let $\big\{ B^{V_{\alpha, T_\kappa}}_{p, j} \big\}^{|\ol\alpha|+p}_{j=1}$ be the B-splines of order $p$ on $[0, 1]$ defined by $V_{\alpha, T_\kappa}$. With this notation, we present the following proposition.

%----------------------------------------------------------------------------
%
\begin{proposition}\label{prop:Z_approx_Bspline}
Given $K_n, L_n\in \mathbb N$, let $M_n \in \mathbb N$ with $M_n \ge m \cdot K_n/{c_{\kappa, 1}}$.
Then for each $p=1, \ldots, m$,
any $T_\kappa \in \Tcal_{K_n}$, any index set $\alpha$, %defined in (\ref{eqn:index_alpha}),
any $j=1, \ldots, |\ol\alpha|+p$, and any $k=1,\ldots, L_n$,
\[
\left| \big[ Z_{p, \alpha, T_\kappa, L_n} \big]_{j, \, k} - B^{V_{\alpha, T_\kappa}}_{p, j}\left( \frac{k-1}{L_n} \right) \right| \, \leq \, 6\cdot \big(2^{p-1}-1 \big) \cdot \frac{\big(M_n\big)^{p-1}}{L_n}, \qquad \forall \ n \in \mathbb N.
\]
\end{proposition}

%
%----------------------------------------------------------------------------
%
\begin{proof}
We prove this result by induction on $p$. Given arbitrary $K_n, L_n \in \mathbb N$, $T_\kappa \in \Tcal_{K_n}$, and $\alpha$ defined in (\ref{eqn:index_alpha}),
%
 %% and $(L_n)$ of Property $\bf H$,
%
we use $Z_p$ to denote $Z_{p, \alpha, T_\kappa, L_n}$ to simplify notation.
Consider $p = 1$ first. It follows from the proof of statement (2) of Lemma~\ref{lem:matrix_FtimesX} that $Z_{1} = E_{\alpha, T_\kappa} \cdot \wt E_{T_\kappa, L_n}$. In view of the definitions of $E_{\alpha, T_\kappa}$ and $\wt E_{T_\kappa, L_n}$ in (\ref{eqn:E_F1_matrices}) and (\ref{eqn:def_wtE}) respectively, we have, for any $j=1, \ldots, |\ol\alpha|+1$ and $k=1, \ldots, L_n$,
\begin{align}
  \big [ Z_1 \big]_{j, \, k} & \, = \, \sum^{K_n}_{\ell=1} \big[ E_{\alpha, T_\kappa} \big]_{j, \,\ell} \cdot \big[ \wt E_{ T_\kappa, L_n} \big]_{\ell, \, k} \, = \, \sum^{K_n}_{\ell=1} \mathbf I_{[\kappa_{i_{j-1}}, \kappa_{i_j})} \left( \kappa_{\ell-1} \right) \cdot \mathbf I_{[\kappa_{\ell-1}, \kappa_\ell)} \left( \frac{k-1}{L_n} \right) \notag \\
   & \, = \, \mathbf I_{[\kappa_{i_{j-1}}, \kappa_{i_j})}\left( \frac{k-1}{L_n} \right) \, = \, B^{V_{\alpha, T_\kappa}}_{1, j}\left( \frac{k-1}{L_n} \right). \label{eqn:Z1_equal_B1}
\end{align}
This shows that the proposition holds for $p=1$.

Suppose that the proposition holds for $p=1, \ldots, \ppm$ with $1\le \ppm \le m-1$. Consider $p=\ppm+1$  now.  Define
\begin{equation}\label{eqn:theta_p}
  \theta_{\ppm, \alpha, T_\kappa, L_n} \, := \, 2 M_n \, \max_{j=1, \ldots, |\ol\alpha|+\ppm} \, \left( \, \max_{k=2, \ldots, L_n} \left|\sum_{\ell=1}^{k-1} \frac{1}{L_n} \big [Z_{\ppm} \big]_{j, \, \ell} - \int_{0}^{\frac{k-1}{L_n}} B^{V_{\alpha, T_\kappa}}_{\ppm, j}(x)\,dx \right| \, \right).
\end{equation}
We show below that $\big| [ Z_{\ppm+1} ]_{j, k} - B^{V_{\alpha, T_\kappa}}_{\ppm+1, j}(\frac{k-1}{L_n}) \big | \leq \theta_{\ppm, \alpha, T_\kappa, L_n}$ for any $j=1, \ldots, |\ol\alpha|+\ppm+1$ and $k=1, \ldots, L_n$. To show this, consider the following cases via the entry formula of $Z_{\ppm+1}$ in (\ref{eqn:Zp_entry}):
\begin{itemize}
  \item [(i)]
 $k=1$ and $j=1, \ldots,  |\ol\alpha|+\ppm+1$. In view of (\ref{eqn:Zp_entry}), we have $[ Z_{\ppm+1} ]_{j, 1} = \delta_{j, 1} = B^{V_{\alpha, T_\kappa}}_{\ppm+1, j}(0)$ for each $j$. This implies that $\big| [ Z_{\ppm+1} ]_{j, 1} - B^{V_{\alpha, T_\kappa}}_{\ppm+1, j}(0) \big| =0 \le \theta_{\ppm, \alpha, T_\kappa, L_n}$ for each $j$.

\item [(ii)]
$j=1$ and $k=2, \ldots, L_n$. It follows from (\ref{eqn:B_spline_derivative}) that
\[
    B^{V_{\alpha, T_\kappa}}_{\ppm+1, 1}(x) =  B^{V_{\alpha, T_\kappa}}_{\ppm+1, 1}(0) - \frac{\ppm}{\kappa_{i_1}- \kappa_{i_{1-\ppm}} } \int^x_0 B^{V_{\alpha, T_\kappa}}_{\ppm, 1}(t) dt
    = 1 - \frac{\ppm}{\kappa_{i_1}- \kappa_{i_{1-\ppm}} } \int^x_0 B^{V_{\alpha, T_\kappa}}_{\ppm, 1}(t) dt.
\]
Since each $\tau_s$ in (\ref{eqn:Zp_entry}) is $\tau_{\alpha, T_\kappa, s}$ defined in (\ref{eqn:tau_k}), we have $\tau_1=\kappa_{i_1}$ and $\tau_{2-\ppm-1}=\kappa_{i_{1-\ppm}}$. Hence
\[
  \left | \big[ Z_{\ppm+1} \big]_{1, k} - B^{V_{\alpha, T_\kappa}}_{\ppm+1, 1}\Big(\frac{k-1}{L_n} \Big) \right | =
  \frac{\ppm}{\kappa_{i_1}- \kappa_{i_{1-\ppm} } } \left| \sum^{k-1}_{\ell=1} \frac{1}{L_n} \big[ Z_{\ppm} \big]_{1, \ell} -   \int^{\frac{k-1}{L_n}}_0 B^{V_{\alpha, T_\kappa}}_{\ppm, 1}(t) dt \right| \, \le \, \theta_{\ppm, \alpha, T_\kappa, L_n},
\]
where the last inequality is, via (\ref{eqn:knot_derivative_bd}), due to  $\frac{\ppm}{\kappa_{i_s}- \kappa_{i_{s-\ppm}} } \le \ppm\cdot K_n/c_{\kappa, 1} \le M_n$ for any $s$.
\item [(iii)]
 $j=2, \ldots, |\ol\alpha|+\ppm$ and $k=2, \ldots, L_n$. It follows from the integral form of (\ref{eqn:B_spline_derivative}) that
\[
  B^{V_{\alpha, T_\kappa}}_{\ppm+1, j}(x) \,  = \, \frac{\ppm}{\kappa_{i_{j-1}}- \kappa_{i_{j-\ppm-1} } } \int^x_0 B^{V_{\alpha, T_\kappa}}_{\ppm, j-1}(t) dt
   - \frac{\ppm}{\kappa_{i_j}- \kappa_{i_{j-\ppm} } } \int^x_0 B^{V_{\alpha, T_\kappa}}_{\ppm, j}(t) dt.
\]
 Using this equation, (\ref{eqn:Zp_entry}) and (\ref{eqn:theta_p}),
 and a similar argument as in Case (ii), we have
\begin{align*}
  \left | \big[ Z_{\ppm+1} \big]_{j, k} - B^{V_{\alpha, T_\kappa}}_{\ppm+1, j}\Big(\frac{k-1}{L_n} \Big) \right | & \, \le \,
   \frac{\ppm}{\kappa_{i_{j-1}}- \kappa_{i_{j-\ppm-1}} } \left| \sum^{k-1}_{\ell=1} \frac{1}{L_n} \big[ Z_{\ppm} \big]_{j-1, \ell} -   \int^{\frac{k-1}{L_n}}_0 B^{V_{\alpha, T_\kappa}}_{\ppm, j-1}(t) dt \right| \\
 & \qquad  + \,
  \frac{\ppm}{\kappa_{i_j}- \kappa_{i_{j-\ppm}} } \left| \sum^{k-1}_{\ell=1} \frac{1}{L_n} \big[ Z_{\ppm} \big]_{j, \ell} -   \int^{\frac{k-1}{L_n}}_0 B^{V_{\alpha, T_\kappa}}_{\ppm, j}(t) dt \right| \\
  & \, \le \, \theta_{\ppm, \alpha, T_\kappa, L_n}.
\end{align*}
\item [(iv)]
 $j=|\ol\alpha|+\ppm+1$ and $k=2, \ldots, L_n$. Similarly, we have from (\ref{eqn:B_spline_derivative}) that
\[
    B^{V_{\alpha, T_\kappa}}_{\ppm+1, |\ol\alpha|+\ppm+1}(x) \, = \, \frac{\ppm}{\kappa_{i_{|\ol\alpha|+\ppm}}- \kappa_{i_{|\ol\alpha|}} } \int^x_0 B^{V_{\alpha, T_\kappa}}_{\ppm, |\ol\alpha|+\ppm}(t) dt.
\]
This, along with (\ref{eqn:Zp_entry}) and a similar argument as in Case (ii), leads to
\begin{align*}
\lefteqn {\left | \big[ Z_{\ppm+1} \big]_{|\ol\alpha|+\ppm+1, k} - B^{V_{\alpha, T_\kappa}}_{\ppm+1, |\ol\alpha|+\ppm+1}\Big(\frac{k-1}{L_n} \Big) \right | } \\
 & \, = \
  \frac{\ppm}{\kappa_{i_{|\ol\alpha|+\ppm}}- \kappa_{i_{|\ol\alpha|}} }
  \left| \sum^{k-1}_{\ell=1} \frac{1}{L_n} \big[ Z_{\ppm} \big]_{|\ol\alpha|+\ppm, \ell} -   \int^{\frac{k-1}{L_n}}_0 B^{V_{\alpha, T_\kappa}}_{\ppm, |\ol\alpha|+\ppm}(t) dt \right|
   \ \le \ \theta_{\ppm, \alpha, T_\kappa, L_n}. \qquad \qquad \qquad
\end{align*}
\end{itemize}
This shows that $\big| [ Z_{\ppm+1} ]_{j, k} - B^{V_{\alpha, T_\kappa}}_{\ppm+1, j}(\frac{k-1}{L_n}) \big |\le \theta_{\ppm, \alpha, T_\kappa, L_n}$ for any $j=1, \ldots, |\ol\alpha|+\ppm+1$ and $k=1, \ldots, L_n$.

Finally, we show that the upper bound $\theta_{\ppm, \alpha, T_\kappa, L_n}$ attains the specified uniform bound, regardless of $\alpha$ and $T_\kappa$.
 When $\ppm =1$, in light of (\ref{eqn:Z1_equal_B1}) and $B^{V_{\alpha, T_\kappa}}_{\ppm, j}(x)=\mathbf I_{[\kappa_{i_{j-1}}, \kappa_{i_j} )}(x)$ on $[0, 1)$, we derive via straightforward computation that for any $j=1, \ldots, |\ol\alpha|+\ppm$ and each $k=2, \ldots, L_n$,
 \begin{equation} \label{eqn:bound_for_p=1}
    \left|\sum_{\ell=1}^{k-1} \frac{1}{L_n} \big [Z_{\ppm} \big]_{j, \, \ell} - \int_{0}^{\frac{k-1}{L_n}} B^{V_{\alpha, T_\kappa}}_{\ppm, j}(x)\,dx \right| \, \le \, \frac{1}{L_n}.
 \end{equation}
 This implies that $\theta_{\ppm, \alpha, T_\kappa, L_n} \le 2 M_n/L_n \le 6 \cdot (2^{\ppm}-1) \cdot (M_n)^{\ppm}/L_n$. In what follows, consider $2\le \ppm \le m-1$.
 Letting $C_{\ppm}:=6\cdot(2^{\ppm-1}-1)$, the induction hypothesis states that  $\big| [ Z_{\ppm} ]_{j, k} - B^{V_{\alpha, T_\kappa}}_{\ppm, j}(\frac{k-1}{L_n}) \big | \le C_{\ppm} \cdot (M_n)^{\ppm-1}/L_n$ for any $j=1, \ldots, |\ol\alpha|+\ppm$ and $k=1, \ldots, L_n$. Moreover, for each $j$, the B-spline  $B^{V_{\alpha, T_\kappa}}_{\ppm, j}$ is continuous on $[0, 1]$ and is differentiable  except at (at most) finitely many points in $[0, 1]$.
 By the derivative formula (\ref{eqn:B_spline_derivative}), we have, for any $x\in [0, 1]$ where the derivative exists,
\begin{eqnarray}
  \left| \Big( B^{V_{\alpha, T_\kappa}}_{\ppm, j}(x) \Big)' \right | &  =  & \left|  \frac{\ppm-1}{\kappa_{i_{j-1}}- \kappa_{i_{j-\ppm}} } B^{V_{\alpha, T_\kappa}}_{\ppm-1, j-1}(x)
   - \frac{\ppm-1}{\kappa_{i_j}- \kappa_{i_{j-\ppm+1}} } B^{V_{\alpha, T_\kappa}}_{\ppm-1, j}(x) \right | \, \le \, \frac{2 (\ppm-1)K_n}{c_{\kappa, 1}},  \label{eqn:Bspline_dev_bd}
\end{eqnarray}
where we use  the upper bound of B-splines and the fact that $\frac{\ppm-1}{\kappa_{i_s}- \kappa_{i_{s-\ppm+1}} } \le (\ppm-1) K_n/c_{\kappa, 1}$ for any $s$.
 Due to the property of $M_n$, we have $(\ppm-1)K_n/c_{\kappa, 1}\le M_n$. This further implies via $\ppm \ge 2$ that
\begin{equation}
  \left| \Big( B^{V_{\alpha, T_\kappa}}_{\ppm, j}(x) \Big)' \right | \, \le \,  \frac{2(\ppm-1)K_n}{c_{\kappa, 1}} \, \le \, 2 M_n \, \le \, 2 ( M_n )^{\ppm-1}. \label{eqn:Bspline_dev_bd_2}
\end{equation}
 For each fixed $j=1, \ldots, |\ol\alpha|+\ppm+1$, we apply Lemma~\ref{lem:funct_bound} with $\wt n:=L_n$, $a:=0$, $b:=1$, $s_k:=k/L_n$, $\varrho:=1/L_n$, $f(x) := B^{V_{\alpha, T_\kappa}}_{\ppm, j}(x)$, $v=(v_k):=\big( [Z_{\ppm}]_{j, k} \big)$, $\mu_1 := C_{\ppm}\frac{( M_n )^{\ppm-1}}{L_n}$, and $\mu_2 := 2 ( M_n )^{\ppm-1}$ to obtain that for each $k=2, \ldots, L_n$,
\[
   \left|\sum_{\ell=1}^{k-1} \frac{1}{L_n} \big [Z_{\ppm} \big]_{j, \, \ell} - \int_{0}^{\frac{k-1}{L_n}} B^{V_{\alpha, T_\kappa}}_{\ppm, j}(x)\,dx \right| \, \le \,  \frac{C_{\ppm} \cdot ( M_n )^{\ppm-1}}{L_n} + \frac{3 ( M_n )^{\ppm-1}}{L_n} \, = \, \big(C_{\ppm} + 3 \big) \frac{( M_n )^{\ppm-1}}{L_n}.
\]
Note that the above upper bound is independent of $\alpha, T_\kappa, j$ and $k$.
By using this result, (\ref{eqn:theta_p}), and $C_{\ppm}=6\cdot(2^{\ppm-1}-1)$, we deduce the following uniform bound for $\theta_{\ppm, \alpha, T_\kappa, L_n}$ independent of $\alpha$ and $T_\kappa$:
\[
   \theta_{\ppm, \alpha, T_\kappa, L_n} \, \le \, 2 M_n \cdot \big(C_{\ppm} + 3 \big) \frac{( M_n )^{\ppm-1}}{L_n} \, = \, 6 \cdot (2^{\ppm}-1) \cdot \frac{( M_n )^{\ppm}}{L_n}.
\]
Therefore, the proposition holds by the induction principle.
\end{proof}

%\gap

The uniform error bound established in Proposition~\ref{prop:Z_approx_Bspline} yields the following important result for the matrix $F^{(m)}_{\alpha, T_\kappa}$ constructed  in Section~\ref{sect:construction_F}.

\begin{proposition} \label{prop:F_times_B_T=B_alpha}
  For any given $m, K_n \in \mathbb N$, any knot sequence $T_\kappa \in \Tcal_{K_n}$, and any index set $\alpha$ and its associated knot sequence $V_{\alpha, T_\kappa}$ defined in (\ref{eqn:V_knot_seq}), the B-splines $\{B_{m,\ell}^{V_{\alpha,T_\kappa}}\}_{\ell=1}^{q_\alpha}$ and $\{B_{m,j}^{T_\kappa}\}_{j=1}^{T_n}$ satisfy for each $\ell=1, \ldots, q_\alpha$,
 \begin{equation}\label{eqn:sum_of_Bsplines}
     \sum_{j=1}^{T_n} \left [F_{\alpha,T_\kappa}^{(m)} \right]_{\ell,j} \, B_{m,j}^{T_\kappa}(x)\,  = \, B_{m,\ell}^{V_{\alpha,T_\kappa}}(x), \qquad \quad \forall \ x \in [0,1].
\end{equation}
\end{proposition}

\begin{proof}
Consider $m = 1$ first. Recall that $\kappa_{i_0}=0$, $\kappa_{i_{|\ol \alpha|+1}}= \kappa_{i_{q_\alpha}}=1$, and $F^{(1)}_{\alpha, T_\kappa}= E_{\alpha, T_\kappa}$ (cf. (\ref{eqn:E_F1_matrices})). It follows from (\ref{eqn:tau_k}) and~(\ref{eqn:entry_F1}) that for each $\ell=1, \ldots, q_\alpha$ and any $x \in [0, 1]$,
\begin{align*}
\sum_{j=1}^{K_n} \left[F_{\alpha,T_\kappa}^{(1)} \right]_{\ell,j}\,B_{1,j}^{T_\kappa}(x) & =
\sum_{j=1}^{K_n-1} \mathbf I_{[\kappa_{i_{\ell-1}},\,\kappa_{i_\ell})}(\kappa_{j-1}) \cdot \mathbf I_{[\kappa_{j-1},\,\kappa_j)}(x)
+ \mathbf I_{[\kappa_{i_{\ell-1}},\,\kappa_{i_\ell})}(\kappa_{K_n-1}) \cdot \mathbf I_{[\kappa_{K_n-1},\,\kappa_{K_n}]}(x)\\
& =
\begin{cases}
 \, \mathbf I_{[\kappa_{i_{\ell-1}},\kappa_{i_\ell})}(x) & \mbox{if } \ \ell \in \{ 1,\dots,q_\alpha-1\}\\
\, \mathbf I_{[\kappa_{i_{q_\alpha-1}},\kappa_{i_{q_\alpha}}]}(x) & \mbox{if } \ \ell =q_\alpha
\end{cases}\\
 & \, = \, B_{1,\ell}^{V_{\alpha,T_\kappa}}(x).
\end{align*}

 In what follows, we consider $m \ge 2$. Recall that when $\alpha$ is the empty set, $q_\alpha = T_n=K_n+m-1$, $F_{\alpha,T_\kappa}^{(m)} = I_{T_n}$,  $B_{m,\ell}^{V_{\emptyset,T_\kappa}}= B_{m,\ell}^{T_\kappa}$ for each $\ell$, and $Z_{m,\emptyset,T_\kappa,L_n} = X_{m,T_\kappa,L_n}\in \mathbb R^{q_\alpha \times (L_n+m-1)}$ (cf. Lemma~\ref{lem:matrix_FtimesX}). Motivated by these observations and $ F^{(m)}_{\alpha, T_\kappa} \cdot X_{m,T_\kappa,L_n}= Z_{m,\alpha,T_\kappa,L_n}$ for any $\alpha$ and $T_\kappa$ (cf. Lemma~\ref{lem:matrix_FtimesX}), a key idea for the subsequent proof is to approximate $B_{m, j}^{T_\kappa}$ and $B_{m,\ell}^{V_{\alpha, T_\kappa}}$ by $X_{m,T_\kappa,L_n}$ and $Z_{m, \alpha, T_\kappa,L_n}$ respectively, where approximation errors can be made arbitrarily small by choosing a sufficiently large $L_n$ in view of Proposition~\ref{prop:Z_approx_Bspline}.

Fix $m, K_n, \alpha, T_\kappa$. Let $M_n :=\lceil m \cdot K_n/c_{\kappa, 1} \rceil$. Hence both $M_n$ and $T_n:=K_n+m-1$ are fixed natural numbers. Since we shall choose a sequence of sufficiently large $L_n$ independent of the above-mentioned fixed numbers, we write $L_n$ as $L_s$ below to avoid notational confusion.
In order to apply Proposition~\ref{prop:Z_approx_Bspline}, we first consider rational $x$ in $[0, 1)$. Let $x_* \in [0, 1)$ be an arbitrary but fixed rational number, and let $(L_s)$ be an increasing sequence of natural numbers (depending on $x_*$) such that $L_s \rightarrow \infty$ as $s\rightarrow \infty$ and for each $s$, $x_* = \frac{i^*_s-1}{L_s}$ for some $i^*_s \in \{ 1,\ldots,L_s\}$. (Here $i^*_s$ depends on $x_*$ and $L_s$ only.)
In light of the observations $B_{m,\ell}^{V_{\emptyset,T_\kappa}}= B_{m,\ell}^{T_\kappa}$ and $Z_{m,\emptyset,T_\kappa,L_s} = X_{m,T_\kappa,L_s}$, it follows from Proposition~\ref{prop:Z_approx_Bspline}
that for each $\ell=1, \ldots, q_\alpha$ and each $s$, %% and any $T_\kappa$ and $\alpha$,
\begin{equation}
  \left| \big[ X_{m,T_\kappa,L_s} \big]_{\ell, \, i^*_s} -B_{m,\ell}^{T_\kappa}(x_*) \, \right| \, = \,
   \left| \big[ Z_{m, \emptyset, T_\kappa,L_s} \big]_{\ell, \, i^*_s} -B_{m,\ell}^{V_{\emptyset, T_\kappa}} \Big(\frac{i^*_s-1}{L_s} \Big) \right| %%\notag
 \, \le \,    \frac{ 6 \, (2^m-1) \, M_n^{m-1}}{L_s}.
 \label{eqn:bound_X_B}
\end{equation}
By using  $Z_{m,\alpha,T_\kappa,L_s} = F^{(m)}_{\alpha, T_\kappa} \cdot X_{m,T_\kappa,L_s}$ (cf. Lemma~\ref{lem:matrix_FtimesX}), we thus have,  for each $\ell=1, \ldots, q_\alpha$,
%%and any $\alpha$ and $T_\kappa$,
\begin{align*}
&\left|\sum_{j=1}^{T_n}  \big[ F_{\alpha,T_\kappa}^{(m)} \big]_{\ell,j} \, B_{m,j}^{T_\kappa}(x_*) - B_{m,\ell}^{V_\alpha,T_\kappa}(x_*)\right| \\
& \, \le \, \sum_{j=1}^{T_n}\left| \big[ F_{\alpha,T_\kappa}^{(m)} \big]_{\ell,j} \cdot \left( B_{m,j}^{T_\kappa}(x_*) - \big[ X_{m,T_\kappa,L_s} \big]_{j,i^*_s}\right) \right| + \left|\sum_{j=1}^{T_n} \big[ F_{\alpha,T_\kappa}^{(m)} \big]_{\ell,j} \cdot \big[X_{m,T_\kappa,L_s} \big]_{j,i^*_s}- B_{m,\ell}^{V_{\alpha,T_\kappa}}(x_*)\right| \\
& \, \le \,  \sum_{j=1}^{T_n}\left| \big[F_{\alpha,T_\kappa}^{(m)} \big]_{\ell,j}\right|  \cdot    \left| B_{m,j}^{T_\kappa}(x_*) - \big [X_{m,T_\kappa,L_s} \big]_{j,i^*_s}\right|  + \left| \big[Z_{m,\alpha,T_\kappa,L_s} \big]_{\ell,i^*_s}  - B_{m,\ell}^{V_{\alpha,T_\kappa}}\Big(\frac{i^*_s-1}{L_s} \Big)\right| \\
& \le T_n \cdot \left[ \left(\frac{2m}{c_{\kappa,1}}\cdot \max\Big(1, \frac{c_{\kappa, 2}}{K_n} \Big) \cdot  T_n \right)^{m-1} \cdot \big( K_n \big)^{m} \right] \cdot  \frac{6 (2^{m-1}-1) M^{m-1}_n }{L_s}  +
 \frac{6\, (2^{m-1}-1)M_n^{m-1}}{L_s},
\end{align*}
where the last inequality follows from the bounds given in statement (2) of Lemma~\ref{lem:F_m_bound}, (\ref{eqn:bound_X_B}), and Proposition~\ref{prop:Z_approx_Bspline} respectively.
 By the virtue of $L^{-1}_s \rightarrow 0$ as $s\rightarrow \infty$, we have $\sum_{j=1}^{T_n}  \big[ F_{\alpha,T_\kappa}^{(m)} \big]_{\ell,j} \, B_{m,j}^{T_\kappa}(x_*) =  B_{m,\ell}^{V_\alpha,T_\kappa}(x_*)$. This shows that (\ref{eqn:sum_of_Bsplines}) holds for all rational $x\in [0, 1)$. Since  $B_{m,j}^{T_\kappa}$ and $B_{m,\ell}^{V_{\alpha, T_\kappa}}$ are continuous on $[0, 1]$ for any $j$, $\ell$, $T_\kappa$, and $\alpha$ when $m \ge 2$, we conclude via the density of rational numbers in $[0, 1)$ that (\ref{eqn:sum_of_Bsplines}) holds for all $x \in [0,1)$. Finally, the continuity of $B_{m,j}^{T_\kappa}$ and $B_{m,\ell}^{V_{\alpha, T_\kappa}}$ also shows that (\ref{eqn:sum_of_Bsplines}) holds at $x=1$.
\end{proof}

By Proposition~\ref{prop:F_times_B_T=B_alpha}, we derive tight uniform bounds for $\big\|(F_{\alpha,T_\kappa}^{(m)})^T \big\|_\infty$ and $\big\|K_n^{-1} \,  \Xi_{\alpha,T_\kappa}^{(m)} \, F_{\alpha,T_\kappa}^{(m)} \big\|_\infty$ in the next corollary; these bounds are crucial for the proof of Theorem~\ref{thm:uniform_Lip} (cf. Section~\ref{sect:proof_unif_Lip_final}).
We introduce more notation. Let $\mathbf e_\ell$ be the $\ell$th standard basis (column) vector in the Euclidean space, i.e., $\big[ \mathbf e_\ell \big]_k=\delta_{\ell, k}$. Moreover, for a given vector $v=(v_1, \ldots, v_k)\in \mathbb R^k$, the number of sign changes of $v$ is defined as the largest integer $r_v \in \mathbb Z_+$ such that for some $1 \leq j_1 < \dots < j_{r_v} \leq k$, $v_{j_i} \cdot v_{j_{i+1}} < 0$ for each $i = 1, \ldots,r_v$ \cite[page~138]{deBoor_book01}. Clearly, $\mathbf e_\ell$ has zero sign changes for each $\ell$.

\begin{corollary} \label{coro:F_T_XiF_bounds}
 For any $m\in \mathbb N$, any knot sequence $T_\kappa \in \Tcal_{K_n}$, and any index set $\alpha$ defined in (\ref{eqn:index_alpha}), the following hold:
  \begin{itemize}
   \item[(1)] $F_{\alpha,T_\kappa}^{(m)} $ is a nonnegative matrix, $\big\|(F_{\alpha,T_\kappa}^{(m)})^T \big\|_\infty = 1$, and
   \item[(2)]  $\big\|K_n^{-1} \,  \Xi_{\alpha,T_\kappa}^{(m)} \, F_{\alpha,T_\kappa}^{(m)} \big\|_\infty \leq  \frac{m}{c_{\kappa,1}}$.
\end{itemize}
\end{corollary}

\begin{proof}
  When $\alpha$ is the empty set, $F_{\alpha,T_\kappa}^{(m)}$ is the identity matrix and $\big\|K_n^{-1} \,  \Xi_{\alpha,T_\kappa}^{(m)} \big\|_\infty \leq  \frac{m}{c_{\kappa,1}}$ (using (\ref{eqn:knot_derivative_bd})) so that the corollary holds. We thus consider the nonempty $\alpha$ as follows.

  (1) Observe that the knot sequence $T_\kappa$ can be formed by inserting additional knots into the knot sequence $V_{\alpha,T_\kappa}$. By Proposition~\ref{prop:F_times_B_T=B_alpha}, we see that for each $\ell=1, \ldots, q_\alpha$,
  $\sum_{j = 1}^{T_n} \big [F_{\alpha,T_\kappa}^{(m)} \big]_{\ell,j} B_{m,j}^{T_\kappa}(x) = B_{m,\ell}^{V_{\alpha,T_\kappa}}(x)= \sum_{i=1}^{q_\alpha} [ \mathbf e_\ell \big]_i B_{m,i}^{V_{\alpha,T_\kappa}}(x)$ for all $x \in [0, 1]$.
  Since $\mathbf e_\ell$ has zero sign changes,
  we deduce via \cite[Lemma 27, Chapter XI]{deBoor_book01} that $\big (F_{\alpha,T_\kappa}^{(m)} \big)_{\ell,\bullet}$ has zero sign changes for each $\ell = 1,\dots,q_\alpha$. This shows that either
  $\big(F_{\alpha,T_\kappa}^{(m)} \big)_{\ell,\bullet} \geq 0$ or $\big(F_{\alpha,T_\kappa}^{(m)} \big)_{\ell,\bullet} \le 0$. In view of the nonnegativity of $B_{m,j}^{T_\kappa}$ and (\ref{eqn:sum_of_Bsplines}), the latter implies that $B_{m,\ell}^{V_{\alpha,T_\kappa}}(x)\le 0$ for all $x\in [0, 1]$. But this contradicts the fact that $B_{m,\ell}^{V_{\alpha,T_\kappa}}(x)>0$ when $x$ is in the interior of the support of $B_{m,\ell}^{V_{\alpha,T_\kappa}}$. Therefore, $F_{\alpha,T_\kappa}^{(m)}$ is a nonnegative matrix.

    When $m=1$, it is clear that $\|(F_{\alpha,T_\kappa}^{(m)})^T\|_\infty=1$. Consider $m \ge 2$ below. Thanks to the nonnegativity of  $F_{\alpha,T_\kappa}^{(m)}$, we have $\|(F_{\alpha,T_\kappa}^{(m)})^T\|_\infty = \|(F_{\alpha,T_\kappa}^{(m)})^T \cdot \onebld_{q_\alpha}\|_\infty$.
   By the construction of $F_{\alpha,T_\kappa}^{(m)}$ in (\ref{eqn:def_F}) and the
   structure of $\wh D^{(q_\alpha)}$, $\Xi_{\alpha,T_\kappa}^{(m-1)}$, $F_{\alpha,T_\kappa}^{(m-1)}$, $\wh \Delta_{m-1,T_\kappa}$, and $\wh S^{(T_n)}$ given by~(\ref{eqn:Xi_matrix}),~(\ref{eqn:wt_Delta}),~(\ref{eqn:def_Swh}), and statement (1) of Lemma~\ref{lem:F_m_bound},  we have
   \begin{eqnarray*}
     \onebld^T_{q_\alpha} \cdot F_{\alpha,T_\kappa}^{(m)}  & = & \onebld^T_{q_\alpha} \cdot \wh D^{(q_\alpha)} \cdot \Xi^{(m-1)}_{\alpha, T_\kappa} \cdot F^{(m-1)}_{\alpha, T_\kappa} \cdot \wh \Delta_{m-1, T_\kappa} \cdot \wh S^{(T_n)} \\
     & = & \mathbf e^T_1 \cdot \Big(  \Xi^{(m-1)}_{\alpha, T_\kappa} \cdot F^{(m-1)}_{\alpha, T_\kappa} \cdot \wh \Delta_{m-1, T_\kappa} \Big) \cdot \wh S^{(T_n)}
     \, = \, \mathbf e^T_1 \cdot \begin{bmatrix} 1 & 0 \\ 0 & \star \end{bmatrix} \cdot \wh S^{(T_n)} \\
     & = &
     %%\, = \,
     \mathbf e^T_1 \cdot \wh S^{(T_n)} = \onebld^T_{T_n}.
   \end{eqnarray*}
  This shows that $\|(F_{\alpha,T_\kappa}^{(m)})^T \cdot \onebld_{q_\alpha}\|_\infty=1$, completing the proof of statement (1).

(2) It follows from (\ref{eqn:sum_of_Bsplines}) that for each $\ell=1, \ldots, q_\alpha$, $\sum_{j=1}^{T_n} \big[F_{\alpha,T_\kappa}^{(m)} \big]_{\ell, j} B^{T_\kappa}_{m,j}(x) = 0$ except on $[\kappa_{i_{\ell-m}},\kappa_{i_\ell}]$, i.e., the support of $B_{m,\ell}^{V_{\alpha,T_\kappa}}$.
Note that $i_s:= s$ if $s<0$, and $i_{|\ol\alpha|+p} := K_n+p-1$ for any $p = 1,\ldots,m$; furthermore,  $\kappa_i = 0$ if $i \leq 0$ and $\kappa_i = 1$ for $i \geq |\ol \alpha|+1$.
Additionally, for any $r = 1,\dots,K_n$, the B-splines $B_{m,j}^{T_\kappa}$ that are not identically zero on $[\kappa_{r-1},\kappa_r]$ are linearly independent when restricted to $[\kappa_{r-1},\kappa_r]$.  Hence,
if $\sum_{j=1}^{T_n} \big[F_{\alpha,T_\kappa}^{(m)} \big]_{\ell,j} B_{m,j}^{T_\kappa}(x) = B_{m,\ell}^{V_{\alpha,T_\kappa}}(x) = 0, \forall \ x\in [\kappa_{r-1},\kappa_r]$ for some $r$, then $\big [F^{(m)}_{\alpha,T_\kappa} \big]_{\ell,j} = 0$ for each $j = r,r+1,\ldots,r+m-1$.
This, along with the fact that $B_{m,\ell}^{V_{\alpha,T_\kappa}}(x) = 0$ except on $[\kappa_{i_{\ell-m}},\kappa_{i_\ell}]$, shows that $\big[F_{\alpha,T_\kappa}^{(m)} \big]_{\ell,j} = 0$ for all $j = 1,2,\ldots,i_{\ell-m}+m-1$ and $j = i_\ell+1,i_\ell+2,\ldots,T_n$.
Moreover, by the nonnegativity of $F_{\alpha,T_\kappa}^{(m)}$ and $\big\|(F_{\alpha,T_\kappa}^{(m)})^T \big\|_\infty =1$, we see that all nonzero entries of $F_{\alpha,T_\kappa}^{(m)}$ are less than or equal to one. Therefore, for each $\ell=1, \ldots, q_\alpha$,
\begin{eqnarray*}
 \left \| \Big( K_n^{-1} \Xi_{\alpha,T_\kappa}^{(m)} F_{\alpha,T_\kappa}^{(m)} \Big)_{\ell, \bullet} \right \|_\infty
& = &  K_n^{-1} \frac{m}{\kappa_{i_\ell}-\kappa_{i_{\ell-m}}} \sum_{j= i_{\ell-m}+m}^{i_\ell} \big[F_{\alpha,T_\kappa}^{(m)} \big]_{\ell,j}\\
& \leq &  K^{-1}_n \cdot \frac{m}{\kappa_{i_\ell}-\kappa_{i_{\ell-m}}} \cdot \Big( i_\ell - i_{\ell-m}-m+1\Big).
\end{eqnarray*}
Since  $\kappa_{i_\ell}-\kappa_{i_{\ell-m}} \ge \frac{c_{\kappa, 1} } {K_n} \big(i_\ell-i_{\ell-m}-m+1 \big)>0$ (see the discussions before (\ref{eqn:knot_derivative_bd})), we obtain,  for each $\ell=1, \ldots, q_\alpha$,  $\big\| \big( K_n^{-1} \Xi_{\alpha,T_\kappa}^{(m)} F_{\alpha,T_\kappa}^{(m)} \big)_{\ell, \bullet} \big\|_\infty \le m/c_{\kappa, 1}$. This completes the proof of statement (2).
\end{proof}

%%%%%%%%%%%%%%%%%%%%%%%%%%%%%%%%%%%%%%%%%%%%%%%%%%%%%%%%%%%%%%%%%%%%%%%%%%%%%%%%%%%%%%%%%%%%

We exploit Proposition~\ref{prop:Z_approx_Bspline}  to derive more uniform bounds and uniform error bounds.
Many of these bounds require $L_n$ to be sufficiently large and satisfy suitable order conditions with respect to $K_n$. We introduce these order conditions as follows.
Let $(K_n)$ be an increasing sequence of natural numbers with $K_n \rightarrow \infty$ as $n\rightarrow \infty$. We say that a sequence $(L_n)$ of natural numbers satisfies
\begin{description} \label{defintion_H_property}
\item Property $\bf H$: if there exist two increasing sequences of natural numbers $(M_n)$ and $(J_n)$ with $M_n\ge m \cdot K_n/c_{\kappa, 1}$ for each $n$ and $(J_n) \rightarrow \infty$ as $n\rightarrow \infty$ such that $L_n = J_n \cdot M^{m+1}_n$ for each $n$, where $c_{\kappa, 1}>0$ is used to define $\mathcal T_{K_n}$ in (\ref{eqn:set_of_T_K_n}), and $m$ is the fixed spline order.
\end{description}
Note that the sequence $(L_n)$ implicitly depends on the sequence $(K_n)$ through $(M_n)$ in this property.

Define the truncated submatrix of $Z_{m, \alpha, T_\kappa, L_n}\in \mathbb R^{q_\alpha \times (L_n +m-1)}$:
\begin{equation} \label{eqn:H_matrix}
  H_{\alpha, T_\kappa, L_n} \, := \, \big( Z_{m, \alpha, T_\kappa, L_n} \big)_{1:q_\alpha, \, 1:L_n}\in \mathbb R^{q_\alpha \times L_n}.
\end{equation}
The importance of $H_{\alpha, T_\kappa, L_n}$ is illustrated in the following facts for given $\alpha$ and $T_\kappa \in \Tcal_{K_n}$:
\begin{itemize}
  \item [(a)] It follows from statement (2) of Lemma~\ref{lem:matrix_FtimesX} that $F^{(m)}_{\alpha, T_\kappa} \cdot X_{m, T_\kappa, L_n} = Z_{m, \alpha, T_\kappa, L_n}$. Hence, by
      (\ref{eqn:H_matrix}), we obtain
%\begin{equation} \label{eqn:H_FtimesX}
$
   H_{\alpha, T_\kappa, L_n} = \big( Z_{m, \alpha, T_\kappa, L_n} \big)_{1:q_\alpha, 1:L_n} = F^{(m)}_{\alpha, T_\kappa} \cdot \big( X_{m, T_\kappa, L_n} \big)_{1:T_n, 1:L_n}.
$
%\end{equation}
%
  \item [(b)] In light of the definition of $\wt\Lambda_{T_\kappa, K_n, L_n}$ (cf. (\ref{eqn:wt_Lambda})) and the result in (a), %%(\ref{eqn:H_FtimesX}),
      we have
\begin{align}
 \lefteqn{ K^{-1}_n \Xi^{(m)}_{\alpha, T_\kappa} F^{(m)}_{\alpha, T_\kappa} \cdot \wt \Lambda_{T_\kappa, K_n,  L_n} \, \big(F^{(m)}_{\alpha, T_\kappa}\big)^T }  \notag \\
 & \, = \, \frac{\Xi^{(m)}_{\alpha, T_\kappa}}{L_n}  \cdot F^{(m)}_{\alpha, T_\kappa}  \big( X_{m, T_\kappa, L_n} \big)_{1:T_n, 1:L_n} \cdot \left( F^{(m)}_{\alpha, T_\kappa} \big( X_{m, T_\kappa, L_n} \big)_{1:T_n, 1:L_n} \right)^T \qquad\qquad \qquad \notag \\
  & \, = \,  \frac{1}{L_n} \cdot \Xi^{(m)}_{\alpha, T_\kappa}\cdot H_{\alpha, T_\kappa, L_n} \cdot \big( H_{\alpha, T_\kappa, L_n} \big)^T. \label{eqn:wtLambda_HHt}
\end{align}
This matrix will be used  in Section~\ref{sect:proof_unif_Lip_final} to approximate $K^{-1}_n\Xi^{(m)}_{\alpha, T_\kappa} F^{(m)}_{\alpha, T_\kappa} \Lambda_{K_n, P, T_\kappa} \big(F^{(m)}_{\alpha, T_\kappa} \big)^T$ in the proof of Theorem~\ref{thm:uniform_Lip}.
  \item [(c)] Note that $F^{(m)}_{\alpha, T_\kappa}$ is the identity matrix when $\alpha$ is the empty set; see the comments below (\ref{eqn:def_F}). This observation, along with the result in (a),
        shows that if $\alpha$ is the empty set, then
      $\big( X_{m, T_\kappa, L_n}\big)_{1:T_n, 1:L_n}  = H_{\alpha, T_\kappa, L_n}$. Moreover,
      it follows from (\ref{eqn:wtLambda_HHt})
    that when $\alpha$ is the empty set,
   $
   \wt \Lambda_{T_\kappa, K_n, L_n}=\frac{K_n}{L_n} H_{\alpha, T_\kappa, L_n} \big( H_{\alpha, T_\kappa, L_n} \big)^T$. These results will be used in Proposition~\ref{prop:error_bd_Lamda}.
\end{itemize}

With the definition of $H_{\alpha, T_\kappa, L_n}$, we
establish a uniform error bound between a B-spline Gramian matrix and
$(L_n)^{-1} \cdot \Xi^{(m)}_{\alpha, T_\kappa} \cdot H_{\alpha, T_\kappa, L_n} \cdot \big( H_{\alpha, T_\kappa, L_n} \big)^T$. In light of (\ref{eqn:wtLambda_HHt}), this result is critical to obtaining a uniform bound of the $\ell_\infty$-norm of the matrix product $\big(\Xi'_\alpha F_\alpha \Lambda_{K_n,P, T_\kappa} \, F_\alpha^T \big)^{-1}$, a key step toward the uniform Lipschitz property. To this end, we first introduce a B-spline Gramian matrix.
Consider the $m$th order B-splines $\big\{ B^{V_{\alpha, T_\kappa}}_{m, j} \big\}^{|\ol\alpha|+m}_{j=1}$ corresponding to  the knot sequence $V_{\alpha, T_\kappa}$ defined in (\ref{eqn:V_knot_seq}) associated with any index set $\alpha$ and $T_\kappa \in \mathcal T_{K_n}$. Specifically, define the Gramian matrix $G_{\alpha, T_\kappa} \in \mathbb R^{q_\alpha\times q_\alpha}$ (where we recall $q_\alpha :=|\ol\alpha|+m$) as
\begin{equation} \label{eqn:Galp_matrix}
   \big[ G_{\alpha, T_\kappa} \, \big]_{i, \, j} \ := \  \frac{  \left \langle \, B^{V_{\alpha, T_\kappa}}_{m, i}, \ B^{V_{\alpha, T_\kappa}}_{m, j}  \right \rangle } {  \big\| B^{V_{\alpha, T_\kappa}}_{m, i} \big\|_{L_1} }, %=
    \qquad \forall \ i,j \,= \, 1, \ldots, q_\alpha.
\end{equation}

\begin{proposition} \label{prop:Galpha_approximation}
  Let $(K_n)$ be an increasing sequence with $K_n \rightarrow \infty$ as $n\rightarrow \infty$,
  and $(L_n)$ be of Property $\bf H$ defined by $(J_n)$ and $(M_n)$. Let $G_{\alpha, T_\kappa}$ and $H_{\alpha, T_\kappa, L_n}$ be defined for $T_\kappa \in \Tcal_{K_n}$ and $\alpha$. %%defined in (\ref{eqn:index_alpha}).
  Then  there exists $n_*\in \mathbb N$, which depends on $(L_n)$ only and is independent of $T_\kappa$ and $\alpha$, such that for  any $T_\kappa \in \Tcal_{K_n}$ with $n \ge n_*$ and any index set $\alpha$,
       \[
          \left \| G_{\alpha, T_\kappa} - \frac{1}{L_n} \cdot \Xi^{(m)}_{\alpha, T_\kappa} \cdot H_{\alpha, T_\kappa, L_n} \cdot \big( H_{\alpha, T_\kappa, L_n} \big)^T \right \|_\infty \, \le \, \frac{6 \cdot c_{\kappa, 1} \cdot (3\cdot 2^{m-1}-2)}{J_n}, \qquad \forall \ n \ge n_*,
       \]
       where $\Xi^{(m)}_{\alpha, T_\kappa}$ is defined in (\ref{eqn:Xi_matrix}).
 %\end{itemize}
\end{proposition}

\begin{proof}
    Given arbitrary $\alpha$, $T_\kappa \in \Tcal_{K_n}$, and $(L_n)$ of Property $\bf H$, we use $H$ and $\Xi^{(m)}$ to denote $H_{\alpha, T_\kappa, L_n}$ and $\Xi^{(m)}_{\alpha, T_\kappa}$ respectively to simplify notation. Also take $r_k:=\frac{k-1}{L_n}$ for $k=1, \ldots, L_n$. When $m=1$, $G_{\alpha, T_\kappa}$ is diagonal (see the summary of B-splines at the beginning of Section~\ref{sec:prob_form} for the reason), and so is $H\cdot H^T$. Using  $\big( \|  B^{V_{\alpha, T_\kappa}}_{m, j} \|_{L_1}\big)^{-1} = \frac{m}{\kappa_{i_j}- \kappa_{{i_{j-m} } } } = \big[\Xi^{(m)}\big]_{j,j}$ and a  result similar to (\ref{eqn:bound_for_p=1}), the desired result follows easily.
    We consider $m \ge 2$ next.
    It follows from the definition of $H\in \mathbb R^{q_\alpha\times L_n}$ and Proposition~\ref{prop:Z_approx_Bspline} that for each $j=1, \ldots, q_\alpha$ and $k=1, \ldots, L_n$,
    \begin{equation} \label{eqn:H_error_bd}
       \left | \big[ H \big]_{j, k} -  B^{V_{\alpha, T_\kappa}}_{m, j}( r_{k}) \right|
        =  \left | \big[ Z_{m, \alpha, T_\kappa, L_n} \big]_{j, k} -  B^{V_{\alpha, T_\kappa}}_{m, j}( r_{k}) \right|
         \le \frac{C_m (M_n)^{m-1}}{L_n} \le \frac{C_m}{J_n  (M_n)^2} %%%\ \forall \ n \in \mathbb N,
    \end{equation}
    for any $n \in \mathbb N$,
    where $C_m := 6\cdot (2^{m-1}-1)$. Since $0\le B^{V_{\alpha, T_\kappa}}_{m, j}(x)\le 1, \ \forall \, x \in [0, 1]$ for each $j$ and
    \begin{align*}
     \lefteqn{  \big[ H \big]_{j, \ell} \cdot \big[ H \big]_{k, \ell} -  B^{V_{\alpha, T_\kappa}}_{m, j}( r_{\ell}) \cdot B^{V_{\alpha, T_\kappa}}_{m, k}( r_{\ell} ) } \\
     & = \left( \big[ H \big]_{j, \ell}- B^{V_{\alpha, T_\kappa}}_{m, j}( r_{\ell} ) \right) \cdot \left( \big[ H \big]_{k, \ell}- B^{V_{\alpha, T_\kappa}}_{m, k}( r_{\ell}) \right) + B^{V_{\alpha, T_\kappa}}_{m, j}( r_{\ell}) \cdot \left( \big[ H \big]_{k, \ell} - B^{V_{\alpha, T_\kappa}}_{m, k}( r_{\ell}) \right) \\
     & \qquad + B^{V_{\alpha, T_\kappa}}_{m, k}( r_{\ell}) \cdot \left( \big[ H \big]_{j, \ell} - B^{V_{\alpha, T_\kappa}}_{m, j}( r_{\ell}) \right),
    \end{align*}
    we deduce, via (\ref{eqn:H_error_bd}), that for each $j,k=1, \ldots, q_\alpha$ and $\ell=1, \ldots, L_n$,
    \[
      \left| \big[ H \big]_{j, \ell} \cdot \big[ H \big]_{k, \ell} -  B^{V_{\alpha, T_\kappa}}_{m, j}( r_{\ell}) \cdot B^{V_{\alpha, T_\kappa}}_{m, k}( r_{\ell} ) \right| \, \le \, \left( \frac{C_m}{J_n (M_n)^2} \right)^2 + \frac{2 C_m}{J_n  (M_n)^2}, \quad \forall \ n\in \mathbb N.
    \]
     Since $m \ge 2$, $\big( B^{V_{\alpha, T_\kappa}}_{m, j}(x) B^{V_{\alpha, T_\kappa}}_{m, k}(x) \big)$ is continuous and differentiable on $[0, 1]$ except at (at most) finitely many points in $[0, 1]$.
     In light of (\ref{eqn:Bspline_dev_bd_2}) %% and the upper bound of B-splines,
      we have, for any $x\in [0, 1]$ where the derivative exists,
     \begin{equation} \label{eqn:deriv_Bproduct}
        \left| \Big( B^{V_{\alpha, T_\kappa}}_{m, j}(x) B^{V_{\alpha, T_\kappa}}_{m, k}(x) \Big)' \right| \, = \, \left| \Big( B^{V_{\alpha, T_\kappa}}_{m, j}\Big)'(x) B^{V_{\alpha, T_\kappa}}_{m, k}(x) +
           B^{V_{\alpha, T_\kappa}}_{m, j}(x) \Big(B^{V_{\alpha, T_\kappa}}_{m, k}\Big)'(x)
         \right| \, \le \, 4 M_n.
     \end{equation}
     Combining these results with $\big( \|  B^{V_{\alpha, T_\kappa}}_{m, j} \|_{L_1}\big)^{-1} = \frac{m}{\kappa_{i_j}- \kappa_{{i_{j-m} } } } = \big[\Xi^{(m)}\big]_{j,j}$, we apply Lemma~\ref{lem:funct_bound} with $\wt n:=L_n$, $a:=0$, $b:=1$, $s_k:=k/L_n$, $\varrho:=1/L_n$, $f(x):= B^{V_{\alpha, T_\kappa}}_{m, j}(x)\cdot B^{V_{\alpha, T_\kappa}}_{m, k}(x)$, $v:=(v_\ell)=\big( [ H ]_{j, \ell} \cdot [ H ]_{k, \ell} \big)$, $\mu_1:=\left( \frac{C_m}{J_n (M_n)^2} \right)^2 + \frac{2 C_m}{J_n  (M_n)^2}$, and $\mu_2:=4 M_n$ to obtain $n_*\in \mathbb N$ depending on $(L_n)$ only such that for any $j, k=1, \ldots, q_{\alpha}$,
    \begin{align*}
      \lefteqn{ \left| \frac{1}{L_n} \big[  \Xi^{(m)} \cdot H \cdot H^T \big]_{j, \, k} -   \big[ G_{\alpha, T_\kappa} \, \big]_{j, \, k} \right| \, = \, \left| \frac{1}{L_n} \big[  \Xi^{(m)} \cdot H \cdot H^T \big]_{j, \, k} -   \frac{  \left \langle \, B^{V_{\alpha, T_\kappa}}_{m, j}, \ B^{V_{\alpha, T_\kappa}}_{m, k}  \right \rangle } {  \big\| B^{V_{\alpha, T_\kappa}}_{m, j} \big\|_{L_1} } \right| }\\
      & \, = \, \frac{m}{ \kappa_{i_j}- \kappa_{{i_{j-m}} } } \left|  \sum^{L_n}_{\ell=1} \frac{v_\ell}{L_n}  - \int^1_{0} f(t) dt \right|  \, \le \, M_n \cdot \left[ \Big( \frac{C_m}{J_n (M_n)^2} \Big)^2 + \frac{2 C_m}{J_n (M_n)^2} + \frac{6 M_n}{L_n} \right] \\
      & \, \le \, \frac{ 3 C_m + 6 }{J_n M_n}, \qquad \qquad \forall \ n \ge n_*,
    \end{align*}
    where we use $L_n = J_n \cdot (M_n)^{m+1}$. Since $G_{\alpha, T_\kappa}$ has $q_\alpha$ columns and $q_\alpha=|\ol\alpha|+m \le K_n+m -1 \le m \cdot K_n \le c_{\kappa, 1}  M_n$, it is deduced that for any $\alpha$ and  $T_\kappa \in \Tcal_{K_n}$,
    \[
       \left \| G_{\alpha, T_\kappa} - \frac{1}{L_n} \cdot \Xi^{(m)} \cdot H \cdot H^T \right \|_\infty \, \le \, q_\alpha \cdot \frac{ 3 C_m + 6 }{J_n M_n} \, \le \, \frac{c_{\kappa, 1} \cdot (3 C_m + 6)}{J_n}, \quad \forall \ n \ge n_*.
    \]
    The proof is completed by noting that %%$C_m=6\cdot (2^{m-1}-1)$ such that
    $3 C_m + 6 = 6 \cdot (3\cdot 2^{m-1}-2)$.
\end{proof}

%\gap

An immediate consequence of Proposition~\ref{prop:Galpha_approximation} is the invertibility of $ \Xi^{(m)}_{\alpha, T_\kappa} H_{\alpha, T_\kappa, L_n}  \big( H_{\alpha, T_\kappa, L_n} \big)^T/L_n$ and the uniform bound of its inverse in the $\ell_\infty$-norm.

\begin{corollary} \label{coro:invert_HHt}
 Let $(K_n)$ be an increasing sequence with $K_n \rightarrow \infty$ as $n\rightarrow \infty$,
 and $(L_n)$ be of Property $\bf H$ defined by $(J_n)$ and $(M_n)$.
 Then there exists
 $n'_*\in \mathbb N$, which depends on $(L_n)$ only, such that for any $T_\kappa \in \Tcal_{K_n}$ with $n \ge n'_*$, any index set $\alpha$, and any $n \ge n'_*$,
 $\frac{1}{L_n} \Xi^{(m)}_{\alpha, T_\kappa} H_{\alpha, T_\kappa, L_n}  \big( H_{\alpha, T_\kappa, L_n} \big)^T$ is invertible, and
 $
  \big\| \big( \, \frac{1}{L_n} \Xi^{(m)}_{\alpha, T_\kappa} H_{\alpha, T_\kappa, L_n}  \big( H_{\alpha, T_\kappa, L_n} \big)^T \, \big)^{-1} \big\|_\infty \, \le \, \frac{3\rho_m}{2}, %%\nu_H.
 $
 where $\rho_m$ is a positive constant depending on $m$ only.
\end{corollary}

\begin{proof}
   Choose arbitrary $\alpha$, $T_\kappa \in \Tcal_{K_n}$, and $(L_n)$ of Property $\bf H$. It follows from \cite[Theorem~I]{Shardin_AM01} (cf. Theorem~\ref{thm:Shardin_bound}) that the Gramian matrix $G_{\alpha, T_\kappa}$ is invertible and there exists a positive constant $\rho_m$ such that for any $T_\kappa \in \Tcal_{K_n}$ and any index set $\alpha$, $\|( G_{\alpha, T_\kappa})^{-1}\|_\infty \le \rho_m$.
   Furthermore,
   it follows from Proposition~\ref{prop:Galpha_approximation}  that for  any $T_\kappa \in \Tcal_{K_n}$ and any index set $\alpha$, $\| G_{\alpha, T_\kappa} - \frac{1}{L_n} \Xi^{(m)}_{\alpha, T_\kappa} H_{\alpha, T_\kappa, L_n}  \big( H_{\alpha, T_\kappa, L_n} \big)^T \|_\infty \le 6 c_{\kappa, 1} (3\cdot 2^{m-1}-2)/J_n, \forall \, n \ge n_*$. Since $J_n\rightarrow \infty$ as $n \rightarrow \infty$, we deduce from Lemma~\ref{lem:matrix_approx} that there exists $n'_*\in \mathbb N$ with $n'_* \ge n_*$ such that for any $T_\kappa \in \Tcal_{K_n}$ with $n \ge n'_*$, any index set $\alpha$, and any $n \ge n'_*$,
 $\frac{1}{L_n} \Xi^{(m)}_{\alpha, T_\kappa} H_{\alpha, T_\kappa, L_n}  \big( H_{\alpha, T_\kappa, L_n} \big)^T$ is invertible, and
 $
  \big\| \big( \frac{1}{L_n} \Xi^{(m)}_{\alpha, T_\kappa} H_{\alpha, T_\kappa, L_n}  \big( H_{\alpha, T_\kappa, L_n} \big)^T \big)^{-1} \big\|_\infty \, \le \, \frac{3}{2}\rho_m.
 $
\end{proof}

For any $K_n \in \mathbb N$ and $T_\kappa \in \Tcal_{K_n}$, let $\{ B^{T_\kappa}_{m, i} \}^{T_n}_{i=1}$ be the B-splines of order $m$ defined by the knot sequence $T_\kappa$, where we recall that $T_n := K_n+m-1$.
Note that $B^{T_\kappa}_{m, i}$ is equal to $B^{V_{\alpha, T_\kappa}}_{m, i}$ when $\alpha$ is the empty set. For the given $T_\kappa$, define the $T_n \times T_n$ matrix $\wh \Lambda_{K_n, T_\kappa}$ as
\begin{equation} \label{eqn:wh_Lambda_matrix}
     \big[ \, \wh \Lambda_{K_n, T_\kappa} \big]_{i, j} \, : = \, K_n \cdot \left\langle B^{T_\kappa}_{m, i}, \, B^{T_\kappa}_{m, j} \right\rangle, \qquad \forall \ i, j= 1, \ldots, T_n.
\end{equation}
Clearly, $\wh \Lambda_{K_n, T_\kappa}$ is positive definite and invertible.
The following result presents important properties of $\wh \Lambda_{K_n, T_\kappa}$. In particular, it shows via $\wh \Lambda_{K_n, T_\kappa}$ that $\wt \Lambda_{T_\kappa, K_n, L_n}$ approximates $\Lambda_{K_n,P, T_\kappa}$ with a uniform error bound, which is crucial to the proof of the uniform Lipschitz property. Note that the constant $\rho_m>0$ used below is given in \cite[Theorem~I]{Shardin_AM01} (cf. Theorem~\ref{thm:Shardin_bound}) and depends on $m$ only.

\begin{proposition} \label{prop:error_bd_Lamda}
  Let $(K_n)$ be an increasing sequence with $K_n \rightarrow \infty$ as $n\rightarrow \infty$,
  and $(L_n)$ be of Property $\bf H$ defined by $(J_n)$ and $(M_n)$. The following hold:
  \begin{itemize}
    \item [(1)] For any $K_n$ and $T_\kappa\in \Tcal_{K_n}$, $\| \big(\wh \Lambda_{K_n, T_\kappa}\big)^{-1} \|_\infty \le m \rho_m/c_{\kappa, 1}$;
    \item [(2)] There exists $n_*\in \mathbb N$, depending on $(L_n)$ only, such that for any $T_\kappa \in T_{K_n}$ with $n \ge n_*$,
        \[
            \left \| \wt \Lambda_{T_\kappa, K_n, L_n} - \wh \Lambda_{K_n, T_\kappa} \right\|_\infty \, \le \, \frac{6 \cdot c_{\kappa, 2}\cdot c_{\kappa, 1} \cdot (3\cdot 2^{m-1}-2)}{J_n},  \qquad \forall \ n \ge n_*;
        \]
     \item [(3)] For any $n$, $K_n$, $P\in \mathcal P_n$, and $T_\kappa\in \Tcal_{K_n}$,
\[
\left \|  \Lambda_{K_n, P, T_\kappa}- {\wh \Lambda}_{K_n, T_\kappa} \right\|_\infty \, \leq \, (2m-1)\left(
\frac{6 m^2 c_\omega c_{\kappa, 2} }{c_{\kappa, 1} } + 3 c_\omega  \right)\frac{K_n}{n}.
\]
  \end{itemize}
\end{proposition}

\begin{proof}
 (1) For any $T_\kappa \in \Tcal_{K_n}$, it follows from (\ref{eqn:Galp_matrix}) and
(\ref{eqn:wh_Lambda_matrix})
 that when $\alpha$ is the empty set, $G_{\emptyset, T_\kappa} = K^{-1}_n \cdot \Xi^{(m)}_{\emptyset, T_\kappa} \cdot \wh \Lambda_{K_n, T_\kappa} $. Hence, in light of $\| \Xi^{(m)}_{\emptyset, T_\kappa} \|_\infty \le m K_n/c_{\kappa, 1}$  and $\|\big( G_{\emptyset, T_\kappa} \big)^{-1} \|_\infty\le \rho_m$ for any $T_\kappa \in \Tcal_{K_n}$, we have
 \[
    \left\| \big(\wh \Lambda_{K_n, T_\kappa}\big)^{-1} \right\|_\infty \, = \, K^{-1}_n \left\| \big( G_{\emptyset, T_\kappa} \big)^{-1} \cdot \Xi^{(m)}_{\emptyset, T_\kappa} \right\|_\infty \, \le \,
     K^{-1}_n \cdot \rho_m \cdot \frac{m K_n}{c_{\kappa, 1} } \, =\,  \frac{m \rho_m}{ c_{\kappa, 1} }, \quad \forall \ T_\kappa \in \Tcal_{K_n}.
 \]

 (2)
  Recall from the comment below (\ref{eqn:wtLambda_HHt}) that $\wt \Lambda_{T_\kappa, K_n, L_n}=\frac{K_n}{L_n} H_{\alpha, T_\kappa, L_n} \big( H_{\alpha, T_\kappa, L_n} \big)^T$ when $\alpha$ is the empty set.
   Also, noting from the proof of statement (1) that $\wh \Lambda_{K_n, T_\kappa} = K_n \cdot  \big(\Xi^{(m)}_{\emptyset, T_\kappa}\big)^{-1} \cdot G_{\emptyset, T_\kappa}$, we obtain via Proposition~\ref{prop:Galpha_approximation} that there exists $n_*\in \mathbb N$, depending on $(L_n)$ only, such that for any $T_\kappa \in T_{K_n}$ with $n \ge n_*$,
  \begin{align*}
    \big \| \wt \Lambda_{T_\kappa, K_n, L_n} - \wh \Lambda_{K_n, T_\kappa} \big\|_\infty & \,  = \, K_n \cdot \left\| \big(\Xi^{(m)}_{\emptyset, T_\kappa}\big)^{-1} \cdot \Big( \, \frac{1}{L_n} \Xi^{(m)}_{\emptyset, T_\kappa} H_{\emptyset, T_\kappa, L_n} \big( H_{\emptyset, T_\kappa, L_n} \big)^T - G_{\emptyset, T_\kappa} \Big) \right\|_\infty  \\
   & \, \le \, K_n \cdot \big\|\big(\Xi^{(m)}_{\emptyset, T_\kappa}\big)^{-1} \big\|_\infty \cdot \left\|  \frac{1}{L_n} \Xi^{(m)}_{\emptyset, T_\kappa} H_{\emptyset, T_\kappa, L_n} \big( H_{\emptyset, T_\kappa, L_n} \big)^T - G_{\emptyset, T_\kappa} \right\|_\infty \\
   & \, \le \, K_n \cdot \frac{c_{\kappa, 2} }{K_n} \cdot \frac{6 \cdot c_{\kappa, 1} \cdot (3\cdot 2^{m-1}-2)}{J_n} \\
   & \, = \, \frac{6 \cdot c_{\kappa, 2}\cdot c_{\kappa, 1} \cdot (3\cdot 2^{m-1}-2)}{J_n},  \qquad \forall \ n \ge n_*.
  \end{align*}

 (3) Consider an arbitrary knot sequence $T_\kappa\in \Tcal_{K_n}$ given by $T_\kappa=\{ 0=\kappa_0<\kappa_1 <\cdots < \kappa_{K_n-1} < \kappa_{K_n}=1\}$ with the usual extension. Furthermore, let $P\in \mathcal P_n$ be a design point sequence given by $P=\{ 0 = x_0 < x_1 < \dots < x_n = 1 \}$. Recall
  the design matrix $\wh X \in \mathbb R^{(n+1) \times T_n}$ with $[\wh X]_{\ell, i} = B^{T_\kappa}_{m, i}(x_\ell)$ for each $\ell$ and $i$, and $\Lambda_{K_n, P, T_\kappa} = K_n \cdot \wh X^T \Theta_n \wh X \in \mathbb R^{T_n\times T_n}$ with $\Theta_n=\mbox{diag}\big(x_1-x_0, x_2-x_1, \ldots, x_{n+1}-x_n \big)$.
  When $m=1$, both $\wh \Lambda_{K_n, T_\kappa}$ and $\Lambda_{K_n, P, T_\kappa}$ are diagonal matrices, and the desired bound follows easily from a similar argument as in Proposition~\ref{prop:Galpha_approximation}. Hence, it suffices to consider $m \ge 2$ below.
  For any fixed $i, j =1, \ldots, T_n$, we assume that there are $(\wt n_i+1)$ design points in $P$ on the support $[\kappa_{i-m}, \kappa_i]$ of $B^{T_\kappa}_{m, i}$. Specifically, there exists $r_i\in \mathbb Z_+$ such that: (i) $x_{r_i}, x_{r_i+1}, \ldots, x_{r_i + \wt n_i} \in [\kappa_{i-m}, \kappa_i]$; (ii) $x_{r_i}=0$ or $x_{r_i-1}< \kappa_{i-m}$; and (iii) $x_{r_i+\wt n_i}=1$ or $x_{r_i+\wt n_i+1}>\kappa_i$. Hence, letting $\omega_\ell:=x_{\ell} - x_{\ell-1}$ for $\ell=1, \ldots, n+1$, it follows from the definitions of $\Lambda_{K_n, P, T_\kappa}$ and $\wh\Lambda_{K_n, T_\kappa}$ that
  \begin{align*}
 \lefteqn{ \left| \Big[ \Lambda_{K_n, P, T_\kappa}- \wh \Lambda_{K_n, T_\kappa} \Big]_{i, \, j} \right| \, = \, K_n \left|\sum_{\ell=r_i}^{r_i+\wt n_i} \omega_\ell \cdot B^{T_\kappa}_{m, i}(x_\ell) \cdot B^{T_\kappa}_{m, j}(x_\ell) - \int_{\kappa_{i-m}}^{\kappa_i} B^{T_\kappa}_{m, i}(x) B^{T_\kappa}_{m, j}(x)\,dx \right| } \\
 &  \leq  K_n \left|\sum_{\ell=r_i}^{r_i+\wt n_i-1} \omega_\ell \cdot B^{T_\kappa}_{m, i}(x_\ell) \cdot B^{T_\kappa}_{m, j}(x_\ell) - \int_{x_{r_i}}^{x_{r_i+\wt n_i}} B^{T_\kappa}_{m, i}(x) B^{T_\kappa}_{m, j}(x)\,dx \right| + K_n \int^{x_{r_i}}_{\kappa_{i-m}} B^{T_\kappa}_{m, i}(x)B^{T_\kappa}_{m, j}(x)\,dx \\
 & \quad + K_n \int^{\kappa_i}_{x_{r_i+\wt n_i}} B^{T_\kappa}_{m, i}(x)B^{T_\kappa}_{m, j}(x)\,dx
 + K_n\cdot \omega_{r_i+\wt n_i} B^{T_\kappa}_{m, i}(x_{r_i+\wt n_i}) \cdot B^{T_\kappa}_{m, j}(x_{r_i+\wt n_i}) \\
& \, \leq \, K_n \left|\sum_{\ell=r_i+1}^{r_i+\wt n_i} \omega_{\ell-1} B^{T_\kappa}_{m, i}(x_{\ell-1})B^{T_\kappa}_{m, j}(x_{\ell-1}) - \int_{x_{r_i}}^{x_{r_i+\wt n_i}} B^{T_\kappa}_{m, i}(x)B^{T_\kappa}_{m, j}(x)\,dx \right| + 3c_\omega\frac{K_n}{n},
\end{align*}
using that for each $i$ and $x\in [0,1]$, $0\leq B^{T_\kappa}_{m, i}(x) \leq 1$ and $\max_\ell(x_{r_i}- \kappa_{i-m}, \kappa_i - x_{r_i+\wt n_i}, \omega_\ell) \le c_\omega/n$.

By the virtue of (\ref{eqn:Bspline_dev_bd}) and (\ref{eqn:deriv_Bproduct}), it is easy to verify that
$
        \left| \big( B^{T_\kappa}_{m, i}(x) B^{ T_\kappa}_{m, j}(x) \big)' \right| \, \le  \, \frac{4 m K_n}{c_{\kappa, 1} }
$
whenever the derivative exists on $[0, 1]$. We apply Lemma~\ref{lem:funct_bound} to the first term on the right hand side with $\wt n := \wt n_i$, $a:= x_{r_i}$, $b:=x_{r_i+\wt n_i}$,
  $s_k := x_{r_i+k}$, $\varrho:= c_\omega/n$, $f(x):= B^{T_\kappa}_{m, i}(x) B^{T_\kappa}_{m, j}(x)$, which is continuous on $[0, 1]$ and is differentiable except at (at most) finitely many points, $v=(v_\ell) \in \mathbb R^{\tilde n_i}$ given by $v_\ell := B^{T_\kappa}_{m, i}(x_{r_i+\ell-1})B^{T_\kappa}_{m, j}(x_{r_i+\ell-1})$,  $\mu_1:=0$, and $\mu_2 := 4 m K_n/c_{\kappa, 1}$,
and we obtain
\begin{align*}
\lefteqn{ K_n \left|\sum_{\ell=r_i+1}^{r_i+\wt n_i} \omega_{\ell-1} B^{T_\kappa}_{m, i}(x_{\ell-1})B^{T_\kappa}_{m, j}(x_{\ell-1}) - \int_{x_{r_i}}^{x_{r_i+\wt n_i}} B^{T_\kappa}_{m, i}(x)B^{T_\kappa}_{m, j}(x)\,dx \right| } \\
 & \, \leq \, K_n \cdot \frac{3}{2} \cdot \frac{4 m K_n}{c_{\kappa, 1} }  \cdot \frac{c_\omega}{n} \cdot \big( \kappa_i-\kappa_{i-m} \big)
 \, \leq \, \frac{ 6 m K^2_n c_\omega  }{ n c_{\kappa, 1} } \cdot \frac{m \, c_{\kappa, 2} }{K_n} \, = \, \frac{6 m^2 c_\omega c_{\kappa, 2} }{c_{\kappa, 1} } \cdot \frac{K_n}{n}. \qquad \qquad \qquad \qquad
 %
%% \left(6c_\omega + \frac{2m^2c_{\kappa}^{(2)}c_\omega}{c_{\kappa}^{(1)}}\right)\frac{K_n}{n}.
\end{align*}
Combining the above results yields $\big|[ \Lambda_{K_n, P, T_\kappa}- \wh \Lambda_{K_n, T_\kappa} ]_{i, \, j} \big | \le \big(  \frac{6 m^2 c_\omega c_{\kappa, 2} }{c_{\kappa, 1} } + 3 c_\omega \big) \frac{K_n}{n}$ for any $i, j=1, \ldots, T_n$.
 It is noted that $B^{T_\kappa}_{m, i}B^{T_\kappa}_{m, j}\equiv 0$ on $[0, 1]$ whenever $|i-j| \geq m$.  Hence both $\Lambda_{K_n, P, T_\kappa}$ and $\wh \Lambda_{K_n, T_\kappa}$ are banded matrices with bandwidth $m$. Thus, for any $n$, $K_n$, $P\in \mathcal P_n$, and $T_\kappa\in \Tcal_{K_n}$,
\[
\left \|  \Lambda_{K_n, P, T_\kappa}- \wh \Lambda_{K_n, T_\kappa} \right\|_\infty \, \leq \, (2m-1)\left(
\frac{6 m^2 c_\omega c_{\kappa, 2} }{c_{\kappa, 1} } + 3 c_\omega  \right)\frac{K_n}{n}.
\]
This completes the proof of statement (3).
\end{proof}

%---------------------------------------------------------------
%
\subsection{Proof of Theorem~\ref{thm:uniform_Lip}} \label{sect:proof_unif_Lip_final}

In this subsection, we use the results established in the previous subsections to show the uniform Lipschitz property stated in Theorem~\ref{thm:uniform_Lip}. Fix the B-spline order $m\in \mathbb N$. Let the strictly increasing sequence $(K_n)$ be such that $K_n \rightarrow \infty$ and $K_n/n \rightarrow 0$ as $n \rightarrow \infty$. Consider the sequence $(L_n)$:
\[
   L_n \, := \, K_n \cdot \left( \left\lceil \frac{m K_n}{c_{\kappa, 1} }  \right\rceil \right)^{m+1}, \qquad \forall \ n \in \mathbb N.
\]
Clearly, $(L_n)$ satisfies Property $\bf H$ with $J_n:=K_n$ and $M_n := \lceil \frac{m K_n}{c_{\kappa, 1} } \rceil$, and depends on $(K_n)$ only.

For any $P \in \mathcal P_n$, $T_\kappa \in \Tcal_{K_n}$, and any index set $\alpha$ defined in (\ref{eqn:index_alpha}), recall that $q_\alpha=|\ol\alpha|+m$, $T_n=K_n+m-1$, and $\Lambda_{K_n, P, T_\kappa}\in \mathbb R^{T_n\times T_n}$ .
We then construct the following matrices based on the development in the past subsections: $F^{(m)}_{\alpha, T_\kappa} \in \mathbb R^{q_\alpha \times T_n}$ (cf. (\ref{eqn:def_F})), $\Xi^{(m)}_{\alpha, T_\kappa}\in \mathbb R^{q_\alpha\times q_\alpha}$ (cf. (\ref{eqn:Xi_matrix})), $X_{m, T_\kappa, L_n}\in \mathbb R^{T_n\times (L_n+m-1)}$ (cf. (\ref{eqn:def_Xmatrix})), $\wt\Lambda_{T_\kappa, K_n, L_n}\in \mathbb R^{T_n\times T_n}$ (cf. (\ref{eqn:wt_Lambda})),
and $ H_{\alpha, T_\kappa, L_n}\in \mathbb R^{q_\alpha\times L_n}$ (cf. (\ref{eqn:H_matrix})). In light of Proposition~\ref{prop:null_basis_new} and (\ref{eqn:piecewise_linear}), we have
$$
  \wh b^\alpha_{P, T_\kappa} (\bar y) \, = \,  \big( F^{(m)}_{\alpha, T_\kappa} \big)^T \cdot \ \big(F^{(m)}_{\alpha, T_\kappa}  \Lambda_{K_n,P, T_\kappa}  \big(F^{(m)}_{\alpha, T_\kappa}\big)^T \big)^{-1} \cdot
  F^{(m)}_{\alpha, T_\kappa}  \,\bar y.
$$
Note that
\begin{align} \label{eqn:upper_bd_uniformL}
 \lefteqn{ \left\| \Big( F^{(m)}_{\alpha, T_\kappa} \Big)^T  \cdot \left(F^{(m)}_{\alpha, T_\kappa} \cdot \Lambda_{K_n,P, T_\kappa} \, \big(F^{(m)}_{\alpha, T_\kappa}\big)^T \right)^{-1} \cdot
  F^{(m)}_{\alpha, T_\kappa} \right\|_\infty } \notag \\
  & \, \le \ \left\| \big( F^{(m)}_{\alpha, T_\kappa} \big)^T \right\|_\infty \cdot \left\| K_n \left(\Xi^{(m)}_{\alpha, T_\kappa} F^{(m)}_{\alpha, T_\kappa} \cdot \Lambda_{K_n,P, T_\kappa} \, \big(F^{(m)}_{\alpha, T_\kappa}\big)^T \right)^{-1} \right\|_\infty \cdot \left\|
  K^{-1}_n \Xi^{(m)}_{\alpha, T_\kappa} F^{(m)}_{\alpha, T_\kappa} \right\|_\infty.
\end{align}
By Corollary~\ref{coro:F_T_XiF_bounds}, we have the uniform bounds $\big\| \big( F^{(m)}_{\alpha, T_\kappa} \big)^T \big\|_\infty=1$ and $\big\|  K^{-1}_n \Xi^{(m)}_{\alpha, T_\kappa} F^{(m)}_{\alpha, T_\kappa} \big\|_\infty \le m/c_{\kappa, 1}$, regardless of $K_n, \alpha, T_\kappa \in \Tcal_{K_n}$.

In what follows, we develop a uniform bound for $\big\| K_n \big(\Xi^{(m)}_{\alpha, T_\kappa} F^{(m)}_{\alpha, T_\kappa} \cdot \Lambda_{K_n,P, T_\kappa} \cdot \big(F^{(m)}_{\alpha, T_\kappa}\big)^T \big)^{-1} \big\|_\infty$ on the right hand side of (\ref{eqn:upper_bd_uniformL}).
Recall from (\ref{eqn:wtLambda_HHt}) that
\[
 K^{-1}_n \Xi^{(m)}_{\alpha, T_\kappa} F^{(m)}_{\alpha, T_\kappa} \cdot \wt \Lambda_{T_\kappa, K_n,  L_n} \, \big(F^{(m)}_{\alpha, T_\kappa}\big)^T  \, = \,\frac{1}{L_n} \cdot \Xi^{(m)}_{\alpha, T_\kappa}\cdot H_{\alpha, T_\kappa, L_n} \cdot \big( H_{\alpha, T_\kappa, L_n} \big)^T.
\]
By Corollary~\ref{coro:invert_HHt}, we deduce the existence of
 $\wt n_{\star}\in \mathbb N$, which depends on $(K_n)$ only, such that for any $T_\kappa \in \Tcal_{K_n}$ with $n \ge \wt n_{\star}$, any index set $\alpha$, and any $n \ge \wt n_{\star}$,
 $K^{-1}_n \Xi^{(m)}_{\alpha, T_\kappa} F^{(m)}_{\alpha, T_\kappa} \cdot \wt \Lambda_{T_\kappa, K_n,  L_n} \cdot \big(F^{(m)}_{\alpha, T_\kappa}\big)^T$ is invertible and $\big\| K_n \big[ \Xi^{(m)}_{\alpha, T_\kappa} F^{(m)}_{\alpha, T_\kappa} \cdot \wt \Lambda_{T_\kappa, K_n,  L_n} \cdot \big(F^{(m)}_{\alpha, T_\kappa}\big)^T\big]^{-1} \big\|_\infty \le \frac{3\rho_m}{2}$. Moreover, noting that
 \begin{align*}
\lefteqn{  \left\| K^{-1}_n \Xi^{(m)}_{\alpha, T_\kappa} F^{(m)}_{\alpha, T_\kappa} \cdot \wt \Lambda_{T_\kappa, K_n,  L_n} \cdot \big(F^{(m)}_{\alpha, T_\kappa}\big)^T -
  K^{-1}_n \Xi^{(m)}_{\alpha, T_\kappa} F^{(m)}_{\alpha, T_\kappa} \cdot \Lambda_{K_n,P, T_\kappa} \cdot \big(F^{(m)}_{\alpha, T_\kappa}\big)^T \right\|_\infty } \\
  & \ \le \ \left \| K^{-1}_n \Xi^{(m)}_{\alpha, T_\kappa} F^{(m)}_{\alpha, T_\kappa} \right\|_\infty \cdot \left \| \wt \Lambda_{T_\kappa, K_n,  L_n} - \Lambda_{K_n,P, T_\kappa} \right\|_\infty \cdot \left \| \big(F^{(m)}_{\alpha, T_\kappa}\big)^T \right\|_\infty     \qquad \qquad \qquad \qquad %\qquad \qquad
 \end{align*}
and using Proposition~\ref{prop:error_bd_Lamda} as well as the uniform bounds for $\big\|
  \big( F^{(m)}_{\alpha, T_\kappa} \big)^T \big\|_\infty$ and $\big \| K^{-1}_n \Xi^{(m)}_{\alpha, T_\kappa} F^{(m)}_{\alpha, T_\kappa} \big\|_\infty $ in Corollary~\ref{coro:F_T_XiF_bounds}, we further deduce via Lemma~\ref{lem:matrix_approx} that there exists $n_* \in \mathbb N$ with $n_* \ge \wt n_\star$ such that for any $\alpha$,  $P \in \mathcal P_n$, and $T_\kappa \in \Tcal_{K_n}$ with $n \ge n_{*}$,
$K^{-1}_n \Xi^{(m)}_{\alpha, T_\kappa} F^{(m)}_{\alpha, T_\kappa} \cdot \Lambda_{K_n,P, T_\kappa} \cdot \big(F^{(m)}_{\alpha, T_\kappa}\big)^T$  is invertible and
\[
  \left\| \Big( K^{-1}_n \Xi^{(m)}_{\alpha, T_\kappa} F^{(m)}_{\alpha, T_\kappa} \cdot \Lambda_{K_n,P, T_\kappa} \cdot \big(F^{(m)}_{\alpha, T_\kappa}\big)^T \Big)^{-1} \right\|_\infty \, \le \, \frac{9\rho_m}{4}.
\]

Finally, combining the above three uniform bounds, we conclude, in light of (\ref{eqn:upper_bd_uniformL}), that the theorem holds with the positive constant $c_\infty:= 9 m \rho_m/(4 c_{\kappa, 1})$
depending on $m, c_{\kappa, 1}$ only, and $n_* \in \mathbb N$ depending on $(K_n)$ only (when $m, c_\omega, c_{\kappa, 1}, c_{\kappa, 2}$ are fixed).

%-------------------------------------------------------------------------------------------
%
\section{Applications to Shape Constrained Estimation} \label{sect:constrained_estimation}

We apply the uniform Lipschitz property of the constrained B-splines established in Theorem~\ref{thm:uniform_Lip} to the nonparametric estimation of shape constrained functions subject to nonnegative derivative constraints in a class of smooth functions, i.e., a H\"older class. Let $L$ and $r$ be positive constants and $m:=\lceil r \rceil \in \mathbb N$ so that $r\in (m-1, m]$.
We introduce the H\"older class
\begin{equation}\label{eqn:holder_class}
H_L^r \, := \, \Big \{f:[0,1] \rightarrow \mathbb R \ \Big | \ \big|f^{(m-1)}(x_1)- f^{(m-1)}(x_2) \big| \leq L|x_1-x_2|^{r-(m-1)}, \ \forall \, x_1, x_2 \in [0,1]  \Big\},
\end{equation}
where $r$ is the H\"older exponent, and $L$ is the H\"older constant. Also, define $\mathcal S_{m, H}(r, L):=\mathcal S_m \cap H_L^r$.
Given a sequence of design points $(x_i)^n_{i=0}$ on $[0, 1]$, consider the following regression problem:
\begin{equation}\label{eqn:model}
   y_i \, = \, f(x_i) + \sigma \varepsilon_i, \quad i = 0,1,\dots,n,
\end{equation}
where $f$ is an unknown true function in $\mathcal S_{m, H}(r, L)$, the $\varepsilon_i$'s are independent  standard normal errors, $\sigma$ is a positive constant, and the $y_i$'s are samples.
The goal of shape constrained estimation is to construct an estimator $\wh f$ that preserves the specified shape of the true function characterized by  $\Scal_m$.
 In the asymptotic analysis of such an estimator, we are particularly interested in its uniform convergence on the entire interval $[0, 1]$ and the convergence rate of $\sup_{f\in \mathcal S_{m, H}(r, L)}\mathbb E(\| \wh f - f\|_\infty)$ as the sample size $n$ is sufficiently large, where $\| \cdot \|_\infty$ denotes the supremum norm of a function on $[0, 1]$. With the help of the uniform Lipschitz property, we show that for general nonnegative derivative constraints (i.e., $m \in \mathbb N$ is arbitrary), the constrained B-spline estimator (\ref{eqn:Bspline_estimator}) achieves uniform convergence on $[0, 1]$ for possibly non-equally spaced design points, and we provide a preliminary convergence rate. These results pave the way for further study of the optimal convergence of shape restricted estimators subject to general nonnegative derivative constraints.

We discuss the asymptotic performance of the constrained B-spline estimator (\ref{eqn:Bspline_estimator}) as follows.
Consider the set $\mathcal P_n$ in (\ref{eqn:set_of_Pn}) for a given $c_\omega \ge 1$. For each $n\in \mathbb N$, let $P := (x_i)^n_{i=1} \in \Pcal_n$ be a sequence of design points on $[0, 1]$. Let $(K_n)$ be an increasing sequence of natural numbers. For simplicity, we consider equally spaced B-spline knots for each $K_n$, i.e., $c_{\kappa_1,1} = c_{\kappa,2}=1$ for $\mathcal T_{K_n}$ defined in (\ref{eqn:set_of_T_K_n}) such that $T_\kappa = (i/K_n)^{K_n}_{i=0}$ for each $K_n$.
Let $y = (y_i)^n_{i=0} \in \mathbb R^{n+1}$ be the sample data vector given in (\ref{eqn:model}). For a given $m$ and a true function $f \in \mathcal S_{m, H}(r, L)$, the constrained B-spline estimator, denoted by $\wh f^B_{P, K_n}$, is given in (\ref{eqn:Bspline_estimator}), and its B-spline coefficient vector $\wh b_{P, K_n}$ is defined in (\ref{eqn:opt_spline_coeff_matrix}). (We replace the subscript $T_\kappa$ in (\ref{eqn:Bspline_estimator}) and (\ref{eqn:opt_spline_coeff_matrix}) by $K_n$ since we are considering equally spaced knots.)
Moreover, we introduce the vector of noise-free data $\vec f := (f(x_0),f(x_1),\dots,f(x_n))^T \in \mathbb R^{n+1}$ and define $\ol f_{P, K_n}:= \sum_{k=1}^{T_n}\ol b_k B_{m,k}^{T_\kappa}$, where $T_n:=K_n+m-1$ and $\ol b_{P, K_n} = (\ol b_k) \in \mathbb R^{T_n}$ is the B-spline coefficient vector characterized by the optimization problem in (\ref{eqn:opt_spline_coeff_matrix}) with $\ol y$ replaced by $K_n \wh X^T \Theta_n \vec f$:
\begin{equation} \label{eqn:def_ol_b}
\ol b_{P, K_n} \, := \, \argmin_{\wt D_{m,T_\kappa} \, b \geq 0} \frac{1}{2} \, b^T \, \Lambda_{K_n, P, T_{\kappa}} \, b - b^T \big( K_n \wh X^T \Theta_n \vec f \ \big).
\end{equation}

Note that for each true $f$, $\mathbb E(\|\wh f^B_{P, K_n} - f\|_\infty) \leq \| f - \ol f_{P, K_n}\|_\infty + \mathbb E(\|\ol f_{P, K_n} - \wh f^B_{P, K_n}\|_\infty)$, where $\| f - \ol f_{P, K_n}\|_\infty$ pertains to the estimator bias and $\mathbb E(\|\ol f_{P, K_n} - \wh f^B_{P, K_n}\|_\infty)$ corresponds to the stochastic error. We develop uniform  bounds for these two terms in the succeeding propositions via the uniform Lipschitz property. For notational simplicity, define  $\gamma:= r-(m-1)$ for $\mathcal S_{m, H}(r, L)$.

\begin{proposition}\label{prop:bias}
Fix $m \in \mathbb N$, and constants $c_\omega \ge 1$, $L>0$, and $r \in (m-1, m]$. Let $(K_n)$ be an increasing sequence  of natural numbers with $K_n \rightarrow \infty$ and $K_n/n \rightarrow 0$ as $n \rightarrow \infty$. Then there exist a positive constant $C_b$ and $\wh n_1 \in \mathbb N$ depending on $(K_n)$ only such that
\begin{equation*} %\label{eqn:bias}
     \sup_{f \in \Scal_{m, H}(r,L), \, P \in \mathcal P_n} \left \| \, f- \ol f_{P, K_n}  \, \right\|_\infty \, \le \, C_b \cdot \big(K_n)^{-\gamma}, \qquad \quad \forall \ n \ge \wh n_1.
\end{equation*}
%%%where $\gamma:= r-(m-1)$.
\end{proposition}

\begin{proof}
Fix an arbitrary true function $f \in \Scal_{m, H}(r,L)$. For any given $P\in \mathcal P_n$ and $K_n$, we write $\ol f_{P, K_n}$ as $\ol f$ to simplify notation.
Define the piecewise constant function $g:[0,1] \rightarrow \mathbb R$ as $g(x):= f^{(m-1)}((i-1)/K_n), \ \forall \, x \in [(i-1)/K_n, i/K_n)$ for each $i=1, \ldots, K_n$, and $g(1):= f^{(m-1)}((K_n-1)/K_n)$. Then define the function $\wt f :[0,1]\rightarrow \mathbb R$ as (a) when $m=1$,  $\wt f := g$; (b) when $m \ge 2$,
\[
\wt f(x) \, := \, \sum_{k=0}^{m-2} \frac{f^{(k)}(0)}{k!} x^k +\int_0^x \int_0^{t_1} \dots \int_0^{t_{m-2}} g(t_{m-1}) \,dt_{m-1} \cdots dt_2 \,dt_1, \quad x \in [0, 1].
\]
Clearly, $\wt f^{(m-1)} = g$ almost everywhere on $[0, 1]$. Since $g$ is given by some first-order B-splines for the knot sequence $T_\kappa=(i/K_n)^{K_n}_{i=0}$ and is increasing on $[0, 1]$, it can be shown via induction on $m$ that $\wt f$ is equal to $\sum_{k=1}^{T_n} \wt b_k B_{m,k}^{T_\kappa}$ for some $\wt b := ( \wt b_k ) \in \mathbb R^{T_n} $. By  Lemma~\ref{lem:shape_ineq} and $f \in \mathcal S_m$, we have $
\wt f \in \Scal_m$  and $\wt D_{m, T_\kappa} \wt b \geq 0$.

In view of $\| f - \ol f\|_\infty \le \| f - \wt f\|_\infty + \| \wt f - \ol f\|_\infty$, we establish a uniform bound for each term on the right-hand side below.

(i) Uniform bound of $\| f - \wt f\|_\infty$. In light of the definitions of the piecewise constant function $g$ and the H\"older class $H^r_L$, we see that $| g(x) - f^{(m-1)}(x)| \le L \cdot (K_n)^{-\gamma}, \forall \ x \in [(i-1)/K_n, i/K_n)$ for each $i=1, \ldots, K_n$. Hence, when $m=1$, $\| f - \wt f\|_\infty \le L \cdot (K_n)^{-\gamma}$. When $m \ge 2$, we have that %for each $x \in [0,1]$,
\begin{align*}
  \big | f(x) - \wt f(x) \big| & \, \leq \, \int_0^x \int_0^{t_1} \cdots \int_0^{t_{m-2}} \left | f^{(m-1)}(t_{m-1}) - g(t_{m-1}) \right | \, dt_{m-1} \cdots dt_2 \, dt_1 \\
		     & \, \leq \, \int_0^x \int_0^{t_1} \cdots \int_0^{t_{m-2}} L \cdot K_n^{-\gamma} \, dt_{m-1} \cdots dt_2 \, dt_1  \, \le \,  \frac{L}{(m-1)!}K_n^{-\gamma}, \quad \forall \ x \in [0,1].
\end{align*}
Hence, $\|f - \wt f\|_\infty \leq \frac{L}{(m-1)!} \cdot K_n^{-\gamma}$ for any $f\in \mathcal S_{m, H}(r, L)$.

(ii) Uniform bound of $\| \wt f - \ol f\|_\infty$. We introduce the vector $\vec{\tilde f} := (\wt f(x_0),\wt f(x_1),\ldots, \wt f(x_n))^T \in \mathbb R^{n+1}$. Since $\wt f = \sum_{k=1}^{T_n} \wt b_k B_{m,k}^{T_\kappa}$, we have $\wh X \wt b  = \vec{\tilde f}$, where $\wh X$ is the design matrix corresponding to the design point sequence $P$ and the B-splines $\{ B^{T_\kappa}_{m, k} \}^{T_n}_{k=1}$ (cf. Section~\ref{subsec:shape_const_Bsplines}). This shows that $\frac{1}{2}\|K_n^{1/2}\Theta_n^{1/2}(\wh X \wt b-\vec{\tilde f})\|^2_2 =0$. Moreover, since $\wt D_{m, T_\kappa} \wt b \ge 0$, we deduce that
\[
\wt b \, = \, \argmin_{\wt D_{m, T_\kappa} \, b \geq 0} \frac{1}{2} b^T \, \Lambda_{K_n, P, T_\kappa} \, b - b^T \, \Big( K_n\wh X^T \Theta_n \vec{ \tilde f} \ \Big).
\]
By using the definition of $\ol f$ before (\ref{eqn:def_ol_b}) and the uniform Lipschitz property in Theorem~\ref{thm:uniform_Lip}, we obtain a positive constant $c_\infty$ depending only on $m$ and a natural number $n_*$ depending only on $(K_n)$, $m$, and $c_\omega$ such that for any $P \in \mathcal P_n$ and $K_n$ with $n \ge n_*$,
\begin{align}
 \big \| \ol f-\wt f \big \|_\infty & \leq  \, \big \|\ol b_{P, K_n} - \wt b \big \|_\infty \, \leq \, c_\infty \cdot \left \|K_n \wh X^T \Theta_n \vec f -   K_n\wh X^T \Theta_n \vec{ \tilde f} \right \|_\infty   \, \leq \, c_\infty \cdot \big \|K_n\wh X^T \Theta_n  \big \|_\infty  \cdot \Big \|\vec f - \vec{ \tilde f} \, \Big \|_\infty \nonumber\\
                          &  \,  \leq \, c_\infty \cdot \big \|K_n \wh X^T \Theta_n \big \|_\infty  \cdot \frac{ L}{(m-1)!}K_n^{-\gamma},  \label{eqn:bias_bound_1}
\end{align}
where we use $\big \|\vec f - \vec{ \tilde f} \, \big \|_\infty \le \| f - \wt f\|_\infty$  and the uniform bound for $\| f - \wt f\|_\infty$ established in (i).
We further show that $\big \|K_n \wh X^T \Theta_n \big \|_\infty$ attains a uniform bound independent of $P\in \mathcal P_n$ and $K_n$ as long as $n$ is large enough. Let $\kappa_i := i/K_n$ for $i=1, \ldots, K_n$ (with the usual extension). It follows from the definition of $\wh X$ and the non-negativity, upper bound, and support of the $B_{m,k}^{T_\kappa}$'s (given at the beginning of Section~\ref{sec:prob_form}) that for each $k = 1,\dots, T_n$,
\begin{eqnarray*}
   \left \|(K_n\wh X^T \Theta_n)_{k\bullet} \right \|_\infty & = & K_n \sum_{i=0}^n B_{m,k}^{T_\kappa}(x_i) \cdot (x_{i+1} - x_i) \, \le \, K_n  \sum_{i=0}^n \mathbf I_{[\kappa_{k-m}, \, \kappa_k]}(x_i) \cdot (x_{i+1} - x_i)  \\
   & \le & K_n \left( \kappa_k-\kappa_{k-m} +\frac{ c_\omega}{n}\right) \, \le \, m +\frac{ c_\omega K_n}{n},
\end{eqnarray*}
where the term $ c_\omega/n$ comes from the fact that  $\kappa_k$ is  in the interval $[x_r, x_{r+1})$ for  some design points $x_r, x_{r+1}$.
Since $K_n/n \rightarrow 0$ as $n \rightarrow \infty$, we obtain $\bar n_1 \in \mathbb N$ (depending only on $(K_n)$) such that for any $n \ge \bar n_1$, $\big \|(K_n\wh X^T \Theta_n)_{k\bullet} \big\|_\infty \le m+1$ for each $k = 1,\dots, T_n$ so that $\big \|K_n \wh X^T \Theta_n \big \|_\infty \le m+1$ for any $P \in \mathcal P_n$ and $K_n$. Combining this with~(\ref{eqn:bias_bound_1}) yields that for any $n \ge \max(n_*, \bar n_1)$, $\|\ol f-\wt f\|_\infty \leq \frac{ c_\infty L (m+1)}{(m-1)!}K_n^{-\gamma}$ for any $f \in \mathcal S_{m, H}(r, L)$, $P \in \mathcal P_n$ and $K_n$.

Utilizing (i) and (ii) and letting $\wh n_1 :=\max(n_*, \bar n_1)$ (depending only on $(K_n)$),
 we conclude that for any $n \ge \wh n_1$, $\sup_{f \in \Scal_{m, H}(r,L), \, P \in \mathcal P_n}  \| \, f- \ol f  \, \|_\infty \, \le \, C_b \cdot K_n^{-\gamma}$, where
 $C_b := \frac{ [c_\infty  (m+1)+1] L}{(m-1)!}$.
\end{proof}

The next proposition gives a uniform bound of the stochastic error of the estimator $\wh f^B_{P, K_n}$.

\begin{proposition}\label{prop:stoch_error}
Fix $m \in \mathbb N$, and constants $c_\omega \ge 1$, $L>0$, $r \in (m-1, m]$, and $q \ge 1$. Let $(K_n)$ be an increasing sequence  of natural numbers with $K_n \rightarrow \infty$, $K_n/n \rightarrow 0$, $K_n/(n^{1/q}\cdot\sqrt{\log n}) \rightarrow 0$ as $n \rightarrow \infty$. Then there exist a positive constant $C_s$ and $\wh n_2 \in \mathbb N$ depending on $(K_n)$ only such that
\begin{equation*} %\label{eqn:stoch_error}
\sup_{f \in \Scal_{m, H}(r,L), \, P \in \mathcal P_n} \mathbb E\left( \, \left \|\wh f^B_{P, K_n}-\ol f_{P, K_n} \right \|_\infty  \, \right) \, \le \, C_s \cdot \sqrt{ \frac{K_n \log n}{n} }, \qquad \forall \ n \ge \wh n_2.
\end{equation*}
\end{proposition}

%%%%%%%%%%%%%%%%%%%%%%%

\begin{proof}
Fix a true function $f \in \mathcal S_{m, H}(r, L)$. Given arbitrary $P \in \mathcal P_n$ and $K_n$,  let $\vec \varepsilon := (\varepsilon_i) \in \mathbb R^{n+1}$, $\omega_i:=x_{i+1} - x_i$ for $i=0, 1, \ldots, n$,  and define $\xi_k := \big(\sqrt{n K_n} \, \wh X^T \Theta_n \big)_{k\bullet} \, \vec \varepsilon = \sqrt{n K_n} \sum_{i=0}^n \omega_i B_{m,k}^{T_\kappa}(x_i)\varepsilon_i$ for each $k = 1,\dots, T_n$.
Since each $\varepsilon_i \sim \mathcal N(0,1)$ and $\varepsilon_i$'s are independent, we have $\xi_k \sim \mathcal N(0,\bar \sigma_k^2)$, where $\bar \sigma_k^2 = \|(\sqrt{n K_n} \, \wh X^T \Theta_n)_{k\bullet}\|_2^2 \geq 0$.
By the proof of item (ii) of Proposition~\ref{prop:bias}, there exists $\bar n_1 \in \mathbb N$ (depending only on $(K_n)$) such that if $n \ge \bar n_1$, then for any $P \in \mathcal P_n$ and $K_n$, $\|(K_n \wh X^T \Theta_n)_{k\bullet}\|_\infty \leq m+1$ for each $k = 1,\ldots, T_n$. Therefore, by using $n \cdot \big(B_{m,k}^{T_\kappa}(x_i) \, \omega_i \big)^2 \le  n \omega_i \cdot \big (B_{m,k}^{T_\kappa}(x_i) \, \omega_i \big)  \le c_\omega \cdot \big (B_{m,k}^{T_\kappa}(x_i) \, \omega_i \big)$ for each $i$, we have, for any $P \in \mathcal P_n$ and $K_n$ with $n \ge \bar n_1$,
\begin{align*}
%
%\bar \sigma_k^2 &=
%
 \left \|\big(\sqrt{n K_n} \, \wh X^T \Theta_n \big)_{k\bullet} \right\|_2^2 & =   \sum_{i=0}^n K_n \cdot n \left(B_{m,k}^{T_\kappa}(x_i) \, \omega_i \right)^2 \leq c_\omega \sum_{i=0}^n K_n B_{m,k}^{T_\kappa}(x_i) \, \omega_i = c_\omega  \big\|(K_n \wh X^T \Theta_n)_{k\bullet} \big\|_\infty \\%%%\le c_\omega (m+1).
                & \leq c_\omega (m+1), \qquad \qquad \qquad  \qquad \forall \ k=1, \ldots, T_n.
\end{align*}
This shows that  $\bar \sigma_k^2 \leq  c_\omega(m+1)$ for each $k = 1,\dots, T_n$. Letting $\xi :=(\xi_1, \ldots, \xi_{T_n})^T \in \mathbb R^{T_n}$, we see that $\sigma  \xi = \sqrt{n K_n} \wh X^T \Theta_n \sigma \vec\varepsilon = \sqrt{n K_n} \wh X^T \Theta_n (y - \vec f)$. Hence,
we deduce, via the uniform Lipschitz property of Theorem~\ref{thm:uniform_Lip}, that there exist a positive constant $c_\infty$ and $n_* \in \mathbb N$ (depending on $(K_n)$ only) such that for any $f \in \Scal_{m, H}(r, L)$, any $P \in \mathcal P_n$ and $K_n$ with $n \ge n_*$,
\begin{align}
 \left \|\wh f^B_{P, K_n} - \ol f_{P, K_n} \right \|_\infty & \, \leq \, \big \|\wh b_{P, K_n} - \bar b_{P, K_n} \big\|_\infty \leq c_\infty \cdot \sqrt{K_n / n} \cdot \left\|\sqrt{n K_n} \wh X^T \Theta_n \big( y -  \vec f \, \big) \right\|_\infty \notag \\
  & \, = \, c_\infty \cdot \sqrt{K_n / n} \cdot \sigma \cdot \| \xi \|_\infty
     \, = \, \wt c \cdot \sqrt{K_n/n} \cdot \bar \xi, \label{eqn:whf_olf_bound}
\end{align}
where $\bar \xi := \| \xi \|_\infty= \max_{k=1, \ldots, T_n} |\xi_k|$ and $\wt c :=   c_\infty \sigma>0$.

Define $\grave\sigma := \sqrt{c_\omega(m+1)}$, and consider the random variable $Z_\xi \sim \mathcal N(0,\grave \sigma^2)$. Since the variance of $\xi_k$ satisfies $\bar \sigma^2_k \le c_\omega(m+1)=\grave\sigma^2$ for each $k$, we have, for  any $u \geq 0$,
\[
\mathbb P\left(|Z_\xi| \geq \frac{u}{\wt c}\sqrt\frac{n}{K_n} \, \right) \, \geq \, \mathbb P\left(|\xi_k| \geq \frac{u}{\wt c}\sqrt\frac{n}{K_n} \, \right), \qquad \forall \ k = 1,\ldots,T_n.
\]
Recall $\wh n_1:=\max(\bar n_1, n_*)$ as in the proof of Proposition~\ref{prop:bias}.
By using (\ref{eqn:whf_olf_bound}), the above inequality, and the implication: $X \sim \mathcal N(0,\sigma^2) \implies \mathbb P(|X| \geq v) \leq \exp\big(-\frac{v^2}{2\sigma^2}\big)$ for any $v \ge 0$, we have that for any $P \in \mathcal P_n$ and $K_n$ with $n \ge \wh n_1$ and any $u \ge 0$,
\begin{align*}
 \mathbb P\left(\|\wh f^B_{P, K_n} - \ol f_{P, K_n}\|_\infty \ge u\right) & \, \leq \,  \mathbb P\left( \bar{\xi} \geq \frac{u}{\wt c}\sqrt\frac{n}{K_n}\right) \, \leq \, T_n \cdot \mathbb P\left( |Z_\xi| \geq  \frac{u}{\wt c}\sqrt\frac{n}{K_n}\right) \\
  & \, \leq \, T_n\cdot \exp\left(-\frac{u^2n}{2 \, \grave \sigma^2 \, \wt c^2 \,K_n}\right).
\end{align*}
Let $W_n := \wt c \, \grave \sigma \sqrt{\frac{2 K_n\log n}{q \cdot n}}$ for the fixed $q\ge 1$.  It follows from the above result and
 $\int^\infty_v e^{-t^2/(2\sigma^2)} dt \le e^{-v^2/(2\sigma^2)} \sigma \sqrt{\pi/2}$ for any $v \ge 0$ that for any $f \in \Scal_{m, H}(r,L)$ and any $P \in \mathcal P_n$ and $K_n$ with $n \ge \wh n_1$,
\begin{align*}
\lefteqn{ \mathbb E \left( \big\|\wh f^B_{P, K_n} - \ol f_{P, K_n} \big\|_\infty \right)   \, = \,
   \int^\infty_0  \mathbb P \left( \big\|\wh f^B_{P, K_n} - \ol f_{P, K_n} \big\|_\infty \ge t \right) dt } \\
& \, \leq \, W_n+ \int_{W_n}^\infty \mathbb P\left(  \big\|\wh f^B_{P, K_n} - \ol f_{P, K_n} \big\|_\infty \ge t \right)\,dt
    \,
						       \leq \, W_n + T_n \cdot \int_{W_n}^\infty  \exp\left(-\frac{n \,t^2}{2 \, \wt c^2 \,\grave \sigma^2K_n}\right)\,dt\\
& \, \le \,   W_n + T_n \wt c \, \grave \sigma \sqrt\frac{\pi K_n}{ 2 n} \exp\left(-\frac{ W_n^2 n}{2\wt c^2 \grave \sigma^2 K_n}\right) \, = \, W_n + \wt c \, \grave \sigma \cdot \sqrt\frac{\pi K_n}{ 2 n} \cdot (K_n+m-1) \cdot n^{-\frac{1}{q} }.
\end{align*}
Since $K_n/(n^{1/q}\cdot\sqrt{\log n}) \rightarrow 0$ as $n \rightarrow \infty$, there exist a constant $C_s>0$ and $\wh n_2 \in \mathbb N$ with $\wh n_2 \ge \wh n_1$ (depending on $K_n$ only) such that for any $f\in \mathcal S_{m, H}(r, L)$, any $P \in \mathcal P_n$ and $K_n$ with $n \ge \wh n_2$,
\[
   \mathbb E \left( \big\|\wh f^B_{P, K_n} - \ol f_{P, K_n} \big\|_\infty \right) \le C_s \cdot \sqrt{\frac{K_n\log n}{n}}.
\]
This leads to the desired uniform bound for the stochastic error.
\end{proof}

Combining the results of Propositions~\ref{prop:bias} and~\ref{prop:stoch_error}, we obtain the following theorem.

\begin{theorem} \label{thm:total_bound}
 Fix $m \in \mathbb N$, and constants $c_\omega \ge 1$, $L>0$, $r \in (m-1, m]$, and $q > 1$. Let $(K_n)$ be a sequence  of natural numbers with $K_n \rightarrow \infty$, $K_n/n \rightarrow 0$, $K_n/(n^{1/q}\cdot\sqrt{\log n}) \rightarrow 0$ as $n \rightarrow \infty$. Then there exist two positive constants $C_b, C_s$ and $\wh n_\star \in \mathbb N$ depending on $(K_n)$ only such that
 \[
   \sup_{f \in \Scal_{m, H}(r,L), \, P \in \mathcal P_n} \mathbb E\left( \, \left \|f - \wh f^B_{P, K_n} \right \|_\infty  \, \right) \, \le \, C_b \cdot \big(K_n)^{-\gamma} + C_s \cdot \sqrt{ \frac{K_n \log n}{n} }, \qquad \forall \ n \ge \wh n_\star.
 \]
\end{theorem}

A specific choice of $(K_n)$ that satisfies the conditions in Theorem~\ref{thm:total_bound} is $K_n = \left\lceil \big( \frac{ n}{\log n} \big)^{1/q} \right\rceil$. This result demonstrates the uniform convergence of the constrained B-spline estimator $\wh f^B_{P, K_n}$ to the true function $f$ on the entire interval $[0, 1]$, and the consistency of this B-spline estimator, including the consistency at the two boundary points, even if design points are not equally spaced. Note that the monotone and convex least-squares estimators are inconsistent at the boundary points due to non-negligible asymptotic bias \cite{Groeneboom_AoS01, Pal_SPL08, woodroofe_SSinica93}, which is known as the spiking problem.

\begin{remark} \rm
 In addition to the uniform convergence,  Theorem~\ref{thm:total_bound} also gives an asymptotic convergence rate of the constrained B-spline estimator subject to general nonnegative derivative constraints. The obtained rate is not optimal yet due to the conservative rate for the bias (cf. Proposition~\ref{prop:bias}). The optimal bias rates have been established for $m=1$ (i.e., the monotone constraint) and $m=2$ (i.e., the convex constraint) in the literature. However, the bias analysis under higher order nonnegative derivative constraints is complex and requires substantial technical developments, which are beyond the scope of this paper and shall be reported in the future.
\end{remark}

%-----------------------------------------------------------------------
%
\section{Concluding Remarks} \label{sect:conclusion}

This paper establishes a critical uniform Lipschitz property for constrained B-splines
 subject to general nonnegative derivative constraints with possibly non-equally spaced design points and/or knots; important asymptotic analysis results are obtained via this property.
Future research topics include optimal rates of convergence of the B-spline
estimator under higher-order derivative constraints, and
 constrained estimation subject to more general shape constraints, e.g.,
those given by multiple derivative constraints or a linear combination of derivatives of different orders.
%

%----------------------------------------------------------------------------
%
%
%\bibliography{UniLipNotes.bib}{}
%\bibliographystyle{plain}
%

\end{document}